\documentclass[11pt]{amsart}
\usepackage{pifont}
\usepackage{mathrsfs}
\usepackage{authblk}
\usepackage{url}
\usepackage{multirow}
 \usepackage[foot]{amsaddr}

\usepackage{anysize}
\marginsize{2.6cm}{2.54cm}{1.8cm}{2.0cm}
\usepackage[small,it]{caption}
\usepackage[OT2,T1]{fontenc}
\DeclareSymbolFont{cyrletters}{OT2}{wncyr}{m}{n}
\DeclareMathSymbol{\Sha}{\mathalpha}{cyrletters}{"58}
 \usepackage[latin1]{inputenc}
 \usepackage[dvips]{graphicx}
 \usepackage{wrapfig}
 \usepackage{amsmath}
 \usepackage{amsthm}
 \usepackage{amsfonts}
 \usepackage{amssymb}
 \usepackage{amscd}
 \usepackage{layout}
\usepackage{verbatim}
 \usepackage{alltt}
\usepackage{yfonts}
\usepackage{color}

\usepackage{hyperref}

\usepackage[all]{xy}
\usepackage{setspace}

\newtheorem*{thma}{Theorem A}
\newtheorem*{thmb}{Theorem B}
\newtheorem*{thmc}{Theorem C}

\newtheorem*{claim}{Claim}

\newcommand{\LL}{\Lambda}
\newcommand{\TT}{\mathbb{T}}
\newcommand{\bfT}{\mathbf{T}}
\newcommand{\RRn}{\mathcal{R}}
\newcommand{\RRc}{\mathcal{R}_{\mathrm{cyc}}}
\newcommand{\TTc}{\mathbb{T}_{\mathrm{cyc}}}
\newcommand{\TTo}{\mathbb{T}}
\newcommand{\QQ}{\mathbb{Q}}
\newcommand{\FF}{\mathcal{F}}

\newcommand{\lra}{\longrightarrow}
\newcommand{\ZZ}{\mathbb{Z}}
\newcommand{\PP}{\mathcal{P}}

\newcommand{\KKK}{\mathcal{K}}

\newcommand{\Gal}{\textup{Gal}}
\newcommand{\KS}{\textbf{\textup{KS}}}
\newcommand{\ES}{\textbf{\textup{ES}}}
\newcommand{\NN}{\mathcal{N}}
\newcommand{\ra}{\rightarrow}

\newcommand{\be}{\begin{equation}}
\newcommand{\ee}{\end{equation}}

\newcommand{\oo}{\mathcal{O}}

\newcommand{\mm}{\hbox{\frakfamily m}}

\newcommand{\FFc}{\mathcal{F}_{\textup{\lowercase{can}}}}
\newcommand{\Hom}{\textup{Hom}}

\newcommand{\Gr}{\textup{Gr}}

\newcommand{\hzero}{{(H.0)}}
\newcommand{\htwo}{{(H.2)}}

\newcommand{\res}{\mathrm{res}}
\newcommand{\cyc}{\textup{cyc}}
\newcommand{\pr}{\mathrm{pr}}

\newcommand{\fp}{\frak{p}}

\newcommand{\GL}{\textup{GL}}

\newcommand{\hzerominus}{(H.$0^-$)}

\newcommand{\htwoplus}{(H.$2^+$)}

\newcommand{\Bf}{\mathbf{f}}
\newcommand{\BF}{\mathbf{BF}}

\newcommand{\ur}{\textup{ur}}
\newcommand{\Bj}{\mathbf{j}}

\newcommand{\EXP}{\mathbf{EXP}}

\numberwithin{equation}{section}
\newtheorem{thm}{Theorem}[section]
\newtheorem{lemma}[thm]{Lemma}
\newenvironment{define}{\par\medskip\noindent\refstepcounter{thm}
\bgroup{\hspace*{-0.15 cm}\bf{Definition}
\thethm.}\bgroup}{\egroup \egroup\par\medskip}\newtheorem{prop}[thm]{Proposition}
\newtheorem{cor}[thm]{Corollary}
\newenvironment{rem}{\par\medskip\noindent\refstepcounter{thm}
\bgroup{\hspace*{-0.15 cm}\bf{Remark} \thethm.}\bgroup}{\egroup
\egroup\par\medskip} \parskip 2pt

\newcounter{Athm}[section]

\renewcommand{\theAthm} {\arabic{Athm}}

\long\def\symbolfootnote[#1]#2{\begingroup%
\def\thefootnote{\fnsymbol{footnote}}\footnote[#1]{#2}\endgroup}

\begin{document}
\title{M\lowercase{ain conjectures for higher rank nearly ordinary families} -- I}

\author{K\^az\i m B\"uy\"ukboduk}
\author{Tadashi Ochiai}

\address[K.B.]{School of Mathematics and Statistics, University College Dublin, Belfield, Ireland} 
\email{kazim.buyukboduk@ucd.ie, ochiai@math.sci.osaka-u.ac.jp}
\address[T.O.]{Osaka University, 1-1 Machikaneyama-cho, Toyonaka, Osaka
560-0043, Japan}



\begin{abstract}
In this article, we present the first half of our project on the Iwasawa theory of 
higher rank Galois deformations over deformations rings of arbitrary dimension. We develop a theory of Coleman maps for a very general class of coefficient rings, devise a dimension reduction procedure for locally restricted Euler systems and finally, put these into use in order to prove a divisibility in a $3$-variable main conjecture for nearly ordinary families of Rankin-Selberg convolutions. 
\end{abstract}

\maketitle
\section{Introduction and set up}
\label{sec:Intro}

Our principal goal in this article is to establish a general machinery to approach the Iwasawa main conjectures for Galois representations of higher rank with coefficients in deformation rings of higher Krull dimension. This consists of two independent steps. 

First, under certain technical hypotheses, we devise an inductive procedure to bound 
the size of a Selmer group for a Galois representation of higher rank over regular rings of higher Krull dimension with the aid of a locally restricted Euler system; see Theorem~B below as a sample of our results, Theorems~\ref{thm:weakleopoldt} and \ref{thm:maingreenbergbound} for their most general form (see also Paragraph~\ref{subsec:history} for a comparison with prior related results). 

Secondly, we construct Coleman maps for a very general class of nearly ordinary Galois deformations, with coefficients in finite extensions of the ring of power series in any number of variables; see Theorem~A below as a sample of our results in this direction. Plugging the $3$-variable Beilinson--Flach classes of Kings--Loeffler--Zerbes~\cite{KLZ2} in our Coleman map and relying on results of \cite{KLZ1}, we obtain another construction of Hida's $p$-adic $L$-function (along the same lines as \cite[Theorem B]{KLZ2}).
 
We finally exhibit an application of the theory we have developed here, with the aid of the Beilinson--Flach Euler system of \cite{KLZ2}. Our result (Theorem C below) is a divisibility in the statement of a $3$-variable main conjecture for the cyclotomic deformation of  the Rankin--Selberg convolution $\Bf_1\otimes\Bf_2$ of two Hida families $\Bf_1$ and $\Bf_2$, which we prove using our Theorem~B with the $3$-variable Beilinson--Flach Euler system. Note that a single-variable version of this result (concerning cyclotomic deformation of the Rankin--Selberg convolution $f_1\otimes f_2$ of two eigenforms $f_1$ and $f_2$) was established in \cite{KLZ2}.

\subsection{A sample of results in this article}
For the sake of the clarity of our exposition in the introduction, we shall phrase our main results only in the the case of Rankin-Selberg products. We first introduce our notation (which we will also rely on in the main body of this note); we refer the reader to Section~\ref{subs:exampleRankinSelberg} and Section~\ref{sec:BFES} for more precise definitions of the objects that show up in the paragraph below. Fix forever a prime $p>7$.

Let $\Bf_1=\sum_{n=1}^{\infty} a_n(\Bf_1)q^n$ and $\Bf_2=\sum_{n=1}^{\infty} a_n(\Bf_2)q^n$ denote two primitive Hida families of elliptic modular cusp forms with respective tame levels $N_1$ and $N_2$ such that $p \nmid N_i$ ($i=1,2$) and central characters $\Psi_i:(\ZZ/pN_i\ZZ)^\times\lra \overline{\QQ}_p^\times$ ($i=1,2$), which are defined over the respective local domains  $\mathbb{I}_{\Bf_i}$ ($i=1,2$), which are both finite flat over the respective one-variable Iwasawa algebra $\mathbb{Z}_p [[\Gamma_i]]$, where $\Gamma_i$ is the group equipped with an isomorphism $\chi_i:\, \Gamma_i\stackrel{\sim}{\lra}1+p\ZZ_p$ for $i=1,2$. We write $\Psi_i=\Psi_i^{(N_i)}\Psi_i^{(p)}$ so that the Dirichlet character $\Psi_i^{(N_i)}$ has prime-to-$p$ conductor and $\Psi_i^{(p)}$ has conductor dividing $p$.

We let $\mathbb{T}_{\Bf_i}$ denote Hida's big Galois representation that Hida associates to $\Bf_i$. Under mild hypotheses (which we shall assume throughout) this is a free $\mathbb{I}_{\Bf_i}$-module of rank $2$ and its restriction $\mathbb{T}_{\Bf_i}\vert_{G_{\QQ_p}}$ admits a $p$-ordinary filtration $F_p^+\TT_{\Bf_i} \subset \TT_{\Bf_i}$, where $F_p^+\TT_{\Bf_i}$ is a free direct summand. 

We set $\Lambda_{\cyc}:=\ZZ_p[[\Gamma_\cyc]]$, where $\Gamma_\cyc$ is the Galois group of the cyclotomic $\ZZ_p$-extension of $\QQ$. We call $\Lambda_{\cyc}$ the cyclotomic Iwasawa algebra. 
We denote by $\Lambda^\sharp_{\cyc}$ (resp. $(\Lambda^\sharp_{\mathrm{cyc}})^\iota$) the free $\Lambda_{\cyc}$-module of rank one on which the absolute Galois group $G_{\QQ}$ acts via the tautological character $\widetilde{\chi}_\cyc :\ G_{\QQ_p} \twoheadrightarrow \Gamma_{\cyc} \hookrightarrow  \Lambda^\times_{\cyc}$ (resp. via the inverse tautological character $\widetilde{\chi}^{-1}_\cyc:\ G_{\QQ_p} \twoheadrightarrow \Gamma_{\cyc} \stackrel{\textup{inv}}{\lra}\Gamma_\cyc\hookrightarrow \Lambda^\times_{\cyc}$, where $\textup{inv}$ is the map given by $\gamma\mapsto \gamma^{-1}$). We set $\mathbb{T}_\cyc = \mathbb{T}_{\Bf_1} \widehat{\otimes}_{\mathbb{Z}_p} \mathbb{T}_{\Bf_2}\widehat{\otimes}_{\mathbb{Z}_p}(\LL_\cyc^\sharp)^\iota$, which is a free module of rank $4$ over 
$\mathcal{R}_\cyc:= \mathbb{I}_{\Bf_1} \widehat{\otimes}_{\mathbb{Z}_p} \mathbb{I}_{\Bf_2}\widehat{\otimes}_{\mathbb{Z}_p}\LL_\cyc$ and we consider the subquotient $F^{-+}\TT_\cyc=\mathbb{T}_{\Bf_1}/F_p^+\TT_{\Bf_1} \,\otimes F_p^+\mathbb{T}_{\Bf_2}\widehat{\otimes}_{\mathbb{Z}_p}(\LL_\cyc^\sharp)^\iota$. Under suitable hypotheses, it is a free $\mathcal{R}_\cyc$-module of rank one. 

A continuous homomorphism $\kappa=\kappa_1\otimes\kappa_2\otimes \kappa_{\mathrm{cyc}} \in \textup{Hom}(\RRc,\overline{\QQ}_p)$ is called arithmetic  if it satisfies the following two conditions: 
\begin{itemize}
\item[(i)] The homomorphism {$\kappa_i:\mathbb{I}_{\Bf_i}\lra \overline{\QQ}_p$} is an arithmetic specialization of weight $w_i \geq 0$ on the ordinary Hida family $\mathbb{I}_{\Bf_i}$, {in the sense that $\kappa_i$ restricted to a certain open subgroup $U$ of $\Gamma_i \subset 
\mathbb{Z}_p [[\Gamma_i ]] \subset \mathbb{I}_{\Bf_i}$ is equal to $\chi^{w_i}_i$ for some non-negative integer $w_i$.} By Hida theory, 
the specialization of $\mathbb{T}_{\Bf_i}$ at each arithmetic specialization $\kappa_i$ corresponds to a cuspform of weight $k_i=w_i+2$  ($i=1,2$).
\item[(ii)] The homomorphism $\kappa_{\mathrm{cyc}}$ is of the form $\chi_\cyc^{j}\phi$ where $\chi_\cyc:\Gamma_\cyc\ra 1+p\ZZ_p$ is the cyclotomic character, $j$ is an integer and $\phi$ is a character of $\Gamma_\cyc$ of finite order. 
\end{itemize}
Given an arithmetic specialization $\kappa$ as above, we write $E_\kappa \subset 
\overline{\QQ}_p$ for the finite extension of $\QQ_p$ generated by the image of $\kappa$ and define $V_\kappa:=\TT_\cyc\otimes_{\kappa}E_\kappa$ (the specialization of $\TT_\cyc$ at $\kappa$). The Galois representation $V_\kappa$ inherits a filtration from $\TT_\cyc$ in the obvious manner. For $i=1,2$, we denote by $a_p(\Bf_i(\kappa_i))$ the $U_p$-eigenvalue  on the $p$-stabilized eigenform $\Bf_i(\kappa_i)$.

We finally let $\mathbb{D}(F^{-+}\TT_\cyc)$ denote the big Dieudonn\'e module associated to the family $F^{-+}\TT_\cyc$ of local Galois representations which interpolates the de Rham Dieudonn\'e modules of its specializations. We shall not provide its explicit form in this introduction but refer the reader to Definition~\ref{define:thefatdieudonnemodule}.  

We are now ready to state a particular case of one of our main results, see Corollary~\ref{cor:dualexponentialmapforH++} for its most general form.
\begin{thma}\label{cor:dualexponentialmapIntro} 
There exists an $\mathcal{R}_\cyc$-linear isomorphism 
$${\EXP}^\ast \,:\, 
H^1(\QQ_p,F^{-+}\TT_\cyc)
\lra 
\mathbb{D}(F^{-+}\TT_\cyc)  
$$ 
which is characterized by the following interpolation property: For every arithmetic specialization $\kappa=\kappa_1\otimes\kappa_2\otimes \kappa_\cyc$ as above with $j>w_2$ the following diagram commutes:
$$
\xymatrix{
H^1(\QQ_p,F^{-+}\TT_\cyc)
\ar[d]_\kappa \ar[rrr]^{{\EXP}^\ast} &&& 
\mathbb{D}(F^{-+}\TT_\cyc )
\ar[d]^\kappa \\ 
H^1(\QQ_p, F^{-+}V_\kappa)
\ar[rrr]_{\hspace*{10pt}e_p^+\,\times\, 
\exp^\ast} 
&&& 
D_{\textup{dR}}(F^{-+}V_\kappa)
}$$
Here $e_p^+:=(-1)^{j-w_2}(j-w_2)!\,e_p$ and $e_p=e_p(\kappa)$ is the $p$-adic multiplier given by
$$e_p=\left( 1 - \dfrac{p^{j-w_2-1}}{a_p(\Bf_1(\kappa_1 ))a_p(\Bf_2(\kappa_2 ))^{-1}\Psi_2^{(N_2)}(p)} \right) \left( 
1 - \dfrac{a_p(\Bf_1(\kappa_1 ))a_p(\Bf_2(\kappa_2 ))^{-1}\Psi_2^{(N_2)}(p)}{p^{j-w_2}}\right)^{-1}$$
in case $F^{-+}V_\kappa$ is crystalline, and 

$$ e_{p}=\left( \dfrac{p^{j-w_2-1}}{a_p(\Bf_1(\kappa_1 ))a_p(\Bf_2(\kappa_2 ))^{-1}\Psi_2^{(N_2)}(p)}\right)^n$$
when  $F^{-+} V_\kappa\,\vert_{I_p} \cong E_\kappa(\omega_2+1-j)(\eta)$ with $\textup{ord}_p(\textup{cond}(\eta))=n\geq1$.

Also, for every $\kappa$ as above with $j \leq w_2$ we also have the following commutative diagram:
$$
\xymatrix{
H^1(\QQ_p,F^{-+}\TTc )
\ar[d]_\kappa \ar[rrr]^(.48){\mathrm{\EXP}^\ast}&&& 
\mathbb{D}(F^{-+}\TTc )
\ar[d]^\kappa\\
H^1 (\QQ_p,F^{-+}V_\kappa )
\ar[rrr]_(.45){\hspace*{15pt}e_p^-\,\times\, 
\log
} 
&&& 
D_{\textup{dR}}(F^{-+}V_\kappa )
}$$
where $e_p^-:=\dfrac{e_p}{(w_2-j)!}$\,.
\end{thma}
{A similar statement is proved in \cite[Theorem 8.2.8]{KLZ2}, by reducing to the case where the $p$-local Galois representation in consideration is unramified, which the authors of op. cit. handle in \cite[Theorem 8.2.3]{KLZ2}. Our construction here goes through directly without passing to an unramified twist of $F^{-+}\TT_\cyc$. Hence it is simpler and the method works in a more general situation.} We also refer the reader to Remark~\ref{rem:comparetorubinwithGausssums} who might be curious about absence of the Gauss sums as compared to \cite[Theorem 8.2.3]{KLZ2}, in the portions of the interpolation formulae that concern non-crystalline specializations. 

Next, we state a particular consequence of the locally restricted Euler system machinery for higher dimensional coefficient rings obtained through the theory developed in Section~\ref{sec:dimreduction}. 
We note that we do not require any $p$-ordinary hypothesis in the most general form (Theorem~\ref{thm:weakleopoldt} below) of our results. 
In this portion, we shall assume that $\mathcal{R}_\cyc $ is isomorphic to a power series ring in three variables $\mathcal{O}[[X_1,X_2,X_3]]$, 
where $\mathcal{O}$ is the ring of integers of a finite extension of $\QQ_p$. This assumption turns out to be not so restrictive; see \cite[Lemma 2.7]{fouquetochiai} in this regard. 

Recall the subquotient $F^{-+}\TT_\cyc=(\mathbb{T}_{\Bf_1}/F_p^+\TT_{\Bf_1}) \widehat\otimes F_p^+\mathbb{T}_{\Bf_2}\widehat{\otimes}_{\mathbb{Z}_p}(\LL_\cyc^\sharp)^\iota$ of $\TT_\cyc$. We also define the quotient $F^{--}\TT_\cyc =(\mathbb{T}_{\Bf_1}/F_p^+\TT_{\Bf_1}) \widehat\otimes 
(\mathbb{T}_{\Bf_2}/F_p^+\mathbb{T}_{\Bf_2})\widehat{\otimes}_{\mathbb{Z}_p}(\LL_\cyc^\sharp)^\iota$  of $\TT_\cyc$ and the submodules 
$$
F^{++}\TT_\cyc=\left(\mathbb{T}_{\Bf_1}\widehat\otimes F_p^+\TT_{\Bf_2}+ F_p^+\TT_{\Bf_1}\widehat\otimes \TT_{\Bf_2}\right)\widehat\otimes(\LL_\cyc^\sharp)^\iota\supset F_p^+\TT_{\Bf_1}\otimes \TT_{\Bf_2}\widehat\otimes(\LL_\cyc^\sharp)^\iota=: F^{+}\TT_\cyc\,.
$$ 
of $\TT_\cyc$. 
Notice that we have a natural isomorphism $F^{++}\TT_\cyc/F^{+}\TT_\cyc\stackrel{\sim}{\lra}F^{-+}\TT_\cyc$.
{\begin{thmb}[Theorem~\ref{thm:mainconjwithoutpadicLfunc}]
\label{thm:weakleopoldtintro}
For all sufficiently large $p$ and under suitable technical hypotheses (that are made precise in the main text), the following 
statements hold: 
\begin{itemize}
\item[(1)] The Greenberg Selmer group $H^1_{\FF_\Gr^*}(\mathbb{Q},\TT_\cyc^\vee (1))^\vee$ attached to $\TT_\cyc$ is torsion,
\item[(2)] $\textup{char}\left(H^1_{\FF_\Gr^*}(\QQ,\TT_\cyc^\vee (1))^\vee\right)\,\, 
\supset 
\,\, \textup{char}\left(H^1(\QQ_p,F^{-+}_p\TT_\cyc)\big{/}\mathcal{R}_\cyc\cdot \res_{+/\textup{f}}\left(\BF^{\Bf_1,\Bf_2}_{1}\right)\right)$
\end{itemize}
where $\BF^{\Bf_1,\Bf_2}_{1}$ is the generalized Beilinson--Flach element of \cite{KLZ2} associated to the family $\Bf_1\otimes\Bf_2$ of Rankin-Selberg convolutions, and the map $\res_{+/\textup{f}}$ is the compositum of the arrows 
$$\ker\left(H^1(\QQ,\TT_\cyc)\stackrel{\res_p}{\lra} H^1(\QQ_p,F^{--}\TT_\cyc)\right)\stackrel{\res_p}{\lra} H^1(\QQ_p,F^{++}\TT_\cyc){\lra}H^1(\QQ_p,F^{-+}\TT_\cyc)\,.$$
\end{thmb}}
The Greenberg Selmer group and other auxiliary Selmer groups that intervene for technical reasons are defined in Section~\ref{sec:selmerstructures}. We note that the $(-1)$-eigenspace for complex conjugation acting on $\TT_\cyc$ has rank two; this is the reason why the Euler system machinery that was readily available prior to our work does not apply directly. In Section~\ref{sec:dimreduction}, we develop a \emph{locally restricted Euler system machinery}; see Theorem~\ref{thm:weakleopoldt} for our general result in this direction. This is one of the key technical ingredients in proving Theorem~B.

On our way to prove Theorem~B, we also prove a big image result for the big Galois representation associated to a Rankin-Selberg product of two Hida families (Theorem~\ref{thm:MR2holds}). We believe that this result is of independent interest.

Combining our construction in Theorem~A with Theorem~B and the reciprocity formulae in \cite[Theorem 6.5.9]{KLZ1} for Beilinson--Flach elements, we have the following result in favour of the Iwasawa main conjectures for nearly-ordinary families of Rankin-Selberg products. 
{\begin{thmc}[Corollary~\ref{cor:mainthmRankinSelberg}]
\label{thm:main}
For all sufficiently large $p$ and under certain hypothesis (see Corollary~\ref{cor:mainthmRankinSelberg} for a precise statement) we have the following divisibility of integral ideals of $\mathbb{I}\widehat\otimes\mathbb{I}\widehat\otimes\LL_\cyc :$  
$$ 
\textup{char}_{\mathbb{I}\widehat\otimes\mathbb{I}\widehat\otimes\LL_\cyc } 
\left(H^1_{\FF_\Gr^*}(\QQ,\TT_\cyc^\vee (1))^\vee \otimes_{\mathcal{R}_\cyc} \mathbb{I}\widehat\otimes\mathbb{I}\widehat\otimes\LL_\cyc  \right)\,\, 
\supset 
\,\, \left(H \cdot  L_p^{\textup{Hida}}(\Bf_1,\Bf_2,\Bj) \right)\,
$$
where $H \in \mathbb{I}_{\Bf_1}$ denotes Hida's congruence divisor associated to $\Bf_1$ and 
$\mathbb{I}\widehat\otimes\mathbb{I}\widehat\otimes\LL_\cyc$ is a certain local domain that is finite flat over  $\mathcal{R}_\cyc$ (see Definition~\ref{DEF_extensionsofbranches} where we introduce these objects). 
\end{thmc}}
We note that the statement of this theorem involves a minimal amount of correction terms to relate the algebraic $p$-adic $L$-function to the analytic $p$-adic $L$-function, thanks to the fact that our Coleman map $\EXP^*$ is surjective.
\subsection{Companion article}
In the second part of this project, we shall present a slight extension of our Theorem~\ref{cor:dualexponentialmapforH++} on the construction of Coleman maps, that will cover arbitrary base fields that are unramified at all primes above $p$. Our construction will apply, for example, in the context of nearly-ordinary families of automorphic representations for $\textup{GL}_n$ over CM fields. Moreover, as our construction gives rise to a collection of Coleman maps for each subquotient that appears in the $p$-ordinary filtration, we will be able to introduce a general ``rank reduction principle" for higher rank Euler systems over very general coefficient rings. We will prove that this theory is non-vacuous: Building on our work in the current note, we will be able to prove that the Beilinson--Flach Euler system of \cite{LLZ, KLZ2} lifts to a three-variable family of rank-$2$ Euler systems.  This will provide us with the first example of a $p$-adic family of non-trivial higher rank Euler system, verifying (an extension of) a conjecture of Perrin-Riou in this context.

\subsection{Related results}
\label{subsec:history} 
Before we present a detailed account of our results in full generality, we first discuss past work related to the contents of this article.

The pioneering construction of Coleman followed by Perrin-Riou's ground-breaking work allows one to interpolate the Bloch-Kato exponential maps along the cyclotomic deformations of Galois representations. In~\cite{ochiai-AJM03}, the second named author has expanded this construction to nearly-ordinary deformations of Galois representations that are afforded by Hida families. This has been subsequently generalized in \cite{ochiai-AJM14} to a treatment of families of Siegel modular forms. 
We have two major remarks: First of all, our construction is very direct as we do not rely on elements of $p$-adic Hodge theory. 
Secondly, it comes out as natural consequence of our rather hands on construction that the Coleman maps we consider are indeed surjective. 

Next, we discuss older results that relate to the portion of our work on the general theory of Euler systems. After Kolyvagin's celebrated work, Kato, Perrin-Riou and Rubin developed a general machinery of Euler systems to treat Galois representations of core Selmer rank one in the sense of Mazur and Rubin~\cite{mr02} and with coefficients in a DVR or their deformations of character type (i.e., for coefficient rings that arise as the universal deformation rings of characters). The second named author introduced in \cite{ochiai-AIF} an extension of this theory to more general coefficient rings (in fact, as general as we may treat here) but still in the case when the core Selmer rank equals one. In order to handle the cases when the core Selmer rank is arbitrary, the second named author established what he called a locally restricted Euler system machinery in \cite{kbbstark,kbbiwasawa,kbbESrankr,kbbCMabvar}; however, all these works allowed only the treatment of deformations of character type. In the current article, we extend all these to work with a very extensive class of deformation rings. The main difficulty we have to handle has to do with Tamagawa factors (and their effect on various control theorems), which turns out to be somewhat more notorious if we do not allow variation in the cyclotomic direction.

We finally note that a single-variable version of Theorem~C which allows only the cyclotomic variable was already proved in \cite[Theorem 11.6.4]{KLZ2}.
\subsection{General Setup}
Let us fix an odd prime $p$ throughout the paper. 
We fix embeddings of the algebraic closure of the field of rationals 
$\overline{\mathbb{Q}}$ into $\mathbb{C}$ and $\overline{\mathbb{Q}_p}$ 
simultaneously. We also fix a system of norm compatible 
$p$-power roots of unity $\{ \zeta_{p^n}\}$. 
Let $R$ be a complete local Noetherian $\ZZ_p$-algebra of mixed characteristic and finite residue field $\texttt{k}=R/\mm_R$, where we denote by $\mm_R$ the maximal ideal of $R$. Let $K$ be either a totally real or a CM number field. Let $\Sigma$ be a set of places of $K$ that contains all places above $p$ as well as all archimedean places. Let $K_{\Sigma}$ be the maximal extension of $K$ unramified outside $\Sigma$ and set $G_{K,\Sigma}:=\textup{Gal}(K_{\Sigma}/K)$. 
For each place $v$ of $K$, we denote by $K_v$ the completion of $K$ at $v$. We also set $G_v:=\textup{Gal}(\overline{K}_{v}/K_v)$. We denote by $\QQ_\cyc$ 
the cyclotomic $\mathbb{Z}_p$-extension of $\QQ$. 
We denote by $K_\cyc$ the composite $K\QQ_\cyc$. For any prime $\fp$ of $K$ above $p$, we denote by 
$K_{\fp ,\cyc}$ the composite $K_\fp \QQ_\cyc$. 
\par 
Let $\TT$ be a free $R$-module of rank $d$ which is endowed with a continuous action of $G_{K,\Sigma}$. When $K$ is totally real, 
we shall set 
$$
d_+=d_{+}(\TT)=\sum_{v\mid \infty} \textup{rank}_R\,H^0(K_v,\TT)=  \textup{rank}_R\, H^0(\mathbb{R},\textup{Ind}^K_{\QQ}\TT)
$$ 
and we set $d_+=\dfrac{d [K:\QQ]}{2}$ when $K$ is a CM field. In either case, we set 
$$
d_-=d_-(\TT):=d[K:\QQ]-d_+\,.
$$

\begin{define}\label{definition:deformationdatum}
Let $R$ be a complete local Noetherian $\ZZ_p$-algebra of mixed characteristic and finite residue field. 
For a continuous ring homomorphism $\kappa: R \lra \overline{\QQ}_p$\,, we denote by $E_\kappa$ the finite extension 
$\textup{Frac}\left(\kappa (R)\right)$ of $\QQ_p$ and by $\frak{o}_\kappa\subset E_\kappa$ its ring of integers. 
Let $\TT$ be a free $R$-module of finite rank and {let 
$
\mathcal{S} $ be a subset of the set $\textup{Hom}_{\rm cont}(R,\overline{\QQ}_p)$ 
of continuous ring homomorphisms $R \stackrel{\kappa}{\lra} \overline{\QQ}_p$ such that the set $\{\ker{\kappa} \}_{\kappa \in \mathcal{S}}\subset \textup{Spec}(R)$ is Zariski dense. }Then, 
we call a pair $(\TT,R, \mathcal{S})$ a \emph{deformation datum} if it satisfies 
the following three properties (Geo), (Crt) and (Pan) for each prime $\frak{p}$ of $K$ above $p$\,:
\begin{itemize}
\item[(Geo)] The representation $V_\kappa\big{|}_{G_{\fp}}:=\left(\TT\otimes_\kappa E_\kappa\right)\big{|}_{G_{\fp}}$ is de Rham as a $G_\frak{p}$-representation. 
\item[(Crt)] Suppose that (Geo) holds true and write 
$$
V_\kappa\big{|}_{G_{\fp}} \otimes_{E_\kappa}\mathbb{C}_p\cong \bigoplus_{n\in \ZZ} \mathbb{C}_p(n)^{m_n(\fp,\kappa)}\, 
$$
as $G_\frak{p}$-representation. We then have 
$$
\sum_{\frak{p}}d^{(\frak{p})}_+ [K_\fp : \QQ_p ]=d_+(\TT)
$$ 
where we have set $d_+^{(\frak{p})}:=\displaystyle{\sum_{n>0}m_n(\fp,\kappa)}$. 
\item[(Pan)] There is a direct summand $F^+_{\frak{p}}\TT\subset \TT$ (as an $R$-submodule) of rank $ d^{(\frak{p})}_+$ that is stable under the $G_{\frak{p}}$-action and we have 
\begin{align*}
& F^+_{\frak{p}} V_\kappa \otimes_{E_\kappa}\mathbb{C}_p\cong\bigoplus_{n>0} \mathbb{C}_p(n)^{m_n(\fp,\kappa)}, \\ 
& (V_\kappa /F^+_{\frak{p}} V_\kappa ) \otimes_{E_\kappa}\mathbb{C}_p\cong\bigoplus_{n\leq 0} \mathbb{C}_p(n)^{m_n(\fp,\kappa)} ,
\end{align*}
where $F^+_{\frak{p}} V_\kappa $ is the $G_\fp$-stable filtration on $V_\kappa$ induced by $F^+_{\frak{p}}\TT$. 
\end{itemize}
\end{define} 
\begin{rem}
The condition {(Pan)} is called the Panchishkin (or sometimes, Dabrowski-Panchishkin condition) in the literature and it is a generalization of the $p$-ordinary condition. Under (Pan), the condition (Crt) amounts to the requirement that for $\kappa\in \mathcal{S}$, the representation $V_\kappa$ be the $p$-adic realization of a critical motive in the sense of Deligne.
\end{rem}

Let $T_\kappa$ be a fixed $G_{K,\Sigma}$\,-\,stable lattice inside $V_\kappa$ and we set $\overline{T}:=\TT\otimes_R \texttt{k}$ and call it the residual representation of $\TT$. We define $F^+_\frak{p}\,\overline{T}:=F^+_\frak{p}\,\mathbb{T}\otimes_R \texttt{k}$. We also set $T_{\kappa,\textup{cyc}}:=T_\kappa\otimes (\Lambda^\sharp_{\mathrm{cyc}})^\iota$ which is equipped with the diagonal action of  $G_{K,\Sigma}$ and refer to it as the cyclotomic deformation of $T_\kappa$. 
\par 
For any topological abelian group $A$, we denote the Pontrjagin dual 
$\mathrm{Hom}_{\mathrm{cont}} (A , \mathbb{Q} /\mathbb{Z})$ by $A^\vee$. 
{For a finitely generated $R$-module $M$ we denote its $R$-linear dual 
$\mathrm{Hom}_R (M,R)$ by $M^R$. We sometimes use the notation $M^\ast$ for the $R$-linear dual of an $R$-module. 
However an $R$-module $M$ is naturally a module over any subring of $R$. Also $M$ might be regarded as a module over some other rings. Hence 
we use the notation $M^R$ in place of $M^\ast$ to avoid the confusion. }

\subsection{Further Notation and Hypotheses}\label{subsection:furthernotation}

For a prime $\lambda \notin \Sigma$, we denote by $K(\lambda)$ the maximal $p$-extension of $K$ contained in the ray class field of $K$ with module $\lambda$, by $\textup{Fr}_\lambda$ an arithmetic Frobenius at $\lambda$. Let $\mathcal{N}_{\Sigma}$ be the set of square free products of primes $\lambda \notin \Sigma$. For $\eta=\lambda_1\cdots\lambda_r \in \NN_\Sigma$, we write $K(\eta):=K(\lambda_1)\cdots K(\lambda_r)$ for the compositum of the fields $K(\lambda_i)$ and define $K_m(\eta):=K_mK(\eta)$. We define $\Delta_\eta:=\textup{Gal}(K(\eta)/K)\cong \Delta_{\lambda_1}\times\cdots\times  \Delta_{\lambda_r}$ and set $\pmb{\Delta}:=\varprojlim \Delta_\eta$. We denote the compositum of all fields $K(\eta)$ as $\eta$ runs through $\NN_\Sigma$ 
by $\mathcal{K}$. 
We will consider the following conditions on the residual representation $\overline{T}$ of $\TT$:
\begin{enumerate}
\item[(H.0)] For every prime $\frak{p}$ of $K$ above $p$, we have $H^0(K_{\frak{p}},\overline{T})=0$.  
\item[(H.$0^-$)]  For every prime $\frak{p}$ of $K$ above $p$, we have  $H^0(K_{\frak{p}},\overline{T}/F^+_{\frak{p}}\overline{T})=0$.
\item[(H.2)] For every prime $\frak{p}$ of $K$ above $p$, we have $H^2(K_{\frak{p}},\overline{T})=0$. 
\item[(H.$2^+$)] For every prime $\frak{p}$ of $K$ above $p$, we have $H^2(K_{\frak{p}},F^+_{\frak{p}}\overline{T})=0$.
\item[(H.++)] For some prime $\frak{p}_{\textup{o}}$ of $K$ above $p$ of degree $one$, 
there exists an $R$-module direct summand $F^{++}{\TT}$ of $\TT$ which is 
an $R$-module of rank $1+d^{(\frak{p}_\textup{o})}_+$ containing $F^+_{\frak{p}_\textup{o}}\,\TT$ and is stable under $G_{\frak{p}_o}$-action. 
\item[(H.$2^{++}$)] Only under the assumption (H.++) and for a prime $\frak{p}_{\textup{o}}$ specified with the condition (H.++), we have $H^2(K_{\frak{p}_0},F^{++}\overline{T})=0$.
\end{enumerate}

For a finite extension $\mathcal{E}$ of $\QQ_p$, we shall write $\oo_\mathcal{E}$ (or simply $\mathcal{O}$, when it is not necessary to specify $\mathcal{E}$) for its ring of integers.  
\subsection*{Acknowledgements} The first named author (K.B.) acknowledges partial support through EC Grant 745691-CriticalGZ. He also thanks the second named author (T.O.) for his invitation to Osaka University where this project was initiated. 
The second named author (T.O) acknowledges partial support through Grant-in-Aid for Scientific Research (B) of JSPS: Grant Number 26287005. He also 
thanks to Ko\c{c} University for their hospitality during his stay in March 2017. 
Authors would also like thank Nesin Mathematics Village for hosting them as part of a Research in Pairs program.
\tableofcontents
\section{Selmer structures and locally restricted Euler/Kolyvagin systems}
\label{sec:selmerstructures}
In this section, we will describe a variety of Selmer modules that we shall work with. Our ultimate goal is to control the Greenberg Selmer modules (which are defined via the Greenberg Selmer structure $\FF_\Gr$ introduced below)  in terms of what we call locally restricted Euler systems. 
The auxiliary Selmer structures we introduce here play a crucial role in our considerations for this goal. 

We suppose in this section that our coefficient ring $\mathcal{R}$ is isomorphic to a power series ring in $r$ variables 
$\oo[[x_1,x_2,\ldots,x_r]]$ over 
the ring of integers $\oo$ of a finite extension of $\QQ_p$. Let $(\TT,\mathcal{R},\mathcal{S})$ denote a deformation datum as above.

\begin{define} For every prime $\lambda$ of $K$ and any subquotient $X$ of $\TT_\cyc$, let 
$$H^1_{\textup{ur}}(K_\lambda,Y)=\ker\left(H^1(K_\lambda,Y)\ra H^1(K^{\textup{ur}}_\lambda,Y)\right)$$
denote the collection of unramified cohomology classes, where $Y=X$ or $X^\vee(1)$.
\end{define}
For any subquotient $Y$ of $\TT$, we write $Y_\cyc$ for its cyclotomic deformation.
\begin{lemma}
\label{lem:MRunramifiedisallunrforcyclo}
{Let $\lambda\nmid p$ be a prime of $K$.} For any subquotient $Y$ of $\TT$, $H^1_{\textup{ur}}(K_\lambda,Y_\cyc)=H^1(K_\lambda,Y_\cyc)$.
\end{lemma}
\begin{proof}
This follows from the proof of Lemma~5.3.1(ii) of \cite{mr02}.
\end{proof}
\begin{prop}
\label{prop:localcohom}\label{equation:vanishingH^2X}
Let $R$ be a complete local Noetherian $\ZZ_p$-algebra of mixed characteristic with finite residue field $\texttt{k}=R/\mm_R$. 
Assume that $R$ is regular. 
Let $X$ be a free $R$-module of finite rank with a continuous $G_\frak{p}$-action satisfying 
$H^0(K_\frak{p},X\otimes \texttt{k})=H^2(K_\frak{p},X\otimes \texttt{k})=0$. 
\par 
Then the $R$-module $H^1(K_{\frak{p}},X)$ is a free $R$-module of rank $d_X:= [K_\frak{p}:\QQ_p]\, \textup{rank}_R\, X$\,.
\end{prop}
\begin{proof}
First, the assumption $H^2(K_\frak{p},X\otimes \texttt{k})=0$ implies that
\begin{equation}\label{equation:vanishingH^2X}
H^2(K_{\fp},Y)=0 \text{ for every $R$-quotient $Y$ of $X$} 
\end{equation}
thanks to Nakayama's lemma and to the fact that the cohomological dimension of $G_{\fp}$ is $2$.
\par 
Now let us choose and fix a regular sequence $r_1 ,\ldots ,r_l$ in $R$ such that $\mm_R$ is generated by $r_1 ,\ldots ,r_l$. 
Note that for each $k$ satisfying $1\leq k \leq l$, we have the exact sequence 
$$
0 \lra X_{k-1 } \xrightarrow{\text{ $\times r_k$ }} X_{k-1} \lra 
 X_k \lra 0
$$
where $X_k$ stands for $X/(r_1 ,\ldots ,r_{k-1})X$. 
This short exact sequence induces the long exact sequence of Galois gohomology 
\begin{multline}\label{equation:longexactseqH^0H^1}
0 \lra H^0 (K_\frak{p},X_{k-1 } )\overset{\times r_k}{\lra} 
H^0 (K_\frak{p}, X_{k-1 }) \lra 
H^0 (K_\frak{p}, X_{k}) \\ 
\lra H^1 (K_\frak{p},X_{k-1 } )\overset{\times r_k}{\lra} 
H^1 (K_\frak{p},X_{k-1 }) \lra 
H^1 (K_\frak{p}, X_{k}) \lra 0. 
\end{multline}
Here the surjectivity of the final map follows from \eqref{equation:vanishingH^2X}. 
Notice that $X_{l}$ is nothing but $X\otimes \texttt{k}$. 
For every $k$, our assumption that $H^0(K_\frak{p},X\otimes \texttt{k})=0$ implies $H^0(K_\frak{p},X_k)=0$ (by reverse induction on $k$, relying on Nakayama's lemma), and in turn
\begin{equation}\label{equation:inductivelyfree}
H^1 (K_\frak{p}, X_{k-1 })[r_k] =0 \text{ and } \dfrac{H^1 (K_\frak{p},X_{k-1 })}{r_k H^1 (K_\frak{p},X_{k-1 })} \cong H^1 (K_\frak{p},X_{k }).
\end{equation} 

Recall for $k=l$ that $H^1 (K_\frak{p},X_{l})= H^1(K_\frak{p},X\otimes \texttt{k})$ is free of 
rank $d_X$ over $\texttt{k}=R/(r_1 ,\ldots ,r_l)$ by our running assumptions that $H^0(K_\frak{p},X\otimes \texttt{k})=H^2(K_\frak{p},X\otimes \texttt{k})=0$ 
and by Tate's local Euler characteristic formula. By Nakayama'a lemma and the second portion of \eqref{equation:inductivelyfree}, 
we have 
\begin{equation}\label{equation:surjectionXeveryk}
\left( R/(r_1 ,\ldots ,r_{k-1})\right)^{\oplus d_X} \twoheadrightarrow H^1 (K_\frak{p}, X_{k-1}) \text{ for $k=1,\ldots ,l+1$}. 
\end{equation}
On applying the functor $(- \otimes_R R/(r_1 ,\ldots ,r_{k}))$ to \eqref{equation:surjectionXeveryk} and using Nakayama's lemma and \eqref{equation:inductivelyfree}, 
we inductively show that the kernel of \eqref{equation:surjectionXeveryk} is zero for $k=1,\ldots ,l+1$, which completes the proof.  
\end{proof}
\begin{define}\label{definition:TandR} 
Let $\mathbf{T}:=\TT$ (resp.$\mathbf{T}:=\TT_\cyc$) be a module over $R=\mathcal{R}$ (resp. $R=\mathcal{R}_\cyc$). 
For every finite abelian $p$-extension $L$ of $K$, we define the semi-local cohomology group 
$${\displaystyle{H^1(L_p,F^\pm\mathbf{T}):=\bigoplus_{\frak{p}|p}\bigoplus_{\frak{q} \mid \fp}\,H^1(L_{\frak{q}},F^\pm_\frak{p}\mathbf{T})}}.$$ 
\end{define}
Until the end of Section~\ref{sec:selmerstructures}, we will let $\mathbf{T}$ denote any one of $\TT$ or $\TT_\cyc$ and depending on which one, we shall let $R=\mathcal{R}$ or $\mathcal{R}_\cyc$ denote the corresponding coefficient ring. 
\begin{cor}
\label{cor:structuresemilocalcohom}
Under the same setting as Definition \ref{definition:TandR}, 
if the hypotheses (H.0) and (H.2) hold true, then the $R$-module $H^1(K_p,\mathbf{T})$ is free of rank $d$. 
If in addition (H.$2^+$) (resp. (H.$0^-$)) holds true, then $H^1(K_p,F^+\mathbf{T})$ (resp.  $H^1(K_p,F^-\mathbf{T})$) is free of rank $d_+$ $($resp.  of rank $d_-$$)$. 
\end{cor}
\begin{rem}
\label{rem:quotients}
Let $I$ be an ideal of $R$ and set $X:=\bfT\otimes R/I$. Then the conclusions of Corollary~\ref{cor:structuresemilocalcohom} hold verbatim when $\bfT$ is replaced by the $R/I$-module $X$.
\end{rem}

For every finite abelian $p$-extension $L$ of $K$ and prime ${\frak{q}} \mid \fp$ of $L$, we define the local Greenberg submodule 
\begin{equation}\label{equation:definitionGR11}
H^1_{\FF_{\Gr}}(L_{\frak{q}} ,\bfT):=\textup{im}\left(H^1(L_{\frak{q}} ,F^+_\fp\bfT)\ra H^1(L_{\frak{q}} ,\bfT)\right)
\end{equation}
as well as the  semi-local Greenberg submodule $H^1_{\FF_\Gr}(L_p,\bfT)$ by setting
\begin{equation}\label{equation:definitionGR22}
H^1_{\FF_\Gr}(L_p,\TT):=\textup{im}\left(H^1(L_p,F^+\bfT)\ra H^1(L_p,\bfT)\right)\,.
\end{equation}
Note that under the hypothesis (H.$0^-$), $H^1_{\FF_\Gr}(L_p,\TT_\cyc)$ is the isomorphic copy of the module $H^1(L_p,F^+\TT_\cyc)$. 
\begin{cor}
\label{cor:greenbergfree}
If (H.0), (H.2), \hzerominus\, and \htwoplus\, hold true, then $H^1_{\FF_\Gr}(K_p,\bfT)$ is a free direct summand (of rank $d_+$) of the free $R$-module $H^1(K_p,\bfT)$.
\end{cor}

\begin{lemma}
\label{lem:hnaetclifts} 
Suppose that the hypotheses \hzero, \htwo, \hzerominus\, and \htwoplus\, hold true. 
For $\eta \in \NN_\Sigma$ and for every prime $\frak{q}$ of $K(\eta)$ above the prime $\frak{p}$ of $K$ (see \ref{subsection:furthernotation} for the notation 
$K(\eta)$), we have 
$$H^0(K(\eta)_\frak{q},\overline{T})=H^2(K(\eta)_\frak{q},\overline{T})=0\,,$$
$$H^0(K(\eta)_\frak{q},F^+_\fp\overline{T})=H^0(K(\eta)_\frak{q},F^-_\fp\overline{T})=0\,.$$
\end{lemma}
\begin{proof} Write 
$\Delta_{\eta , \frak{q}} \subset \Delta_\eta$ for the decomposition group of $\frak{q}$ and identify it with the local Galois group
$\Gal(K(\eta)_\frak{q}/K_\fp)$. Let $X$ be any $G_\fp$-module of finite $p$-power cardinality and such that $H^0(K_\fp,X)=0$.  If $\Delta_{\eta , \frak{q}}$ is trivial, then $H^0(K(\eta)_\frak{q},X)=H^0(K_\fp,X)=0$. If $\Delta_{\eta , \frak{q}}$ is not trivial, note that 
$\Delta_{\eta , \frak{q}}$ a non-trivial $p$-group by the definition of $K(\eta )$. 
Therefore we have 
$\#H^0\left(K(\eta)_\frak{q}, X \right) \equiv\# H^0\left(K(\eta)_\frak{q},X\right)^{\Delta_{\eta , \frak{q}}}$ mod $p$. 
By assumption, the group $H^0\left(K_\fp,X\right)= H^0\left(K(\eta)_\frak{q},X\right)^{\Delta_{\eta , \frak{q}}}$ is trivial 
and thus $\# H^0\left(K(\eta)_\frak{q},X\right)^{\Delta_{\eta , \frak{q}}}=1$. 
Hence, the cardinality of the set $H^0\left(K(\eta)_\frak{q},X\right)$ is a natural number dividing $p$ and congruent to $1$ mod $p$ 
and we conclude that $H^0\left(K(\eta)_\frak{q},X\right)=0$. 
The proof of the lemma follows using this fact with $X=\overline{T},\overline{T}^{\vee}(1), F_\fp^+\overline{T}$ and $F_\fp^-\overline{T}$ along with our running hypothesis.
\end{proof}
\begin{prop}\label{prop:surj}
Suppose that \hzero\, and \htwo\, hold true. Then for every $\eta,\mu \in \NN_\Sigma$ with $\eta \mid \mu$, the corestriction map
\begin{equation}\label{equation:corestriction}
H^1(K(\mu)_p,\bfT)\lra H^1(K(\eta)_p,\bfT)
\end{equation}
on the semi-local cohomology at $p$ is surjective.
 \end{prop}
 
\begin{proof}
By Tate local duality theorem, in order to prove the map \eqref{equation:corestriction} is surjective, 
it suffices to prove that restriction map 
\begin{equation}\label{equation:restriction}
H^1(K(\eta)_p,{\bfT}^\vee (1)) \lra H^1(K(\mu)_p,{\bfT}^\vee (1))
\end{equation}
is injective. To prove this, it suffices to show that 
\begin{equation}\label{equation:restriction2}
H^1(K(\eta)_\fp,{\bfT}^\vee (1)) \lra H^1(K(\mu)_{\frak{q}},{\bfT}^\vee (1))
\end{equation}
is injective for any prime $\fp$ of $K(\eta )$ and for any prime $\frak{q}$ of $K(\mu)$ over $\frak{p}$. 
As $K(\mu)_{\frak{q}} / K(\eta)_\fp$ is a succession of cyclic $p$-extensions, it suffices to prove that the map
\eqref{equation:restriction2} is injective assuming that $K(\mu)_{\frak{q}} / K(\eta)_\fp$ is a cyclic $p$-extension. By the inflation-restriction sequence of Galois cohomology, the kernel of \eqref{equation:restriction2} is isomorphic to the $\mathrm{Gal}(K(\mu)_{\frak{q}} / K(\eta)_\fp)$-coinvariants  of $H^0 (K(\mu)_{\frak{q}},{\bfT}^\vee (1))$. It follows by Lemma~\ref{lem:hnaetclifts} that the cohomology group $H^0 (K(\mu)_{\frak{q}},{\bfT}^\vee (1))$ vanishes and this completes the proof.  
\end{proof}

\begin{prop}
\label{prop:extendpropabovetominusplus}
In addition to \hzero\, and \htwo, suppose that \hzerominus\, and \htwoplus\, hold true as well. Then for every $\eta,\mu \in \NN_\Sigma$ with $\eta \mid \mu$, the corestriction map
 $$H^1(K(\mu)_p,F^\pm\bfT)\lra H^1(K(\eta)_p,F^\pm\bfT)$$ on the semi-local cohomology at $p$ is surjective.
\end{prop}
\begin{proof}
The additional hypothesis allows the proof of Proposition~\ref{prop:surj} work verbatim for $F^\pm\bfT$.
\end{proof}

 \begin{prop}
 \label{prop:semilocalstructure}  Let $R$ be a complete local Noetherian $\ZZ_p$-algebra of mixed characteristic with finite residue field $\texttt{k}=R/\mm_R$. 
Assume that $R$ is regular. 
 If \hzero\, and \htwo\, hold true, then for every $\eta \in \NN_\Sigma$, the semi-local cohomology group $H^1(K(\eta)_p,\bfT)$ is a free $R[\Delta_{\eta}]$-module of rank $d$.
 \end{prop}

 \begin{proof}
By Shapiro's lemma on Galois cohomology, we have $H^1(K(\eta)_p,\bfT)\cong H^1(K_p,\bfT \otimes_R R [\Delta_\eta ]^\sharp )$ where $R[\Delta_\eta ]^\sharp $ is a free $R[\Delta_\eta ]$-module of rank $1$  
on which $\mathrm{Gal}(\overline{K} /K)$ acts tautologically. By assumption, $R[\Delta_\eta ]$ is a semi-local ring whose local components are all 
regular. 
By Proposition \ref{prop:surj}, the corestriction map $H^1(K(\eta)_p,\bfT)\lra H^1(K_p,\bfT )$ is equal to a surjective map 
induced by the functor $- \otimes_{R[\Delta_\eta ]} R$. 
Applying the argument of the proof of Proposition~\ref{prop:localcohom} componentwise on $R [\Delta_\eta ]$, 
$H^1(K(\eta)_p,\bfT)$ is a free $R$-module of rank $d\cdot|\Delta_{\eta}|$. 
 \end{proof}

\begin{prop}
\label{prop:semilocalstructureforminusplus} In addition to \rm{\hzero}\,and \rm{\htwo}, suppose that  
\rm{\hzerominus}\, and \rm{\htwoplus}\, hold true as well. Then for every $\eta \in \NN_\Sigma$, the $R[\Delta_\eta]$-module $H^1(K(\eta)_p,F^\pm\bfT)$ is free of rank $d_{\pm}$.
 \end{prop}
\begin{proof}
The additional hypothesis allows the proof of Proposition~\ref{prop:semilocalstructure} work verbatim for $F^+\bfT$ and $F^-\bfT$.
\end{proof}
 \begin{cor}
 \label{cor:thickfree}
Assuming \hzero\,and \htwo\, hold, $\displaystyle{\varprojlim_{\mu \in \mathcal{N}_\Sigma} }H^1(K(\mu)_p,\bfT)$ is a free $R[[\pmb{\Delta}]]$-module of rank $d$ and the natural projection map 
  $$
  \varprojlim_{\mu \in \mathcal{N}_\Sigma} H^1(K(\mu)_p,\bfT) \lra H^1(K(\eta)_p,T)
  $$ 
  is surjective for every $\eta \in \NN_\Sigma$ and quotient $T$ of $\bfT$. 
  
  If \htwoplus\, $($resp.  \hzerominus$)$ holds true as well, then $\displaystyle{\varprojlim_{\mu \in \mathcal{N}_\Sigma} }H^1(K(\mu)_p,F^+\bfT)$ 
(resp. the module $\displaystyle{\varprojlim_{\mu \in \mathcal{N}_\Sigma} } H^1(K(\mu)_p,F^-\bfT)$) is a free $R[\pmb{\Delta}]$-module of rank $d_+$ (resp. of rank $d_-$).
 \end{cor}
 \begin{proof}
The first portion is immediate after Propositions~\ref{prop:surj} and \ref{prop:semilocalstructure}. The second assertion follows from Propositions~\ref{prop:extendpropabovetominusplus} and \ref{prop:semilocalstructureforminusplus}.
 \end{proof}
\begin{cor}
 \label{cor:thickfreeGreenberg}
Suppose that the hypotheses \hzero, \htwo, \hzerominus\, and 
\htwoplus\, hold true. Then the submodule $\displaystyle{\varprojlim_{\mu \in \mathcal{N}_\Sigma} } H^1_{\FF_\Gr}(K(\mu)_p,\bfT) \subset 
\displaystyle{\varprojlim_{\mu \in \mathcal{N}_\Sigma} } H^1(K(\mu)_p,\bfT)$ is a free direct summand of rank $d_+$\,.
\end{cor}

If we assume in addition that (H.++) holds true, there there is a more natural way to define the auxiliary module $H^1_{++}(F_p,\bfT)$. When (H.++)  is valid, we shall always stick to this more natural definition (given as in Definition~\ref{def:auxiliarylineo} below) without any further warning. 
\begin{define}
\label{def:auxiliarylineo}
Let $F/K$ be a finite subextension in $\KKK/K$. If the hypotheses \hzero, \htwo, \hzerominus, \htwoplus\, as well as (H.++) hold true, we set 
$$H^1_{++}(F_p,\bfT):=\left(\oplus_{\substack{\frak{p} \mid p \\ \frak{p} \nmid \fp_{\textup{o}}}}H^1_{\FF_{\Gr}}(F_\frak{p},\TT)\right)\oplus 
\left(\oplus_{{\frak{p} \mid \fp_{\circ}}}\, \textup{im}\left(H^1(F_{\frak{p}},F^{++}\bfT)\ra H^1(F_{\frak{p}},\bfT)\right)\right)$$
and define $\mathbb{V}_F:=\left(\oplus_{{\frak{p} \mid \fp_{\circ}}}\, \textup{im}\left(H^1(F_{\frak{p}},F^{++}\bfT)\ra H^1(F_{\frak{p}},\bfT)\right)\right) \,.
$
\end{define}
By the definition of filtrations given before Theorem B and 
by the definitions \eqref{equation:definitionGR11} and \eqref{equation:definitionGR22}, we have 
$H^1_{\FF_{\Gr}}(F_{\frak{p}_0},\bfT) \subset \mathbb{V}_F$. 
\subsection{Selmer structures and Selmer groups}
\label{subsec:selmerstr}
We now define various Selmer structures attached to the deformation datum $(\TT,\mathcal{R},\mathcal{S})$ and 
its cyclotomic deformation $(\TT_\cyc,\mathcal{R}_\cyc,\mathcal{S}_{\cyc})$ and we shall rely on these structures throughout this article.  
\par 
Let $\Sigma^{(p)}\subset \Sigma$ denote the set of places of $K$ that do not lie above $p$.
We retain our convention from the previous section concerning our use of the symbols $\bfT$ and $R$. Let $L$ be any finite extension of $K$.
\begin{define}
\label{def:Selmerstructures} 
\textbf{(1)} The \emph{canonical Selmer structure} $\FFc$ is defined via the local conditions
\begin{itemize}
\item[(i)] $H^1_{\FFc}(L_\lambda,\bfT):=\ker\left(H^1(L_\lambda,\bfT)\lra H^1(L_\lambda^\textup{ur},\bfT) \right)$ at primes $\lambda$ of $L$ lying above those in $\Sigma^{(p)}$,
\item[(ii)] $H^1_{\FFc}(L_p,\bfT)=H^1(L_p,\bfT)$ at primes above $p$.
\end{itemize}
\textbf{(2)} Whenever the hypothesis (Pan) holds, the \emph{Greenberg Selmer structure} $\FF_{\Gr}$ is defined via the local conditions 
\begin{itemize}
\item[(i)] $H^1_{\FF_{\Gr}}(L_\lambda,\bfT):=H^1_{\FFc}(L_\lambda,\bfT)$ at primes $\lambda$ of $L$ lying above those in $\Sigma^{(p)}$,
\item[(ii)] 
$H^1_{\FF_{\Gr}}(L_p,\bfT):=\textup{im}\left(H^1(L_p,F^+\bfT) \lra H^1(L_p,\bfT)\right)$
at primes above $p$.
\end{itemize}
\textbf{(3)} The Selmer structure $\FF_{+}$ is defined via the local conditions 
\begin{itemize}
\item[(i)] $H^1_{\FF_{+}}(L_\lambda,\bfT):=H^1_{\FFc}(L_\lambda,\bfT)$ at primes $\lambda$ of $L$ lying above those in $\Sigma^{(p)}$,
\item[(ii)] $H^1_{\FF_+}(K_p,\bfT):=H^1_{++}(K_p,\bfT)$ at primes above $p$.
\end{itemize}
\end{define}
\begin{rem}
Notice in the situation when $\bfT=\TT_\cyc$, it follows from Lemma~\ref{lem:MRunramifiedisallunrforcyclo} that $H^1_{\FF_{\Gr}}(K_\lambda,\TT_\cyc)=H^1(K_\lambda,\TT_\cyc)$ for every $\lambda\in \Sigma^{(p)}$.
\end{rem}
\begin{define}
\label{define:selmerobjects}
Given a Selmer structure $\FF$ on $\bfT$, we may \emph{propagate} it to a quotient $X$ via Example 1.1.2 of \cite{mr02}. 
We define $H^1_\FF (K_\lambda ,X)^{~}$ to be the image of $H^1_\FF (K_\lambda ,\bfT ) \subset H^1 (K_\lambda ,\bfT )$ via 
$H^1 (K_\lambda ,\bfT ) \lra H^1 (K_\lambda ,X )$. Then we define the Selmer group associated to $\FF$ by setting
$$H^1_{\FF}(K,X):=\ker \left(H^1(K_\Sigma/K,X)\lra \bigoplus_{\lambda\in \Sigma}\frac{H^1(K_\lambda,X)}{H^1_\FF(K_\lambda,X)}\right)$$
\par 
We also define the dual Selmer structure $\FF^*$ on the dual representation $X^\vee(1)$ by setting (for every place $\lambda$ of $K$)
$$H^1_{\FF^*}(K_\lambda,X^\vee(1)):=H^1_{\FF}(K_\lambda,X)^\perp\,,$$
the orthogonal complement of $H^1_{\FF}(K_\lambda,X)$ under the local Tate pairing, and similarly, the dual Selmer group $H^1_{\FF^*}(K,X^\vee(1))$\,.
\end{define}

\subsection{Locally restricted Euler systems and descend to Kolyvagin systems}
\subsubsection{Existence of locally restricted Kolyvagin systems}
\label{subsec:locallyrestricKSexists}
For each prime $\lambda\notin \Sigma$, we set 
$$P_\lambda(X):=\det\left(1-\textup{Fr}_\lambda^{-1} X \mid  \TT^{R}(1)\right)
$$
where $\TT^R=\textup{Hom}(\TT,R)$ is the $R$-linear dual of $\TT$.

\begin{define}
\label{define_ES}
Suppose that we have a pair $(\TT_\cyc,\KKK)$.  
A collection of elements $\mathbf{c}:=\{c_{K(\eta)}\}_{\eta\in \NN_{\Sigma}}$ 
such that $c_{K(\eta)}\in H^1(K(\eta)_\Sigma/K(\eta),\TT_\cyc)$ is called an \emph{Euler system} if the following conditions hold: 
\begin{itemize}
\item[(ES1)] for each $\eta\lambda,\eta\in \NN_\Sigma$ we have $\textup{Cor}_{K(\eta\lambda)/K(\eta)}\, c_{K(\eta\lambda)}=P_\lambda(\textup{Fr}_\lambda^{-1}) c_{K(\eta)}$\,,\\
\item[(ES2)] for every $\lambda\mid \eta \in \NN_\Sigma$ we have $\textup{Cor}_{K(\eta\lambda)/K(\eta)}\, c_{K(\eta\lambda)}= c_{K(\eta)}$\,. \\

\end{itemize}
\end{define}

We denote the collection of Euler systems for the triple $(\TT_\cyc,\Sigma,\KKK)$ by $\textup{ES}(\TT_\cyc)$. 
\begin{rem}
\label{rem:altertheeulerfactors}
The \emph{Euler polynomials} $P_\lambda(X)$ that appear in the definition of an Euler system above may be altered to yield equivalent theories; see \cite[\S\, IX]{r00} for a discussion in this regard and \cite[Lemma 7.3.4]{LLZ} where this alteration is utilized on the Beilinson--Flach Euler system.
\end{rem}
\begin{define}\label{definition:locallyrestrictedES}
\label{def:locallyrestrictedES}
An Euler system $\mathbf{c}\in \textup{ES}(\TT_\cyc)$ is called locally restricted if we have 
$$\res_p\left(c_F\right) \in H^1_{++}(F_p,\TT_\cyc)$$
for every finite extension $F$ of $K$ contained in $\KKK$. We denote the collection of locally restricted Euler systems by $ \textup{ES}^+(\TT_\cyc)$\,.
\end{define}
For an $(r+2)$-tuple ${\frak{s}}=(s_0,s_1,\ldots,s_r, s_{r+1}) \in \left(\ZZ^+\right)^{r+2}$ and a fixed topological generator $\gamma$ of $\Gamma_\cyc$, we define the ideal 
$$\mathcal{I}_{\frak{s}}:=(\pi_\oo^{s_0},X_1^{s_1},\ldots, X_{r}^{s_r}, (\gamma-1)^{s_{r+1}}) \subset \mathcal{R}_\cyc$$ 
and set $\TT_\frak{s}:=\TT_\cyc/\mathcal{I}_{\frak{s}}\TT_\cyc$. 
We denote the $(r+2)$-tuple $(1,1,\ldots, 1)$ by $\frak{l}$. We note that 
$\mathcal{I}_\frak{l}=\mm_{\mathcal{R}_\cyc}$ is the maximal ideal and $\TT_\frak{l}=\overline{T}$ is the residual representation.  For tuples $\frak{s}$ and $\frak{s}^\prime$ we write $\frak{s}\preccurlyeq\frak{s}^\prime$ to mean that $s_i\leq s_{i+1}$ for $i=0,1,\ldots, r+1$. 
\begin{define}
\label{define:KSprimesprimaryform}
We define the set of primes $\PP_{{\frak{s}}}$ to be the set that consists of primes $\lambda \notin \Sigma$ of $K$ that satisfy
\begin{enumerate}
\item[{\rm (K.1)}] $\TT/(\mathcal{I}_{{\frak{s}}}  \TT + (\textup{Fr}_{\lambda}-1)\TT)$ is a free $R/\mathcal{I}_{{\frak{s}}}$\,-module of rank one,
\item[{\rm (K.2)}] $\mathbf{N}\lambda-1 \in \mathcal{I}_\frak{s}$.
\end{enumerate}
and such that $\PP_\frak{s}\supset \PP_{\frak{s}^\prime}$ whenever $\frak{s} \preccurlyeq \frak{s}^{\prime}$. The collection $\{\PP_\frak{s}\}$ is called the collection of \emph{Kolyvagin primes}. 
\end{define}

We shall prove (under suitable hypotheses) in Lemma~\ref{lem:Kolyvaginprimesexist} that a collection of Kolyvagin primes exists.

For a Selmer structure $\FF$ on $\TT$ and $\frak{s}$ as above, we may define the module of Kolyvagin systems on the artinian module $\TT_\frak{s}$. We will not include its precise definition in this note and refer the reader to \cite{mr02, kbbdeform}.  Given a Selmer structure $\FF$ on $\TT_\cyc$, we let $\KS(\TT_\frak{s},\FF,\PP_\frak{s})$ denote the module of Kolyvagin systems for the Selmer triple $(\TT_\frak{s},\FF,\PP_\frak{s})$ and set 
$$\KS(\TT,\FF):=\varprojlim_\frak{s} \varinjlim_{\frak{j}\succcurlyeq \frak{s}}\KS(\TT_\frak{s},\FF,\PP_\frak{j})\,.$$
\begin{define}
\label{def:restrictedKS} The collection $\KS(\TT,\FF_+)$ is called the module of \emph{locally restricted Kolyvagin systems.}
\end{define}
\subsubsection{Existence of Kolyvagin primes}
\label{subsub:Kolyvaginprimesexist}
In this section, we explain how to obtain a useful collection of Kolyvagin primes under suitable hypotheses. These hypotheses are exactly those considered in \cite[Section 3.5]{mr02} and are often satisfied. The properties we will rely on are as follows.
\begin{itemize}
\item[{(MR1)}] $\overline{T}$ is absolutely irreducible as $G_K$-module.
\item[{(MR2)}] There exists an element $\tau\in G_K$ with the properties that 
\begin{itemize}
\item $\TT/(\tau-1)\TT$ is a free $R$-module of rank one,
\item $\tau$ acts trivially on $\pmb{\mu}_{p^\infty}$.
\end{itemize}
\item[{(MR3)}] $H^0(K,\overline{T})=H^0(K,\overline{T}^\vee(1))=0$\,.
\item[{(MR4)}] Either $\textup{Hom}_{\texttt{k}[[G_K]]}(\overline{T},\overline{T}^\vee(1))=0$; or else $p>4$\,.
 \end{itemize}
 Note that the hypotheses \hzero\, implies the first vanishing condition in {\rm (MR3)}.
 \begin{lemma}
 \label{lem:Kolyvaginprimesexist}
 Suppose that the hypothesis 
 {{\rm (MR2)}}\, holds true and fix $\tau \in G_K$ satisfying {{\rm (MR2)}}. Consider the set of primes 
 \be\label{def:goodsetofKolyprimes}\frak{P}_\frak{s}:=\{\lambda \in \NN_\Sigma: \textup{Fr}_\lambda \hbox{ is conjugate to } \tau \hbox{ in } 
 \Gal\left(K(\TT_\frak{s},\pmb{\mu}_{p^{s_0}})/K\right)\}\,.
 \ee
 where $K(\TT_\frak{s},\pmb{\mu}_{p^{s_0}})$ is the fixed field of $\ker\left(G_{K,\Sigma}\ra \textup{Aut}\left(\TT_\frak{s}\oplus \pmb{\mu}_{p^{s_0}}\right)\right)$\,.
 Then the set $\frak{P}_\frak{s}$ verifies the properties {\rm (K.1)} and {\rm (K.2)} above, and furthermore, we have $\frak{P}_\frak{s} \supset \frak{P}_{\frak{s}^\prime}$ if $\frak{s}\preccurlyeq\frak{s}^\prime$.
 \end{lemma}
 \begin{proof}
 It is clear that the set $\frak{P}_s$ verifies the properties ${(K.1)}$ and ${(K.2)}$. 
 Fix a pair $$
 \frak{s}=(s_0,s_1,\ldots,s_r, s_{r+1})\preccurlyeq\frak{s}^\prime =(s^\prime_0,s^\prime_1,\ldots,s^\prime_r, s^\prime_{r+1})
 $$ and note that $K(\TT_{\frak{s}^\prime},\pmb{\mu}_{p^{s_0^{\prime}}})\supset K(\TT_\frak{s},\pmb{\mu}_{p^{s_0}})$. 
 Then a prime $\lambda \in \mathcal{N}_\Sigma$ belongs to $\frak{P}_{\frak{s}^\prime}$ if and only if there exists $\sigma\in G_K$ such that 
 $$\sigma  \textup{Fr}_\lambda  \sigma^{-1} \equiv \tau \hbox{ within $\mathrm{Gal} (K(\TT_{\frak{s}^\prime},\pmb{\mu}_{p^{s_0^{\prime}}})/K)$}.$$
 Hence, we have $\sigma  \textup{Fr}_\lambda  \sigma^{-1} \equiv \tau 
 \hbox{ in the quotient $\mathrm{Gal} (K(\TT_\frak{s},\pmb{\mu}_{p^{s_0}})/K)$}$ of 
 $\mathrm{Gal} (K(\TT_{\frak{s}^\prime},\pmb{\mu}_{p^{s_0^{\prime}}})/K)$ as well, and hence $\lambda \in \frak{P}_{\frak{s}}$ as claimed.
 \end{proof}
 We will take the set of primes $\frak{P}_\frak{s}$ in (\ref{def:goodsetofKolyprimes}) to be the set of primes $\PP_\frak{s}$ that are required in Definition~\ref{define:KSprimesprimaryform}. 
\subsubsection{Euler systems to Kolyvagin systems map}
\label{subsub:EStoKS}
Let $\FFc$ denote the canonical Selmer structure on $\TT_\cyc$, obtained from the Selmer structure $\FF_\Gr$ by relaxing the local conditions at all primes above $p$. The following theorem is proved in \cite[Appendix A]{mr02} only when the base field $K$ equals $\QQ$ and when the coefficient ring is either a discrete valuation ring or the cyclotomic Iwasawa algebra. The arguments go through for a general coefficient ring and base field to yield the following result:
\begin{thm}[Mazur-Rubin]
\label{thm:ESKSgeneral}
Assume that the hypotheses {\rm (MR1)-(MR4)} hold true and suppose for every $\lambda \in \PP_\frak{l}$ the homomorphisms $\{\textup{Fr}_{\lambda}^{p^k}-1\}_{k\in \ZZ}$ are injective on $\TT$.  Then there exists a map 
$$
\Psi^{\textup{MR}}:\,\textup{ES}(\TT_\cyc)\lra \KS(\TT_\cyc,\FFc,\frak{P}_{\frak{l}})
$$
such that $c_K \in H^1(K_{\Sigma}/K,\TT_\cyc)$ coincides with $\kappa_1 \in H^1(K_{\Sigma}/K,\TT_\cyc)$
if the Euler system $\mathbf{c}=\{c_{K(\eta)}\}_{\eta\in \NN_{\Sigma}} \in \textup{ES}(\TT_\cyc)$ and $\mathbf{c}$ maps to $\pmb{\kappa}=\{\kappa_\eta\}\in \KS(\TT_\cyc,\FFc,\frak{P}_\frak{l})$ 
where $\frak{P}_\frak{l}$ is the set of Kolyvagin primes from (\ref{def:goodsetofKolyprimes}).
\end{thm}
\begin{thm}
\label{thm:ESKSlocallyrestricted}
In addition to the hypotheses of Theorem~\ref{thm:ESKSgeneral}, assume that the hypotheses {\rm (H.0), (H.2), (H.$0^-$)} and {\rm (H.$2^+$)} hold true. Then 
the map $\Psi^{\textup{MR}}$ restricted to $\textup{ES}^{+}(\TT_\cyc) \subset \textup{ES}(\TT_\cyc) $ (see Definition \ref{definition:locallyrestrictedES} 
for the definition of $\textup{ES}^{+}(\TT_\cyc)$) induces a map 
$$\Psi^{\textup{MR}}:\,\textup{ES}^{+}(\TT_\cyc)\lra \KS(\TT_\cyc,\FF_+)\,.$$ 
This assertion remains valid if we use the local conditions given by Definition~\ref{def:auxiliarylineo} assuming the truth of {\rm (H.$2^{++}$)}.
\end{thm}
\begin{proof}
Suppose $\mathbf{c}=\{c_{K(\eta)}\}\in \textup{ES}^{+}(\TT_\cyc)$ and we set
$$
\Psi^{\textup{MR}}(\mathbf{c})=\{\kappa_\eta(\frak{s})\}_{\eta \in \NN_{\frak{s}}} \in \KS(\TT_\cyc,\FFc,\frak{P}_\frak{l}),$$ 
where $\kappa_\eta(\frak{s}) \in H^1(K_\Sigma/K,{\TT}_\frak{s})$ for every $(r+2)$-tuple $\frak{s}$. Fix $\frak{s}$ and $\eta\in \NN_{\frak{s}}$ until the end of this proof. We content to prove that $\kappa_\eta(\frak{s})  \in H^1_{\FF_+(\eta)}(K,\TT_{\frak{s}}),$ where $\FF_+(\eta)$ is the Selmer structure defined as in~\cite[Example 2.1.8]{mr02}. However, the proof of Theorem 5.3.3 of loc. cit. shows already that $\kappa_\eta(\frak{s})  \in  H^1_{\FFc(\eta)}(K,\TT_{\frak{s}}).$
 Since $\FF_+$ and $\FFc$ determine the same local conditions away from the primes above $p$, it suffices to show that
\be
\label{eqn:respontheline}
\textup{res}_p (\kappa_\eta(\frak{s}) ) \in H^1_{++}(K_p,\TT_{\frak{s}})\,.
\ee
We recall that $H^1_{++}(K(\eta)_p,\TT_{\frak{s}})$ is by definition the image of $H^1_{++}(K(\eta)_p,\TT_\cyc)$, which in turn is isomorphic to $H^1_{++}(K(\eta)_p,\TT_\cyc)/\mathcal{I}_\frak{s} H^1_{++}(K(\eta)_p,\TT_\cyc)$. In particular, it is a free $\mathcal{R}_\cyc/\mathcal{I}_\frak{s}[\Delta_\eta]$-module of rank $1+d_+$ by Proposition~\ref{lem:MRunramifiedisallunrforcyclo}. We will verify the truth of (\ref{eqn:respontheline}) assuming the truth of (H.++); the verification in general (where the submodule $H^1_{++}(K(\eta)_p,\TT_\cyc)$ is defined via Definition~\ref{def:auxiliarylineo}) is very similar.
Let $$\left\{\kappa_{[K,\eta,\frak{s}]} \in H^1(K_{\Sigma}/K,\TT_{\frak{s}}) \right\}_ {\eta \in \NN_{\frak{s}}}$$
 be the collection of Kolyvagin's derived classes (given as in \cite[Definition 4.4.10]{r00}) associated to the locally restricted Euler system $\mathbf{c}$. 
 
It follows from Equation (33) in \cite[Appendix A]{mr02} (which explains how to relate the class $\kappa_\eta(\frak{s}) $ to $\kappa_{[K,\eta,\frak{s}]}$) that (\ref{eqn:respontheline}) follows once we verify that
\be\label{eqn:respontheline2}
\textup{res}_p (\kappa_{[K,\eta,\frak{s}]}) \in H^1_{++}(K_p,\TT_{\frak{s}})\,.
\ee
We remark that what we denote by $\kappa_\eta(\frak{s}) $ here corresponds to $\kappa_n^\prime$ and $\kappa_{[K,\eta,\frak{s}]}$ to $\kappa_n$ in loc. cit. Let $D_{\eta}$ denote the derivative operator of Kolyvagin, as in~\cite[
Definition 4.4.1]{r00}. The class $\kappa_{[K,\eta,\frak{s}]}$ is defined as the inverse image of
$D_{\eta}c_{K(\eta)}$ (mod~$\mathcal{I}_\frak{s}$) under the
restriction map (which is an isomorphism since we assumed {\rm (H.3)})
$$H^1(K_\Sigma/K,\TT_{\frak{s}}) \lra H^1(K_{\Sigma}/K(\eta),\TT_{\frak{s}})^{\Delta_\eta}.$$
Thence, $\textup{res}_p(\kappa_{[K,\eta,\frak{s}]})$ maps to $\textup{res}_p\left(D_{\eta}c_{K(\eta)}\right)$ (mod~$\mathcal{I}_{\frak{s}}$) under the map (which is also an isomorphism thanks to (H.2))
$$H^1(K_p,\TT_{\frak{s}}) \stackrel{\sim}{\lra} H^1(K(\eta)_p,\TT_{\frak{s}})^{\Delta_{\eta}}\,.$$
Under this map, $H^1_{++}(K_p,\TT_{\frak{s}}) \subset H^1(K_p,\TT_{\frak{s}})$ is mapped isomorphically onto the module $H^1_{++}(K(\eta)_p,\TT_{\frak{s}})^{\Delta_{\eta}}$. This follows from the following commutative diagram,
$$\xymatrix@C=.65cm{
H^1_{++}(K_p,\TT)/\mathcal{I}_{\frak{s}} H^1_{++}(K_p,\TT_\cyc)\ar[d]_{\sim} \ar[r]^(.65){\sim}& H^1_{++}(K_p,\TT_{\frak{s}}) \ar@{^{(}->}[r]\ar[d]_{\sim}& H^1(K_p,\TT_{\frak{s}})\ar[d]_{\sim}\\
\left(H^1_{++}(K(\eta)_p,\TT_\cyc)/\mathcal{I}_{\frak{s}} H^1_{++}(K(\eta)_p,\TT_\cyc)\right)^{\Delta_\eta} \ar[r]^(.67){\sim} &H^1_{++}(K(\eta)_p,\TT_{\frak{s}})^{\Delta_\eta}\ar@{^{(}->}[r]&H^1(K(\eta)_p,\TT_{\frak{s}})^{\Delta_\eta}\
}$$
where the vertical isomorphism on the left follows from Proposition~\ref{prop:localcohom}, and the isomorphism in the middle thanks to the one on the left.

Since $\textup{res}_p$ is a $\Delta_\eta$-equivariant map, $\textup{res}_p(D_{\eta}c_{K(\eta)})=D_{\eta}\textup{res}_p(c_{K(\eta)}).$ Furthermore, since $\mathbf{c} \in \textup{ES}^+(\TT_\cyc)$ is locally restricted, we have $\textup{res}_p\left(c_{K(\eta)}\right) \in H^1_{++}(K(\eta)_p,\TT_\cyc)$. On the other hand we know by \cite[Lemma 4.4.2]{r00} that the derived class $D_{\eta}c_{K(\eta)}$ (mod~$\mathcal{I}_{\frak{s}}$) is fixed by $\Delta_{\eta}$, which in turn implies that
$$\textup{res}_p\left(c_{K(\eta)}\right) \,(\textup{mod\,} \mathcal{I}_{\frak{s}}) \in\left(H^1_{++}(K(\eta)_p,\TT_\cyc)/\mathcal{I}_{\frak{s}} H^1_{++}(K(\eta)_p,\TT_\cyc)\right)^{\Delta_\eta} .$$
This concludes the proof of the containment (\ref{eqn:respontheline2}) and also the proof of the theorem.\end{proof}

For some of our main applications, we will utilize the Beilinson--Flach (locally restricted) Euler system of \cite{LLZ, KLZ1, KLZ2}). Even in cases where we do not have an Euler system at our disposal one may still prove the following result, which shows that the method developed in Section~\ref{sec:dimreduction} is still non-vacuous for a wide variety of cases of interest.

We consider the following condition at primes in $\Sigma^{(p)}$: 
\begin{enumerate}
\item[(H.Tam)] For each non-archimedean place $\lambda \in \Sigma$ that does not lie above $p$, we have 
$$H^0(K_\lambda,\overline{T})=H^2(K_\lambda,\overline{T})=0.$$ 
\end{enumerate}
\begin{rem}
If ({H.Tam.}) holds true, it follows from the local Euler characteristic formulae and Nakayama's Lemma that  we have $H^1(K_\lambda,X)=0$ for every $\lambda\in \Sigma^{(p)}$ and all quotients $X$ of $\TT_\cyc$.
\end{rem}
\begin{thm}
\label{thm:theexistenceofKS}
Suppose that the hypotheses {\rm (MR1) - (MR.4)}, {\rm{\htwo}}\, hold true for a free $\mathcal{R}$-module $T$ with 
continuous $G_K$-action. 
We also assume that one of the following conditions is valid. 
\begin{itemize}
\item[(i)] The ring $\mathcal{R}$ is a discrete valuation ring and $H^0(K_\lambda^{\textup{ur}},T\otimes\QQ_p/\ZZ_p)$ is $p$-divisible for every $\lambda \in \Sigma^{(p)}$\,.
\item[(ii)] The ring $\mathcal{R}$ is isomorphic to a power series ring in $n$ variables $\LL_{\mathcal{O}}^{(n)}$ and  {\rm({H.Tam})} holds.
\end{itemize}
Then the $\mathcal{R}_\cyc$-module $\KS(\TT_\cyc,\FF_+)$ is free of rank one.
\end{thm}
For the proof in the situation of (i), see \cite[Theorem A]{kbb}. For the proof in the situation of (ii), see \cite[Theorem A]{kbbdeform}. See also \cite{kbbCMabvar, kbblei1} for its various incarnations. 
We remark again that the results in loc.cit. are stated only for the Selmer structure $\FFc$ and for the base field $\QQ$, but their generalization to our set up is entirely formal.


\section{Dimension reduction and locally restricted Euler systems}
\label{sec:dimreduction}
 We retain our convention from Section~\ref{sec:selmerstructures} concerning our use of the symbols $\bfT$ and $R$. Throughout Section~\ref{sec:dimreduction}, we shall assume that $R$ is isomorphic to a power series ring in $n$ variables $\LL_{\mathcal{O}}^{(n)}$. 
\subsection{Generalities}
\label{subsec:dimreducegeneralities}
In this subsection, we shall recall basic results from \cite{ochiai-AIF} that will be instrumental in our dimension reduction argument. 

\begin{define}
\label{def:charideal}
Let $n\geq 1$ be an integer. 
\begin{enumerate}
\item 
{\it A linear element} $l$ in an $n$-variable Iwasawa algebra 
$\Lambda^{(n)}_{\mathcal{O}} = \mathcal{O}[[x_1 ,\ldots ,x_n]] $ 
is a polynomial $l =a_0 +a_1 x_1 + \cdots a_n x_n 
\in \Lambda^{(n)}_{\mathcal{O}}$ with $a_i \in \mathcal{O}$  
of degree at most one such that 
$l$ is not divisible by $\pi_{\mathcal{O}}$ and is not invertible 
in $\Lambda^{(n)}_{\mathcal{O}}$. That is to say, $l$ is a polynomial of degree 
at most one such that $a_0$ is divisible by $\pi_{\mathcal{O}}$, 
but not all $a_i$ are divisible by $\pi_{\mathcal{O}}$. 
\item 
An ideal of $\Lambda^{(n)}$ that is generated by a linear element is called a \emph{linear ideal}. We denote by 
$$
\mathcal{L}^{(n)}_{\mathcal{O}} = 
\left\{ (l) \subset \Lambda^{(n)}_{\mathcal{O}} \ \left\vert \ 
\text{$l$ is a linear element in $\Lambda^{(n)}_{\mathcal{O}}$} 
\right. \right\} . 
$$
the set of all 
linear ideals of $\Lambda^{(n)}_{\mathcal{O}}$. 
\item 
Let $n\geq 2$. For a torsion $\Lambda^{(n)}_{\mathcal{O}}$-module $M$, 
we denote by $\mathcal{L}^{(n)}_{\mathcal{O}} (M)$ a subset of 
$\mathcal{L}^{(n)}_{\mathcal{O}}$ which consists of $(l) \subset 
\mathcal{L}^{(n)}_{\mathcal{O}}$ satisfying the following conditions: 
\begin{enumerate} 
\item 
The quotient $M/(l)M$ is a torsion 
$\Lambda^{(n)}_{\mathcal{O}}/(l)$-module. 
\item 
The image of the characteristic ideal 
$\mathrm{char}_{\Lambda^{(n)}_{\mathcal{O}}}(M) \subset 
\Lambda^{(n)}_{\mathcal{O}}$ in $\Lambda^{(n)}_{\mathcal{O}}/(l)$ 
is equal to the characteristic ideal 
$\mathrm{char}_{\Lambda^{(n)}_{\mathcal{O}}/(l)}(M/(l)M) \subset 
\Lambda^{(n)}_{\mathcal{O}}/(l)$. 
\end{enumerate} 
\end{enumerate}
\end{define}
We have the following proposition which characterizes 
the characteristic ideal of a given torsion 
$\Lambda^{(n)}_{\mathcal{O}}$-module for $n\geq 2$: 
\begin{prop}\label{pro:rednn-1}
Let $n \geq 2$ be an integer and let $M$ and $N$ be a 
finitely generated torsion $\Lambda^{(n)}_{\mathcal{O}}$-modules. 
Then the following three statements are equivalent. 
\begin{enumerate}
\item 
We have $\mathrm{char}_{\Lambda^{(n)}_{\mathcal{O}}}(M) 
\supset \mathrm{char}_{\Lambda^{(n)}_{\mathcal{O}}}(N)$.  
\item 
Let $\mathcal{O}'$ be arbitrary complete discrete valuation ring  
which is finite flat over $\mathcal{O}$. 
Then, for all but finitely many 
$(l) \in \mathcal{L}^{(n)}_{\mathcal{O}'}(M_{\mathcal{O}'}) 
\cap \mathcal{L}^{(n)}_{\mathcal{O}'} (N_{\mathcal{O}'})$, we have 
the inclusion 
$$\mathrm{char}_{\Lambda^{(n)}_{\mathcal{O}'}/(l)}
(M_{\mathcal{O}'}/(l)M_{\mathcal{O}'}) \supset 
\mathrm{char}_{\Lambda^{(n)}_{\mathcal{O}'}/(l)}
(N_{\mathcal{O}'}/(l)N_{\mathcal{O}'}).$$ 
\item 
There exists a complete discrete valuation ring $\mathcal{O}'$ 
which is finite flat over $\mathcal{O}$ such that 
we have the inclusion 
$$\mathrm{char}_{\Lambda^{(n)}_{\mathcal{O}'}/(l)}
(M_{\mathcal{O}'}/(l)M_{\mathcal{O}'}) \supset 
\mathrm{char}_{\Lambda^{(n)}_{\mathcal{O}'}/(l)}
(N_{\mathcal{O}'}/(l)N_{\mathcal{O}'})$$  
for all but finitely many 
$(l) \in \mathcal{L}^{(n)}_{\mathcal{O}'}(M_{\mathcal{O}'}) 
\cap \mathcal{L}^{(n)}_{\mathcal{O}'} (N_{\mathcal{O}'})$. 
\end{enumerate}
\end{prop}

A set of ideals 
$\mathcal{E}_{\mathcal{O}}=\{ I_m \subset \Lambda_{\mathcal{O}} \ \vert \ m\in \ZZ_{\geq 1} \}$ is called Eisenstein type if $I_m =  (E_m(x))$ where $E_m(x)$ is an Eisenstein polynomial of degree $m\geq 1$ in $\mathcal{O}[x]$. The following proposition provides the initial step of the induction step. 

\begin{prop}\label{pro:char0}
Let $M$ and $N$ be finitely generated torsion $\mathcal{O} [[x]]$-modules. 
\begin{enumerate}
\renewcommand{\labelenumi}{(\theenumi)}
\item 
The following conditions are equivalent: 
\begin{enumerate}
\item 
There exists an integer $h\geq 0$ such that $\mathrm{char}_{\Lambda_{\mathcal{O}}}(M) \supset (\pi^h_{\mathcal{O}}) \mathrm{char}_{\Lambda_{\mathcal{O}}}(N)$. 
\item Let $\mathcal{O}'$ be arbitrary complete discrete valuation ring which is finite flat over $\mathcal{O}$. Then there exists a positive integer $c$ depending only on $M_{\mathcal{O}'}$ and $N_{\mathcal{O}'}$ such that 
$${\# (M_{\mathcal{O}'}/IM_{\mathcal{O}'})} \leq c  {\# (N_{\mathcal{O}'}/IN_{\mathcal{O}'})}$$ 
for all but finitely many $I \in \mathcal{L}_{\mathcal{O}'}$. 
\end{enumerate}
\item 
Concerning the error term $(\pi^h_{\mathcal{O}})$, the following two statements are equivalent: 
\begin{enumerate}
\item 
Let $M_{(\pi_{\mathcal{O}})}$ $($resp. $N_{(\pi_{\mathcal{O}})})$ be the localization of $M$ $($resp. $N)$ at the prime ideal $(\pi_{\mathcal{O}})$. Then we have $\mathrm{length}_{(\Lambda_{\mathcal{O}})_{(\pi_{\mathcal{O}})}}(M_{(\pi_{\mathcal{O}})}) \leq 
\mathrm{length}_{(\Lambda_{\mathcal{O}})_{(\pi_{\mathcal{O}})}}(N_{(\pi_{\mathcal{O}})}) $. 
\item 
There exist a set of ideals $\mathcal{E}_{\mathcal{O}} = \{ I_m \ \vert \ m\in \ZZ_{\geq 1} \}$ of Eisenstein type and a positive integer $c$ that depends only on $M$ and $N$ such that $${\# (M /I_m M)} \leq c  {\# (N /I_m N )}$$ for all but finitely many $I_m$. 
\end{enumerate}
\end{enumerate}
\end{prop}
For the proof of Proposition \ref{pro:rednn-1} (resp. Proposition \ref{pro:char0}), we refer the reader to \cite[Prop. 3.6]{ochiai-AIF} 
(resp. \cite[Prop. 3.11]{ochiai-AIF}).

\subsection{The Euler system bound}
\label{sub:ESboundmain}
Throughout Section~\ref{sub:ESboundmain}, we assume the hypotheses {\rm (MR1) - (MR4)} of Mazur and Rubin (that we have recalled in Section~\ref{subsub:Kolyvaginprimesexist}), \hzero, (H.$0^-$), \htwo\, and \htwoplus.  Recall also our convention from Section~\ref{sec:selmerstructures} that $\bfT$ stands either for $\TT$ or $\TT_\cyc$, and correspondingly, $R$ is either $\mathcal{R}$ or $\mathcal{R}_\cyc$.

\begin{define}\label{definition:tamagawanumber}
Let $\mathcal{R}$ be a discrete valuation ring. 
\begin{enumerate}
\item[(1)]
We define the Selmer structure $\FF_{\Sigma+}$ on $\TT$ determined as follows 
$$
H^1_{\FF_{\Sigma+}}(K_v,\TT):=
\begin{cases}
H^1_{\FF_{+}}(K_v,\TT) & \text{ if $v$ is above $p$,} \\
\ker\left(H^1(K_v,\TT)\ra H^1(K_v^{\textup{ur}},\TT\otimes\QQ_p)\right) 
& \text{ if $v\in \Sigma^{(p)}$.} 
\end{cases}
$$
\item[(2)]
We set $\mathbb{V}:=\mathbb{T}\otimes\QQ_p$ and define 
$${\displaystyle\textup{Tam}_{\Sigma^{(p)}}(\mathbb{V}):=\bigoplus_{\lambda\in \Sigma^{(p)}}\frac{H^1_{\FF_{\Sigma+}}(K_\lambda,\TT)}{H^1_{\FF_\Gr}(K_\lambda,\TT)}}$$
whose order is precisely the $p$-part of the Tamagawa factors at $\lambda$. Notice that since we have assumed ${\rm (MR1)}$, 
$G_K$-stable lattices of $\mathbb{V}$ are all isomorphic to $\mathbb{T}$. Thus, our notation for Tamagawa numbers here is justified.
\end{enumerate}
\end{define}
The second portion of the following theorem is proved in \cite{mr02,kbbstark,kbbESrankr} and its remaining parts (1) and (3) follows from (2) by Poitou-Tate global duality. It constitutes the base case for our main result in this section (Theorem~\ref{thm:weakleopoldt}).
\begin{thm}
\label{thm:weakleopoldt0}
Suppose that $\mathcal{R}$ is a discrete valuation ring and $\mathbf{c}\in \textup{ES}^+(\TT_\cyc)$ is a locally restricted Euler system. Suppose that (the image of) its initial term $c_K\in H^1(K_\Sigma/K,\TT)$ is non-vanishing. Then, the following statements hold under the running hypotheses 
of \S \ref{sub:ESboundmain}.
\begin{itemize}
\item[(1)] The $\mathcal{R}$-module $H^1_{\FF_+^*}(K,\TT^\vee (1))^\vee$ is torsion and $H^1_{\FF_{+}}(K,\TT)$ is free of rank one.
\item[(2)] $\textup{Fitt}\left(H^1_{\FF_{\Sigma+}^*}(K,\TT^\vee (1))^\vee\right)\,\, \supset \,\, \textup{Fitt}\left(H^1_{\FF_{\Sigma+}}(K,\TT)\big{/}R c_K\right)$.\\
\item[(3)] $\textup{Fitt}\left(H^1_{\FF_+^*}(K,\TT^\vee (1))^\vee\right)\,\, \supset \,\, \textup{Fitt}\left(H^1_{\FF_+}(K,\TT)\big{/}R c_K\right){\textup{Fitt}\left(\textup{Tam}_{\Sigma^{(p)}}(\TT)\right)}$.
\end{itemize}
In (2) and (3), $\textup{Fitt}$ stands for the initial Fitting ideal.
\end{thm}
\begin{proof}
We first explain why (2) holds true, which is readily proved in \cite{kbbESrankr} and that refines the results in \cite{mr02,kbbstark}. All references in this paragraph are to \cite{mr02} unless otherwise is stated. Under the Euler systems to Kolyvagin systems map (Theorem 3.2.4), the locally restricted Euler system with initial term $ c_K\neq 0$ gives rise to a Kolyvagin system for the Selmer structure $\FF_{\Sigma+}$, {whose initial term is the class $c_K$ and in particular non-zero}. The only non-trivial input here (in addition to Theorem 3.2.4) is that the derived classes verify the appropriate local conditions at the primes above $p$. This follows from the proof of Theorem~\ref{thm:ESKSlocallyrestricted} above.  Furthermore, the Selmer structure $\FF_{\Sigma+}$ verifies the hypotheses of Section 5.2 and it is easy to see that its core Selmer rank $\chi(T,\FF_{\Sigma+})$ (in the sense of Definition 4.1.11) equals to one. {Under our running assumptions, Theorem 5.2.2 shows that $H^1_{\FF_{\Sigma+}^*}(K,\TT^\vee(1))$ has finite cardinality and Corollary 5.2.6 shows $H^1_{\FF_{\Sigma+}}(K,\TT)$ is of rank one. }{ Under our running assumptions, Corollary~5.2.13(ii) applies (used with the choices indicated above) and it is a restatement of the containment in (2), as $H^1_{\FF_{\Sigma+}}(K,\TT)$ is a free $\mathcal{R}$-module of rank one. } 
\par 
We will next prove (3) and deduce (1) as a consequence. Consider the Poitou-Tate global duality sequence
\begin{align*}
0\ra H^1_{\FF_{+}}(K,\TT)\lra H^1_{\FF_{\Sigma+}}(K,\TT)\lra&\textup{Tam}_{\Sigma^{(p)}}(\TT)\\
&\lra H^1_{\FF_{+}^*}(K,\TT^\vee (1))^\vee\lra H^1_{\FF_{\Sigma+}^*}(K,\TT^\vee (1))^\vee\ra0
\end{align*}
We therefore conclude that
$$\frac{\textup{Fitt}\left(H^1_{\FF_{+}}(K,\TT)/R c_K\right){\textup{Fitt}\left(\textup{Tam}_{\Sigma^{(p)}}(\TT)\right)}}{\textup{Fitt}\left(H^1_{\FF_{+}^*}(K,\TT^\vee (1))^\vee\right)} =\frac{\textup{Fitt}\left(H^1_{\FF_{\Sigma+}}(K,\TT)/R c_K \right)}{\textup{Fitt}\left(H^1_{\FF_{\Sigma+}^*}(K,\TT^\vee (1))^\vee \right)}$$
and (3) follows from (2). Moreover, the $\mathcal{R}$-module $\textup{Tam}_{\Sigma^{(p)}}(\TT)$ is torsion as explained in \cite[Lemma I.3.5]{r00}. As $c_K\in H^1_{\FF_{+}}(K,\TT)$ is non-zero and since the $R$-module $H^1_{\FF_{\Sigma+}}(K,\TT)\supset H^1_{\FF_{+}}(K,\TT)$ has rank one by the discussion in the first paragraph of this proof, (1) follows as a consequence of (3). 

\end{proof}

Recall our running assumption (besides those we have recorded at the start of Section~\ref{sub:ESboundmain}) that $R $ is isomorphic to the power series ring $\LL_{\mathcal{O}}^{(n)}$ in $n$ variables. 

\begin{thm}
\label{thm:weakleopoldt}
Suppose that the Krull dimension of $R$ is at least two and $\mathbf{c}\in \textup{ES}^+(\TT_\cyc)$ is a locally restricted Euler system. Suppose that (the image of) its initial term $c_K\in H^1(K_\Sigma/K,\bfT)$ is non-vanishing. Then, the following statements hold under the running hypotheses 
of \S \ref{sub:ESboundmain}.
\begin{itemize}
\item[(1)] $H^1_{\FF_+^*}(K,\bfT^\vee (1))^\vee$ is $R$-torsion. 
\item[(2)] When $\bfT=\TT_\cyc$, the $R$-module $H^1_{\FF_{+}}(K,\bfT)$ is free of rank one over $R$. 
For general $\bfT$, the $R$-module $H^1_{\FF_{+}}(K,\bfT)$ is torsion-free of generic rank one over $R$. 
\item[(3)] $\textup{char}\left(H^1_{\FF_+^*}(K,\bfT^\vee (1))^\vee\right)\,\, \supset \,\, \textup{char}\left(H^1_{\FF_+}(K,\bfT)\big{/}R c_K\right).$
\end{itemize}
\end{thm}
\begin{rem}
\label{rem:cyclotomicdeformationforcesunr}
Since we require our Euler  system $\mathbf{c}$ to extend in the cyclotomic direction, it follows by a standard argument (see \cite[Proposition B.3.4]{r00} for details) that $\textup{res}_{\frak{q}}(c_F)\in H^1_{\textup{ur}}(F_\frak{q},\bfT)$ is unramified at every place $\frak{q}$ of $F$ which lies above a prime of $\Sigma^{(p)}$.
\end{rem}

The proof of this theorem is rather involved and will occupy the entire Section~\ref{subsec:proofofmainESbound}. Our argument relies on Lemmas \ref{lem:cartesianfordimension1}, \ref{lem:localcontrol}, \ref{lem:globalcontrol}, \ref{lem:localcontrolhigherdim} and \ref{lem:globalcontrolhigherdim} below, which are all stated and proved within Section~\ref{subsec:proofofmainESbound}. 
\subsubsection{Proof of Theorem~\ref{thm:weakleopoldt}}
\label{subsec:proofofmainESbound}
We will proceed by induction on the Krull dimension of $R$ (that equals $r+1$ or $r+2$, depending on our choice of coefficient rings among $\mathcal{R}$ and $\mathcal{R}_\cyc$) after proving the case $r=1$ separately. Notice that the base case $r=0$ is already settled by Theorem~\ref{thm:weakleopoldt0} above. \\
\\
\textbf{Case} $\mathbf{r=1}$. Our proof will closely follow the arguments in \cite[Section 5.3]{mr02}. However, we remark that our coefficient ring $R$ will not necessarily carry a Galois action (unlike the coefficient ring $\LL$ considered in op. cit.). Notice that when $R=\mathcal{R}_\cyc$ is the cyclotomic Iwasawa algebra, then once we invoke Lemma~\ref{lem:cartesianfordimension1} below, we may prove Theorem~\ref{thm:weakleopoldt} (1) in a manner identical to \cite[Theorem 5.3.6]{mr02} and Theorem~\ref{thm:weakleopoldt} (3) to \cite[Theorem 5.3.10(i)]{mr02}. Before going into the proof, we set some notation.

We define the modules
\begin{align*} & \textup{Loc}^+_p (\bfT):=\displaystyle{\bigoplus_{v\mid p} }\dfrac{H^1(K_v,\bfT)}{H^1_{\FF_+}(K_v,\bfT)} \\ 
& \textup{Loc}_\Sigma^+ (\bfT):=\bigoplus_{\lambda \in \Sigma}
\dfrac{H^1(K_\lambda,\bfT)}{H^1_{\FF_+}(K_\lambda,\bfT)}=\textup{Loc}^+_p (\bfT)\oplus \bigoplus_{\lambda\in \Sigma^{(p)}}H^1(I_\lambda,\bfT)^{\textup{Fr}_{\lambda}=1}.
\end{align*}
\begin{rem}\label{rem:cyclomeansLocSigmaistorsionfree}
Since $\textup{Loc}^+_p (\bfT)$ is a finitely generated torsion-free $R$-module thanks to our running hypothesis (H.$0^-$) assumed at the start of Section~\ref{sub:ESboundmain}, we have 
\begin{equation}\label{equation:torsion_part_Loc}
\textup{Loc}_\Sigma^+ ( \bfT)_{\textup{R-tor}}=\left( \bigoplus_{\lambda\in \Sigma^{(p)}}H^1(I_\lambda,\bfT)^{\textup{Fr}_{\lambda}=1}\right)_{\textup{R-tor}}.
\end{equation}
When $\bfT=\TT_\cyc$, we have $H^1(I_v,\TT_\cyc)=0$ and hence, $\textup{Loc}_\Sigma^+ (\bfT)=\textup{Loc}_p^+ (\bfT)$ is a finitely generated torsion-free 
$R$-module again due to the hypothesis (H.$0^-$).
\end{rem}
Let $\mathcal{S}_R$ be the set of height-one primes of $R$. 
We further define the \emph{exceptional set of primes} $\mathcal{E}_R$ for $\bfT$ to be the subset 
of $\mathcal{S}_R$ by setting
\begin{multline}
\mathcal{E}_R=\{\pi_\oo R\} \cup
\{\frak{P}\in \mathcal{S}_R : H^2(K_\Sigma/K,\bfT)[\frak{P}] \hbox{ is infinite}\} 
\\ 
\cup\{\frak{P}\in \mathcal{S}_R:\oplus_{\lambda_\in \Sigma}H^2(K_\lambda,\bfT)[\frak{P}] \hbox{ is } \hbox{ infinite}\}
\cup \{\frak{P}\in \mathcal{S}_R:\oplus_{\lambda_\in \Sigma^{(p)}}H^1(I_\lambda,\bfT)[\frak{P}] \hbox{ is } \hbox{ infinite}\}
\\ 
\cup \{\frak{P}\in \mathcal{S}_R: \textup{Tam}_{\Sigma^{(p)}}(\bfT_{\frak{P}}/\frak{P}\bfT_{\frak{P}})\neq 0\}\,.
\end{multline}
Here, $\pi_\oo \in \oo$ is a uniformizing element and $\bfT_{\frak{P}}$ denotes the localization of $\bfT$ at $\frak{P}$ 
for a height-one prime $\frak{P}$ of $R$. We remark that the definition of $\textup{Tam}_{\Sigma^{(p)}}(\bfT_{\frak{P}}/\frak{P}\bfT_{\frak{P}})$ makes sense 
with Definition \ref{definition:tamagawanumber} since $\bfT_{\frak{P}}/\frak{P}\bfT_{\frak{P}}$ is a finite dimensional
$R_{\frak{P}}/\frak{P}$-vector space and $R_{\frak{P}}/\frak{P}$ is a finite extension of $\mathbb{Q}_p$. 
\begin{define}
When $r=1$ and $\frak{P}\in \mathcal{S}_R \setminus \{ \pi_{\mathcal{O}}R\} $, we define the \emph{degree} of $\frak{P}$ to be 
the quantity $\textup{rank}_{\mathcal{O}} R/\frak{P}$.
\end{define}
\begin{lemma}
\label{lemma:badplacesavoidingtamagawadefectisfinite}
The set $\mathcal{E}_R$ has finite cardinality. 
\end{lemma}
\begin{proof}
Using the fact that cohomology groups $H^2(K_\Sigma/K,\bfT)$, $H^2(K_\lambda,\bfT)$ and $H^1(I_\lambda,\bfT)$ are 
finitely generated over $R$ and the fact that $R$ is of Krull dimension $2$ by assumption, we see that 
$H^2(K_\Sigma/K,\bfT)[\frak{P}]$, $H^2(K_\lambda,\bfT)[\frak{P}]$ and $H^1(I_\lambda,\bfT)[\frak{P}]$ are of finite cardinality 
for all but finitely many $\frak{P}\in \mathcal{S}_R$. As for the 
last set $\{\frak{P}\in \mathcal{S}_R: \textup{Tam}_{\Sigma^{(p)}}(\bfT_{\frak{P}}/\frak{P}\bfT_{\frak{P}})\neq 0\}$. 
The finiteness of this set is precisely the content of \cite[7.6.10.10]{nek}. This completes the proof of the lemma. 
\end{proof}
We remind the readers that we continue to work under the assumption that $r=1$ until we announce otherwise (on page~\pageref{pageref:caser>1}).
\begin{lemma}
\label{lem:cartesianfordimension1}
\begin{enumerate}
\item[(1)]
For every height one prime $\frak{P}$ of $R$, the map $\bfT\twoheadrightarrow \bfT/\frak{P}\bfT$ induces a map 
$$\pr_\frak{P}:\,\dfrac{H^1_{\FF_+}(K,\bfT)}{\frak{P}H^1_{\FF_+}(K,\bfT)}\lra H^1_{\FF_+}(K,\bfT/\frak{P}\bfT).
$$
If we assume that $\frak{P}\notin \mathcal{E}_R$, 
the kernel and the cokernel of $\pr_\frak{P}$ is finite and bounded independently of $\frak{P}$. 
When $\bfT=\TT_\cyc$, the map $\pr_\frak{P}$ is injective.
\item[(2)]
For every height one prime $\frak{P}\notin \mathcal{E}_R$, the map $\bfT\twoheadrightarrow \bfT/\frak{P}\bfT$ induces 
an isomorphism 
$$\pr_\frak{P}^*:\,H^1_{\FF_+^*}(K,\bfT^\vee(1)[\frak{P}])\stackrel{\sim}{\lra} H^1_{\FF_+^*}(K,\bfT^\vee(1))[\frak{P}]\,.$$
\end{enumerate}
\end{lemma}
\begin{proof}
For a height one prime $\frak{P}$ of $R$, we consider the exact sequence 
\be\label{eqn:descentmodfrakP}0\lra \bfT {\lra}\bfT\lra \bfT/\frak{P}\bfT\lra0\, , 
\ee
where the first map is multiplication by a generator of $\frak{P}$.  
The $G_{K,\Sigma}$-cohomology of this sequence yields an {injective map 
\be\label{eqn:globaldescentraw}
\dfrac{H^1(K_\Sigma/K,\bfT)}{\frak{P}H^1(K_\Sigma/K,\bfT)} \hookrightarrow H^1(K_\Sigma/K,\bfT/\frak{P}\bfT)
\ee
with cokernel isomorphic to $H^2(K_\Sigma/K,\bfT)[\frak{P}]$. When $\frak{P}\not\in \mathcal{E}_R$, note that $H^2(K_\Sigma/K,\bfT)[\frak{P}]$ is contained in the maximal finite submodule of $H^2(K_\Sigma/K,\bfT)$. Therefore, the size of the cokernel of the map (\ref{eqn:globaldescentraw}) is finite and 
bounded by a constant that depends only on $\bfT$ as $\frak{P}$ varies away from $\mathcal{E}_R$.}

Consider the following diagram with exact rows and cartesian squares,
\be\label{eqn:nonperfectdescent}\xymatrix{\textup{Tor}_1^R(R/\frak{P},M_\Sigma^+)\ar[r]&{\displaystyle\frac{ H^1_{\FF_+}(K,\bfT)}{\frak{P}H^1_{\FF_+}(K,\bfT)}}\ar[r]\ar@{-->}[d]&{\displaystyle\frac{ H^1(K_\Sigma/K,\bfT)}{\frak{P}H^1(K_\Sigma/K,\bfT)}}\ar[r]^(.6){\res_\Sigma}\ar@{^{(}->}[d]&{\displaystyle \frac{\textup{Loc}_{\Sigma}^+ (\bfT)}{\frak{P}\cdot\textup{Loc}_{\Sigma}^+ (\bfT)}}\ar[d]^{\nu}\\
0\ar[r]&{\displaystyle{ H^1_{\FF_+}(K,\bfT/\frak{P}\bfT)}}\ar[r]&H^1(K_\Sigma/K,\bfT/\frak{P}\bfT)\ar[r]&{\textup{Loc}_{\Sigma}^+ 
(\bfT /\frak{P}\bfT)}
}
\ee
where $M_\Sigma^+\subset \textup{Loc}^+_\Sigma (\bfT)$ is the image of $\res_\Sigma: H^1(K_{\Sigma}/K,\bfT)\ra \textup{Loc}^+_\Sigma (\bfT)$ and the dotted arrow (which is our map $\pr_{\frak{P}}$) is induced from the rest of the diagram.

In order to verify our claims concerning the kernel of $\pr_{\frak{P}}$, note that we have 
$$\textup{Tor}_1^R(R/\frak{P},M_\Sigma^+)\hookrightarrow \left(\bigoplus_{\lambda \in \Sigma^{(p)}}H^1(I_\lambda ,\bfT)^{\textup{Fr}_{\lambda}=1}\right)[\frak{P}]$$ by \eqref{equation:torsion_part_Loc}. As the $R$-module $\left(\bigoplus_{\lambda \in \Sigma^{(p)}}H^1(I_\lambda ,\bfT)^{\textup{Fr}_{\lambda}=1}\right)$ 
is finitely generated, its maximal $\frak{P}$-torsion submodule has finite order bounded independently as $\frak{P}$ varies away from $\mathcal{E}_R$. The assertion in the case $\bfT=\TT_\cyc$ also follows since 
we have $H^1(I_\lambda,\TT_\cyc)=0$. 
\par 
In order to control the cokernel of $\pr_{\frak{P}}$, it suffices to prove that the kernel of the map
$$
\nu:\,\dfrac{\textup{Loc}_{\Sigma}^+(\bfT)}{\frak{P}\cdot\textup{Loc}_{\Sigma}^+ (\bfT)}\lra \textup{Loc}_{\Sigma}^+(\bfT /\frak{P}\bfT)$$
is bounded by a constant which is independent of $\frak{P}$ (as the same assertion readily holds for the cokernel of the map in the middle, which is the map (\ref{eqn:globaldescentraw}) above). Consider now the following commutative diagram with exact rows.
$$
\xymatrix{&{\displaystyle \bigoplus_{v\mid p} \frac{H^1_{\FF_+}(K_v,\bfT)}{\frak{P}H^1_{\FF_+}(K_v,\bfT)}}\ar[r]\ar@{->>}[d]&{\displaystyle \bigoplus_{v\mid p} \frac{H^1(K_v,\bfT)}{\frak{P}H^1(K_v,\bfT)}}\ar[r]\ar[d]^{\frak{h}}&{\displaystyle \frac{\textup{Loc}_{p}^+(\bfT)}
{\frak{P}\cdot\textup{Loc}_{p}^+(\bfT)}}\ar[r]\ar[d]^\nu&0\\
0\ar[r]&{\displaystyle \bigoplus_{v\mid p} H^1_{\FF_+}(K_v,\bfT/\frak{P}\bfT)}\ar[r]&{\displaystyle \bigoplus_{v\mid p} H^1(K_v,\bfT/\frak{P}\bfT)}\ar[r]&\textup{Loc}_{p}^+ (\bfT /\frak{P}\bfT)\ar[r]&0
}$$
Here, the surjectivity of the vertical arrow on the left follows by the definition of induced Selmer structures. By Snake Lemma, it remains to control the kernel of $\frak{h}$, which is evidently injective. The portion of the lemma concerning $\pr_\frak{P}$ now follows.

The second portion concerning the map $\pr_\frak{P}^*$ is \cite[Lemma 3.5.3]{mr02}, which applies since we assume (MR1) and (MR3).
\end{proof}
For $\frak{P}\notin \mathcal{E}_R$\,, we let $(R/\frak{P})^{\mathrm{int}}$ denote the integral closure of of $R/\frak{P}$ in its field of fractions $\textup{Frac}(R/\frak{P})$. We set $(\bfT/\frak{P}\bfT )^{\mathrm{int}}:=\bfT\otimes_R (R/\frak{P})^{\mathrm{int}}$ and since we assume {\rm (MR1)}, observe that it is the unique $G_K$-stable $(R/\frak{P})^{\mathrm{int}}$-lattice in $\bfT_{\frak{P}}/\frak{P}\bfT_{\frak{P}}$. Notice that 
$$H^1(K_p,(\bfT/\frak{P}\bfT )^{\mathrm{int}})\cong H^1(K_p,\bfT/\frak{P}\bfT)\otimes_{R/\frak{P}}(R/\frak{P})^{\mathrm{int}}$$ 
and it therefore follows from Proposition~\ref{prop:localcohom} (used together with the vanishing of $H^2(K_p,\bfT)$ thanks to our running hypotheses) and Corollary~\ref{cor:structuresemilocalcohom} that $H^1(K_p,(\bfT/\frak{P}\bfT )^{\mathrm{int}})$ is a free $(R/\frak{P})^{\mathrm{int}}$-module of rank $d\cdot[K:\QQ]$, spanned by the image of $H^1(K_p,\bfT/\frak{P}\bfT)$ under the tautological inclusion $\iota_\frak{P}: R/\frak{P}\hookrightarrow (R/\frak{P})^{\mathrm{int}}$. We let $H^1_{++}(K_p,(\bfT/\frak{P}\bfT )^{\mathrm{int}})\subset H^1(K_p,(\bfT/\frak{P}\bfT )^{\mathrm{int}})$ denote the submodule spanned by the image $H^1_{++}(K_p,\bfT/\frak{P}\bfT)$ under $\iota_{\frak{P}}$. Then the $(R/\frak{P})^{\mathrm{int}}$-module $H^1_{++}(K_p,(\bfT/\frak{P}\bfT )^{\mathrm{int}})$ is a direct summand of 
the $(R/\frak{P})^{\mathrm{int}}$-module $H^1(K_p,(\bfT/\frak{P}\bfT )^{\mathrm{int}})$ of rank $1+d_+$.

Define the Selmer structure $\FF_{+}^{\frak{P}}$ on $(\bfT/\frak{P}\bfT )^{\mathrm{int}}$ by setting 
\begin{align*}
H^1_{\FF_{+}^{\frak{P}}}(K_\lambda,(\bfT/\frak{P}\bfT )^{\mathrm{int}}):= \ker(H^1(K_\lambda,(\bfT/\frak{P}\bfT )^{\mathrm{int}}) \lra H^1(K_\lambda^\textup{ur},(\bfT/\frak{P}\bfT )^{\mathrm{int}}\otimes\QQ_p)) 
\end{align*}
for $\lambda \in \Sigma^{(p)}$ and requiring the local conditions
$$H^1_{\FF_{+}^{\frak{P}}}(K_p,(\bfT/\frak{P}\bfT )^{\mathrm{int}}):=H^1_{++}(K_p,(\bfT/\frak{P}\bfT )^{\mathrm{int}})$$
at $p$.
{\begin{lemma}
\label{lem:finiteisunrawayfromER} For each $\lambda\in \Sigma^{(p)}$ we have,
$$H^1_{\FF_{+}^{\frak{P}}}(K_\lambda,(\bfT/\frak{P}\bfT )^{\mathrm{int}})=\ker\left(H^1(K_\lambda,(\bfT/\frak{P}\bfT )^{\mathrm{int}})
\lra H^1(K_\lambda^\textup{ur},(\bfT/\frak{P}\bfT )^{\mathrm{int}})\right)$$
In other words, the local condition determined by ${\FF_{+}^{\frak{P}}}$ on $(\bfT/\frak{P}\bfT )^{\mathrm{int}}$ agrees with the unramified condition. 
\end{lemma}
\begin{proof}
The containment 
$$H^1_{\FF_{+}^{\frak{P}}}(K_\lambda,(\bfT/\frak{P}\bfT )^{\mathrm{int}})\supset \ker\left(H^1(K_\lambda,(\bfT/\frak{P}\bfT )^{\mathrm{int}})\lra H^1(K_\lambda^\textup{ur},(\bfT/\frak{P}\bfT )^{\mathrm{int}})\right)$$ 
is obvious. The index of this containment is precisely given by the $p$-part of the Tamagawa factor for the Galois representation $(\bfT/\frak{P}\bfT )^{\mathrm{int}}$ at $\lambda$, which is trivial as $\frak{P}\notin  \mathcal{E}_R$ (therefore, $\textup{Tam}_{\Sigma^{(p)}}(\bfT_{\frak{P}}/\frak{P}\bfT_{\frak{P}})=0$ by definition).
\end{proof}}
Notice that $\iota_\frak{P}$ induces maps 
$$\iota_{\frak{P}}^{v}:\,H^1_{\FF_+}(K_v,\bfT/\frak{P}\bfT)\lra H^1_{\FF_{+}^{\frak{P}}}(K_v,(\bfT/\frak{P}\bfT )^{\mathrm{int}})$$
and for the cohomology groups on Cartier duals the maps 
$$\iota_{\frak{P}}^{v,*}:\,H^1_{\FF_{+}^{\frak{P}}}(K_v,((\bfT/\frak{P}\bfT )^{\mathrm{int}})^\vee(1))\lra 
H^1_{\FF_+^*}(K_v,\bfT^\vee(1)[\frak{P}])\,.$$
for every place $v\in \Sigma$.

\begin{lemma}
\label{lem:localcontrol}
Both maps $\iota_{\frak{P}}^{v}$ and $\iota_{\frak{P}}^{v,*}$ have finite kernel and cokernel whose sizes are bounded by a constant that only depends on $\bfT$, degree of $\frak{P}$ and $[(R/\frak{P})^{\mathrm{int}}:R/\frak{P}]$ as the height one prime $\frak{P}$ varies away from $\mathcal{E}_R$. 
\end{lemma}
\begin{proof}
The case that concerns the primes $v \mid p$ is obvious by the discussion above. Let us take $v$ to be $\lambda\in \Sigma^{(p)}$. 
By Lemma~\ref{lem:finiteisunrawayfromER}, the map $\iota_{\frak{P}}^{v}$ is identical to the composition $\varphi_2 \circ \varphi_1$ of the following maps 
for $\frak{P}\notin  \mathcal{E}_R$:
\begin{align*}
& \varphi_1\,: H^1_{\FF_+}(K_\lambda,\bfT/\frak{P}\bfT)\hookrightarrow H^1_{\textup{ur}}(K_\lambda,\bfT/\frak{P}\bfT), 
\\
& \varphi_2\,:H^1_{\textup{ur}}(K_\lambda,\bfT/\frak{P}\bfT) \lra  H^1_{\textup{ur}}(K_\lambda,(\bfT/\frak{P}\bfT )^{\mathrm{int}})\,, 
\end{align*}
where $H^1_{\textup{ur}}(K_\lambda,\bfT/\frak{P}\bfT):=\ker\left(H^1(K_\lambda,\bfT/\frak{P}\bfT)\lra H^1(K_\lambda^{\textup{ur}},\bfT/\frak{P}\bfT)\right)$. 
Hence, it suffices to prove that the kernel and the cokernel of the maps $\varphi_1$ and $\varphi_2$ are bounded only in terms of $\bfT$ and $[(R/\frak{P})^{\mathrm{int}}:R/\frak{P}]$ as $\frak{P}$ varies away from $\mathcal{E}_R$:
\par 
The fact that the map $\varphi_1$ is injective follows from the definition of the induced Selmer structure $\FF_+$ on $\bfT/\frak{P}\bfT$. Consider the following commutative diagram with exact rows:
$$
\xymatrix{&H^1_{\FF_+}(K_\lambda,\bfT)\otimes R/\frak{P} \ar[r]\ar[d]& H^1(K_\lambda,\bfT)\otimes R/\frak{P}\ar[r]\ar@{^{(}->}[d]&H^1(K_\lambda^{\textup{ur}},\bfT)\otimes R/\frak{P}\ar[r] \ar@{^{(}->}[d]&0
\\
0\ar[r]&H^1_{\textup{ur}}(K_\lambda,\bfT/\frak{P}\bfT)\ar[r]&H^1(K_\lambda,\bfT/\frak{P}\bfT)\ar[r]&H^1(K_\lambda^{\textup{ur}},\bfT/\frak{P}\bfT)\ar[r]&0. 
}$$
The left vertical map is the composition of the natural surjective map 
$H^1_{\FF_+}(K_\lambda,\bfT)\otimes R/\frak{P} \twoheadrightarrow 
H^1_{\FF_+}(K_\lambda,\bfT /\frak{P} \bfT) $ (see Definition \ref{define:selmerobjects}) and the map $\varphi_1$.  
Hence, we have an isomorphism $\textup{coker}(\varphi_1)\cong H^2(K_\lambda,\bfT)[\frak{P}]$ by Snake Lemma. 
The cohomology group $H^2(K_\lambda,\bfT)$ is isomorphic to the $G_{K_\lambda}$-coinvariants
$\bfT (-1)_{G_{K_\lambda}}$ of $\bfT (-1)$, which is finitely generated over $R$ by local Tate duality. 
Hence, the size of $H^2(K_\lambda,\bfT)[\frak{P}]$ is bounded when $\mathfrak{P}$ varies. 
\par 
Next, we turn our attention to $\varphi_2$, which is nothing but the map obtained by 
applying the functor $\otimes_{R/\mathfrak{P}} H^1_{\textup{ur}}(K_\lambda,\bfT/\frak{P}\bfT)$ to 
the injection $R/\mathfrak{P} \hookrightarrow (R/\mathfrak{P})^{\mathrm{int}}$. 
Hence, the kernel (resp. cokernel) of the map $\varphi_2$ is 
$\mathrm{Tor}_{R/\mathfrak{P}} \left(\tfrac{(R/\mathfrak{P})^{\mathrm{int}}}{R/\mathfrak{P}}, H^1_{\textup{ur}}(K_\lambda,\bfT/\frak{P}\bfT)\right) $ (resp. $\dfrac{(R/\mathfrak{P})^{\mathrm{int}}}{R/\mathfrak{P}}\otimes_{R/\mathfrak{P}} H^1_{\textup{ur}}(K_\lambda,\bfT/\frak{P}\bfT)$), which is bounded in the desired manner.
\end{proof}
\begin{lemma}
\label{lem:globalcontrol}
For every height one prime $\frak{P}\subset R$, the composition $\bfT\twoheadrightarrow \bfT/\frak{P}\bfT\hookrightarrow (\bfT/\frak{P}\bfT )^{\mathrm{int}}$ induces maps
\begin{align*}
& \pi_{\frak{P}}:\,\dfrac{H^1_{\FF_+}(K,\bfT)}{\frak{P}H^1_{\FF_+}(K,\bfT)}\lra H^1_{\FF_+^{\frak{P}}}(K,(\bfT/\frak{P}\bfT )^{\mathrm{int}})\\  
& \pi_{\frak{P}}^*:\,H^1_{\FF_+^{\frak{P},*}}(K,((\bfT/\frak{P}\bfT )^{\mathrm{int}})^\vee (1))\lra H^1_{\FF_+^{*}}(K,\bfT^\vee(1))[\frak{P}]\,.
\end{align*}
When $\frak{P}$ is not exceptional, both maps $\pi_{\frak{P}}$ and $\pi_{\frak{P}}^{*}$ have finite kernel and cokernel, of size bounded depending only on $\bfT$, degree of $\frak{P}$ and $[(R/\frak{P})^{\mathrm{int}}:R/\frak{P}]$. Furthermore, when $\bfT=\TT_\cyc$ the map $\pi_{\frak{P}}$ is injective.
\end{lemma}
\begin{proof}
When $\bfT=\TT_\cyc$, this is nothing but \cite[Proposition 5.3.14]{mr02}. Here we extend their arguments to treat the more general set up 
where $\bfT$ is not necessarily equal to $\TT_\cyc$. The main difficulty arises due to Tamagawa factors and their variation in families; we may circumvent this issue thanks to Lemma~\ref{lemma:badplacesavoidingtamagawadefectisfinite} and Lemma~\ref{lem:localcontrol}. 

As in the proof of Lemma~\ref{lem:cartesianfordimension1}, we start with the observation that the size of the cokernel of the map (\ref{eqn:globaldescentraw}) 
is bounded by a constant that depends only on $\bfT$ as $\frak{P}$ varies away from $\mathcal{E}_R$. Likewise, the size of the cokernel of the {map \be\label{eqn:globaldescentrawbasechangedtoIC}
H^1(K_\Sigma/K,\bfT/\frak{P}\bfT) \lra H^1(K_\Sigma/K,(\bfT/\frak{P}\bfT )^{\mathrm{int}})
\ee
is also bounded by a constant depending only on $\bfT$ and $[(R/\frak{P})^{\mathrm{int}}:R/\frak{P}]$. It is also not difficult to see that the map \eqref{eqn:globaldescentrawbasechangedtoIC} is injective, using the following facts as in the proof of \cite[Proposition 5.3.14]{mr02}:
\begin{itemize}
\item[(i)] The $R$-module $\tfrac{(R/\mathfrak{P})^{\mathrm{int}}}{R/\mathfrak{P}}$ has a Jordan-H\"older filtration in which all quotients isomorphic to $R/\mm_R$,
\item[(ii)] Our running assumption that $H^0(K_\Sigma/K,\overline{T})=0$ combined with \cite[Lemma 3.5.2]{mr02} shows that 
$$H^0(K_\Sigma/K,(\bfT/\frak{P}\bfT )^{\mathrm{int}}\big{/}(\bfT/\frak{P}\bfT ))=0\,.$$ 
\end{itemize} 
} Furthermore, the existence of the map $\iota_{\frak{P}}^{v}$ tells us that the map (\ref{eqn:globaldescentrawbasechangedtoIC}) restricts to an injective map
\be\label{eqn:globaldescentSelmerbasechangedtoIC}
H^1_{\FF_+}(K,\bfT/\frak{P}\bfT) \hookrightarrow H^1_{\FF_+^{\frak{P}}}(K,(\bfT/\frak{P}\bfT )^{\mathrm{int}})
\ee
whose cokernel is bounded only in terms of $\bfT$, degree of $\frak{P}$ and $[(R/\frak{P})^{\mathrm{int}}:R/\frak{P}]$  by our previous observations concerning the cokernel of the map (\ref{eqn:globaldescentrawbasechangedtoIC}) and Lemma~\ref{lem:localcontrol}. As the map $\pi_{\frak{P}}$ is the compositum of $\pr_{\frak{P}}$ and the map (\ref{eqn:globaldescentSelmerbasechangedtoIC}), the desired properties of $\ker(\pi_\frak{P})$ and $\textup{coker}\left(\pi_\frak{P}\right)$ follows from the observations above and Lemma~\ref{lem:cartesianfordimension1}.

The bounds on the kernel and cokernel of the map ${\pi}^*_{\frak{P}}$ is obtained in a similar manner, using the second portion of Lemma~\ref{lem:cartesianfordimension1} and \ref{lem:localcontrol}.
\end{proof}

We are now ready to resume with the proof of Theorem~\ref{thm:weakleopoldt} in the situation when $r=1$. 

\begin{proof}[Proof of Theorem~\ref{thm:weakleopoldt} for $r=1$]
Since $c_K \in H^1(K_\Sigma/K,\bfT)$ is non-zero, there are only finitely many height-one primes $\frak{Q}\subset R$ with $c_K\in \frak{Q} H^1(K_\Sigma/K,\bfT)$. Pick a prime $\frak{P}\not\in \mathcal{E}_R$ such that $c_K\not\in \frak{P} H^1(K_\Sigma/K,\bfT)$. 
Then the leading term $\kappa_1^{(\frak{P})}\in H^1_{\FF_+^{\frak{P}}}(K,(\bfT/\frak{P}\bfT )^{\mathrm{int}})$ of 
the Kolyvagin system $\kappa^{(\frak{P})}\in \KS((\bfT/\frak{P}\bfT )^{\mathrm{int}},\FF_+^{\frak{P}})$ obtained as 
the specialization mod $\frak{P}$ of 
the Kolyvagin system $\Psi^{\textup{MR}}(\mathbf{c})\in \KS(\bfT,\FF_+)$ is non-zero. 
Applying \cite[Theorem 5.2.2]{mr02} for $\kappa^{(\frak{P})}$ with the discrete valuation ring $(R/\frak{P})^{\mathrm{int}}$, it follows that the module $H^1_{\FF_+^{\frak{P},*}}(K,((\bfT/\frak{P}\bfT )^{\mathrm{int}})^\vee (1))$ has finite cardinality. Thus, by the statement of Lemma~\ref{lem:globalcontrol} concerning $\pi^\ast_\frak{P}$, the module $H^1_{\FF_+^{*}}(K,\bfT^\vee (1))[\mathfrak{P}]$ is also 
finite and hence cotorsion over $R/\frak{P}$. By the structure theorem of finitely generated $R$-modules, 
$H^1_{\FF_+^{*}}(K,\bfT^\vee (1))$ must be cotorsion over $R$ and this concludes the proof of (1) when $r=1$. 
\par 
{In order to prove (2), notice that $H^1_{\FF_+}(K,\bfT)\subset H^1(K_\Sigma/K,\bfT)$ is $R$-torsion-free as  we assume (H.0). Hence, since $c_K\in H^1_{\FF_+}(K,\bfT)$ is non-zero, it follows that the generic $R$-rank of $H^1_{\FF^+}(K,\bfT)$ is at least one. }
In the case when $\bfT=\TT_\cyc$, let us choose a non-exceptional height one prime $\frak{P}$ that is generated by a linear element. Note that 
$R/\frak{P}$ is a discrete valuation ring and finite flat over $\mathbb{Z}_p$. 
By applying the second part of Theorem~\ref{thm:weakleopoldt0}(1) with $\TT = \bfT /\mathfrak{P} \bfT$, we see that  
$H^1(K_\Sigma/K,\bfT /\mathfrak{P} \bfT )$ is free of rank one over $R/\mathfrak{P}$. Thanks to the injectivity of $\pi_\mathfrak{P}$, 
$\dfrac{H^1(K_\Sigma/K,\bfT )}{\mathfrak{P} H^1(K_\Sigma/K,\bfT )}$
is a non-zero $R/\mathfrak{P}$-submodule of $H^1(K_\Sigma/K,\bfT /\mathfrak{P} \bfT )$. Since non-trivial submodules of free modules of rank one over a discrete valuation ring is still free of rank one, we see that 
$\dfrac{H^1(K_\Sigma/K,\bfT )}{\mathfrak{P} H^1(K_\Sigma/K,\bfT )}$ is free of rank one over $R/\mathfrak{P}$ in this case. 
By Nakayama's lemma, we conclude that $H^1(K_\Sigma/K,\bfT )$ is free of rank one over $R$ when $\bfT=\TT_\cyc$. 
\par 
We now consider the case of general $\bfT$. 
As in (\ref{eqn:nonperfectdescent}), we have the following diagram for each $\mathfrak{P}$ generated by a linear element:
\be\label{eqn:nonperfectdescentmodlinearelements}
\xymatrix{\textup{Tor}_1^R(R/\mathfrak{P},M_\Sigma^+)\ar[r]^{f_\frak{P}}&{\displaystyle\frac{ H^1_{\FF_+}(K,\bfT)}{\mathfrak{P} H^1_{\FF_+}(K,\bfT)}}\ar[r]\ar[d]_{\nu_\mathfrak{P}}&{\displaystyle\frac{ H^1(K_\Sigma/K,\bfT)}{\mathfrak{P}H^1(K_\Sigma/K,\bfT)}}\ar[r]^(.6){\res_\Sigma}\ar@{^{(}->}[d]&{\displaystyle \frac{\textup{Loc}_{\Sigma}^+(\bfT)}
{\mathfrak{P}\textup{Loc}_{\Sigma}^+(\bfT)}}\ar[d]\\
0\ar[r]&{\displaystyle{ H^1_{\FF_+}(K,\bfT/\mathfrak{P} \bfT)}}\ar[r]&H^1(K_\Sigma/K,\bfT/\mathfrak{P} \bfT)\ar[r]&{\textup{Loc}_{\Sigma}^+(\bfT /\mathfrak{P} \bfT)}
}
\ee

Suppose on the contrary to our claim that the $R$-module $H^1_{\FF_+}(K,\bfT)$ has rank at least $2$. By Theorem~\ref{thm:weakleopoldt0} (1) 
and the structure theorem of $R$-modules, $\ker(\nu_\mathfrak{P})$ is not $R/\mathfrak{P}$-torsion. {A simple diagram chase shows that $\ker(\nu_\mathfrak{P})\subset f_{\frak{P}}\left(\textup{Tor}_1^R(R/\mathfrak{P},M_\Sigma^+)\right)$.} Thus for all but finitely many $\mathfrak{P}$ 
generated by linear elements, the $R/\mathfrak{P}$-module $\textup{Tor}_1^R(R/\mathfrak{P},M_\Sigma^+)\cong M_\Sigma^+[\mathfrak{P}]$ is non-torsion as well. This is only possible when the characteristic ideal of $M_{R\textup{-tor}}$ is divisible by all but finitely many $\mathfrak{P}$, which is clearly absurd. 
\par 
For the proof of (3), we make crucial use of Proposition \ref{pro:char0}. 
By making use of Theorem~\ref{thm:weakleopoldt0}(2) and applying Proposition~\ref{pro:char0}(1)(b) and Proposition~\ref{pro:char0}(2)(b) to $M= H^1_{\FF_+^*}(K,\bfT^\vee (1))^\vee$ and $N= H^1_{\FF_+}(K,\bfT)\big{/}R c_K$, we conclude the proof of 
the desired divisibility. We remark that the case when $\bfT=\TT_\cyc$ is also the subject of \cite[Section 5.3]{mr02} and our assertion is verified as part of \cite[Theorem 5.3.10]{mr02}. 
\end{proof}

\textbf{Case} $\mathbf{r>1.}$\label{pageref:caser>1} Suppose now that $R$ is a regular ring of dimension $r+1>2$ which is isomorphic to a power series ring with coefficients in $\mathcal{O}$ in $r$ variables. Since we assume that $c_K$ is non-zero, we have $c_K \notin l H^1(K_{\Sigma}/K,\bfT)$ 
for all but finitely many $(l) \in \mathcal{L}^{(n)}_{\mathcal{O}}$ in the sense of Definition~\ref{def:charideal}. 
The obvious projection maps give rise to a restricted Euler system $\mathbf{c}^{(l)} \in \ES^+(\bfT/ l\bfT)$ 
The initial term $c_K^{(l)}$ is non-zero and the condition {\rm (MR2)} for $\bfT/ l\bfT$ holds true for all but finitely many $(l ) \in \mathcal{L}^{(n)}_{\mathcal{O}}$. 
Identifying the ring $R/(l)$ with a power series ring in $r-1$ variables (resp. in $r$-variables in case $\bfT=\TT_\cyc$), 
it follows by induction that 
\begin{itemize}
\item[(Ind1)] The $R/(l)$-module $H^1_{\FF_{+,l}^*}(K,\bfT/ l\bfT^\vee (1))$ is cotorsion and $H^1_{\FF_{+,l}}(K,\bfT/ l\bfT)$ has rank one;
\item[(Ind2)] $\textup{char}_{R/(l)}(H^1_{\FF_{+,l}^*}(K,(\bfT/ l\bfT )^\vee (1))^\vee)\, \, \supset \, \, 
\textup{char}_{R/(l)}(H^1_{\FF_{+,l}}(K, \bfT/ l\bfT)/(R/(l)) c_K^{(l)})$
\end{itemize} 
for all but finitely many choices of $(l ) \in \mathcal{L}^{(n)}_{\mathcal{O}}$. Here, $\FF_{+,l}$ is the Selmer structure on $\bfT/ l\bfT$ which is given by the local conditions determined by $\FF_+$ at primes above $p$ and by the unramified local conditions $H^1_{\FF_{+,l}}(K_\lambda,\bfT/ l\bfT):=H^1_{\textup{ur}}(K_\lambda,\bfT/ l\bfT)$ at primes $\lambda \in \Sigma^{(p)}$. Let us set the following module:
\begin{equation}\label{equation:moduleQ}
Q:=\dfrac{H^1(K_p,\bfT)}{H^1_{\FF_+}(K_p,\bfT)+\res_p(H^1_{\FFc}(K,\bfT))}.
\end{equation}
\begin{define}
\label{def:nonexceptionallinearelements}
We define the \emph{exceptional set of linear elements} $\mathcal{E}_{\bfT}^{(r)}$ as the set of $(l)\in \mathcal{L}^{(n)}_{\mathcal{O}}$ such that 
\begin{equation*}
\left( H^2(K_\Sigma/K, \bfT) \oplus \bigoplus_{v\in\Sigma}H^2(K_v,\bfT) \oplus \bigoplus_{\lambda\in\Sigma^{(p)}}H^1(I_\lambda,\bfT)^{\textup{Fr}_\lambda=1} 
\oplus Q \oplus H^1_{\FFc^*}(K,\bfT^\vee(1))^\vee \right) [l]
\end{equation*} 
is not a pseudo-null $R$-module. 
We set $$
\mathbf{Err}_{\Sigma}:= \oplus_{v\in\Sigma}H^2(K_v,\bfT)_{\textup{null}}
$$ 
where we write $M_{\textup{null}}$ for the maximal pseudo-null submodule of a finitely generated $R$-module $M$. 
\end{define}
For each positive integer $r$, note that the set $\mathcal{E}_{\bfT}^{(r)}$ has finite cardinality.
\begin{lemma}
\label{lem:localcontrolhigherdim}
The module 
$$
\mathcal{D}_{\Sigma,l}^+:=\bigoplus_{v\in \Sigma} \dfrac{H^1_{\FF_{+,l}}(K_v,\bfT/ l\bfT)}{H^1_{\FF_+}(K_v,\bfT/ l\bfT)}$$ 
is isomorphic to a submodule of $\mathbf{Err}_{\Sigma}[l]$ for all $l \not\in \mathcal{E}_{\bfT}^{(r)}$.
\end{lemma}
\begin{proof}
We sketch a proof which is based on \cite[Lemma 4.1]{ochiai-AIF} and a simple diagram chase as in the proof of Lemma~\ref{lem:cartesianfordimension1}. 
For $v\nmid p$, observe that the submodule $H^1_{\FF_+}(K_v,\bfT/ l\bfT)\subset H^1_{\FF_+,l}(K_v,\bfT/ l\bfT)$ may be identified with the image of the map $\psi$ that is given as part of the commutative diagram below with exact rows:
$$\xymatrix{&H^1_{\FF_+}(K_v,\bfT)\otimes R/(l) \ar[r]\ar@{^{(}->}[d]^{\psi}& H^1(K_v,\bfT)\otimes R/(l)\ar[r]\ar@{^{(}->}[d]&H^1(K_v^{\textup{ur}},\bfT)\otimes R/(l)\ar[r] \ar@{^{(}->}[d]&0
\\
0\ar[r]&H^1_{\textup{ur}}(K_v,\bfT/ l\bfT)\ar[r]&H^1(K_v,\bfT/ l\bfT)\ar[r]&H^1(K_v^{\textup{ur}},\bfT/ l\bfT)\ar[r]&0.
}$$
By Snake Lemma, we have 
$$
\dfrac{H^1_{\FF_{+,l}}(K_v,\bfT/ l\bfT)}{H^1_{\FF_+}(K_v,\bfT/ l\bfT)}\cong\textup{coker}(\psi)\subset H^2(K_v,\bfT)[l]\subset H^2(K_v,\bfT)_{\textup{null}}.
$$
The final containment follows from the fact that $l \not\in \mathcal{E}_{\bfT}^{(r)}$. A similar diagram for primes above $p$ concludes the proof. \end{proof}

We return back to the proof of Theorem~\ref{thm:weakleopoldt}.
\begin{proof}[Proof of Theorem~\ref{thm:weakleopoldt} \rm{(1)} for $r>1$] It follows using Lemma~\ref{lem:localcontrolhigherdim} together with (Ind1) that the $R/(l)$-module $H^1_{\FF_{+}^*}(K,(\bfT/ l\bfT )^\vee(1))$ is also cotorsion for all but finitely many $(l) \in \mathcal{L}^{(r)}_{\mathcal{O}}$. 
Under the running hypothesises of \S \ref{sub:ESboundmain}, it is not difficult to prove the following control result 
$$
H^1_{\FF_{+}^*}(K, (\bfT/ l\bfT) ^\vee(1)) \cong 
H^1_{\FF_{+}^*}(K, \bfT^\vee (1))[l]
$$ (see \cite[Lemma~3.5.3]{mr02}) . 
This concludes the proof of (1) by using the structure theorem of finitely generated $R$-modules. 
\end{proof}
\begin{proof}[Proof of Theorem~\ref{thm:weakleopoldt} \rm{(2)} for $r>1$] It follows from Lemma~\ref{lem:localcontrolhigherdim} and (Ind1) that the $R/(l)$-module $H^1_{\FF_{+}}(K,\bfT/ l\bfT)$ is torsion-free of generic rank one for all but finitely many linear elements $l$. 
Furthermore, thanks to our running hypotheses (H.0) and (H.2), one may argue as in the proof of Proposition~\ref{prop:localcohom} to show that $H^1_{\FF_{+}}(K,\bfT/ l\bfT)$ is $R/(l)$-torsion free. Notice that we have an injective homomorphism
\be\label{eqn:injectivedescentmodlglobal}
\dfrac{H^1(K_\Sigma/K,\bfT)}{ l H^1(K_\Sigma/K,\bfT) }\hookrightarrow H^1(K_\Sigma/K,\bfT/ l\bfT)\,.
\ee

We first handle the case when $\bfT=\TT_\cyc$ and $R=\mathcal{R}_\cyc$. Observe that the quotient $\dfrac{H^1(K_p,\bfT)}{H^1_{\FF_+}(K_p,\bfT)}$ is torsion free by our assumptions and it therefore follows from the exact sequence (which is deduced from the fact that $H^1(K_\lambda,\bfT)=H^1_{\textup{ur}}(K_\lambda,\bfT)$ for $\lambda \in \Sigma^{(p)}$)
$$
0\lra H^1_{\FF_+}(K,\bfT)\lra H^1(K,\bfT) \lra \dfrac{H^1(K_p,\bfT)}{H^1_{\FF_+}(K_p,\bfT)}
$$
that the map (\ref{eqn:injectivedescentmodlglobal}) induces an injection
$$
\dfrac{H^1_{\FF_+}(K,\bfT)}{l H^1_{\FF_+}(K,\bfT)}\hookrightarrow H^1_{\FF_+}(K,\bfT/ l\bfT)\,.$$ 

As $H^1_{\FF_+}(K,\bfT/ l\bfT)$ is a torsion-free $R/(l)$-module of rank one for all but finitely many $l$, it follows that the quotient 
$\dfrac{H^1_{\FF_+}(K,\bfT)}{l H^1_{\FF_+}(K,\bfT)}$ is either trivial or has $R/(l)$-rank one. If the former were the case, it would follow from Nakayama's lemma that $H^1_{\FF_+}(K,\bfT)=0$, contrary to the fact that this module contains the non-zero element $c_K$. We therefore conclude that the quotient 
$\dfrac{H^1_{\FF_+}(K,\bfT)}{l H^1_{\FF_+}(K,\bfT)}$ is an $R/(l)$-module of rank one, for all but finitely many $l$. 
We may repeat this argument to find a sequence of linear elements $\{l_i\}_{i=1}^{r-1}$ where $l_i\in \mathcal{R}_{l_{i-1}}$ for $i<r-1$ 
(with the convention that $l_0=1$) and $l_{r-1} \in \LL_\cyc$, to  conclude that the quotient module $\dfrac{H^1_{\FF_+}(K,\bfT)}{(l_1,\ldots,l_{r-1}) H^1_{\FF_+}(K,\bfT)}$ is torsion-free of rank-one over the discrete valuation ring $R/(l_1,\ldots,l_{r-1})$. In particular, the module $\dfrac{H^1_{\FF_+}(K,\bfT)}{(l_1,\ldots,l_{r-1}) H^1_{\FF_+}(K,\bfT)}$ 
is cyclic over $R/(l_1,\ldots,l_{r-1})$. The proof of (2) when $\bfT=\TT_\cyc$ now follows by Nakayama's lemma.

Next, we handle the general case. As in (\ref{eqn:nonperfectdescent}), we have the following diagram for each linear element $l$:
\be\label{eqn:nonperfectdescentmodlinearelements}
\xymatrix{\textup{Tor}_1^R(R/(l),M_\Sigma^+)\ar[r]&{\displaystyle\frac{ H^1_{\FF_+}(K,\bfT)}{lH^1_{\FF_+}(K,\bfT)}}\ar[r]\ar[d]_{\nu_l}&{\displaystyle\frac{ H^1(K_\Sigma/K,\bfT)}{lH^1(K_\Sigma/K,\bfT)}}\ar[r]^(.6){\res_\Sigma}\ar@{^{(}->}[d]&{\displaystyle \frac{\textup{Loc}_{\Sigma}^+(\bfT)}
{l\textup{Loc}_{\Sigma}^+(\bfT)}}\ar[d]\\
0\ar[r]&{\displaystyle{ H^1_{\FF_+}(K,\bfT/l\bfT)}}\ar[r]&H^1(K_\Sigma/K,\bfT/l\bfT)\ar[r]&{\textup{Loc}_{\Sigma}^+(\bfT /l \bfT)}
}
\ee

Suppose on the contrary to our claim that the $R$-module $H^1_{\FF_+}(K,\bfT)$ has rank at least $2$. By our induction hypothesis 
and the structure theorem of $R$-modules, this means that $\ker(\nu_l)$ is not $R/(l)$-torsion. 
The same argument as the proof of Proof of Theorem~\ref{thm:weakleopoldt} (2) for $r=1$, 
we prove that the $R$-module $H^1_{\FF_+}(K,\bfT)$ has rank at most one. By the existence of a non-trivial element 
(hence non-torsion element by our running hypothesis (H.0)) $c_K \in H^1_{\FF_+}(K,\bfT)$ thus implies that 
$H^1_{\FF_+}(K,\bfT)$ has rank one, as required.
\end{proof}
We define the following two $R$-modules: 
\begin{align*}
& \mathbf{Err}_{+}:= \left( H^2(K_\Sigma/K,\bfT)\oplus Q\oplus H^1_{\FFc^*}(K,\bfT^\vee(1))^\vee \right)_{\textup{null}}  
\\ 
& K_l:=\textup{coker}\left( \dfrac{H^1_{\FF_+}(K,\bfT)}{l H^1_{\FF_+}(K,\bfT)}\stackrel{\nu_l}{\lra} H^1_{\FF_+}(K,\bfT/ l\bfT)\right).
\end{align*}
Before we explain the proof of Theorem~\ref{thm:weakleopoldt} {\rm (3)}, we give the following preliminary lemma. 
\begin{lemma}
\label{lem:globalcontrolhigherdim}
We have 
\be\label{eqn:globalcontrolhigherdimerror1}
\textup{char}_{R/(l)}\left(K_l\right)\, \, \supset \, \,  \textup{char}_{R/(l)}\left(\mathbf{Err}_{+}[l]\right)
\ee
as $l$ varies over non-exceptional linear elements. Moreover, we have
\be\label{eqn:globalcontrolhigherdimerror2}
\textup{char}_{R/(l)}\left(\ker(\nu_l)\right)\, \, \supset \, \,  \textup{char}_{R/(l)}\left(\textup{Loc}^+_{\Sigma} (\bfT)_{\textup{null}}[l]\right) .
\ee
\end{lemma}
Before going into the proof, we set the following modules 
$$
C:=\textup{coker}\left(H^1_{\FF_+}(K,\bfT)\ra H^1_{\FFc}(K,\bfT)\right) ,
\hspace*{1cm}
H^1_{/+}(K_p,\bfT):= \dfrac{H^1(K_p,\bfT)}{H^1_{\FF_+}(K_p,\bfT)}.
$$  
\begin{proof}
By the definition of the Selmer structure $\FF_+$, we have the following exact sequence :
$$
0\lra C\lra H^1_{/+}(K,\bfT) \lra Q\lra 0, $$
where the module $Q$ was defined in \eqref{equation:moduleQ}. 
Since the $R$-module $H^1_{/+}(K,\bfT)$ is torsion-free,
we have that the following exact sequence
\be\label{eqn:localpreparationdescentatp}
0\lra Q[l]\lra C\otimes R/(l)\stackrel{\iota_l}{\lra} H^1_{/+}(K_p,\bfT)\otimes R/(l)\lra Q\otimes R/(l)\lra 0
\ee
by applying $\otimes_R R/(l)$ to the above sequence. 
The definition of the Selmer structure $\FF_+$ yields the following commutative diagram with exact rows:
$$\xymatrix@C=.15cm{&H^1_{\FF_+}(K,\bfT)\otimes R/(l) \ar[rr]\ar[d]^{f_l}&& H^1_{\FFc}(K,\bfT)\otimes R/(l) \ar[r]\ar[d]^{g_l}& C \otimes R/(l)\ar[r]\ar[d]^{h_l}& 0\\
0\ar[r]&H^1_{\FF_+}(K,\bfT /l\bfT )\ar[rr]&& H^1_{\FFc}(K,\bfT /l\bfT )\ar[r]& H^1_{/+}(K_p,\bfT/{l}\bfT )\oplus {\displaystyle\bigoplus_{\lambda\in \Sigma^{(p)}}}\dfrac{H^1_{\FF_{+,l}}(K_\lambda,\bfT/l\bfT )}{H^1_{\FF_+}(K_\lambda,\bfT/l\bfT )}.
}$$
By definition we have $\textup{coker}(f_l)\cong K_l$. As for $\textup{coker}(g_l)$, we have the following claim: 
\begin{claim}
The module $\textup{coker}(g_l)$ fits in an exact sequence 
$$H[l]\lra \textup{coker}(g_l)\lra H^2(K_\Sigma/K,\bfT)[l]$$ 
where $H$ isomorphic to a subquotient of $H^1_{\FFc^*}(K,\bfT^\vee(1))^\vee_{\textup{null}}$\,. In particular, 
$$\textup{char}_{R/(l)}\left(\textup{coker}(g_l)\right)\, \, \supset \, \, \textup{char}_{R/(l)}\left(H^1_{\FFc^*}(K,\bfT^\vee(1))^\vee[l] \right) \textup{char}_{R/(l)}\left(H^2(K_\Sigma/K,\bfT)[l]\right)$$
for every $l \not\in\mathcal{E}_{\bfT}^{(r)}$. 
\end{claim}
\begin{proof}[Proof of Claim] 
We define the modules
$$
\frak{K} \subset \displaystyle{\bigoplus_{v \in \Sigma^{(p)}}} \subset H^1(K_v^\ur,\bfT) 
\hspace*{1cm} \text{(resp. $\frak{K}_l\subset \displaystyle{\bigoplus_{v \in \Sigma^{(p)}}} H^1(K_v^\ur,\bfT/l\bfT)$)}
$$ 
to be the image of 
$H^1(K_\Sigma/K,\bfT)$ (resp. $H^1(K_\Sigma/K,\bfT/ l\bfT)$) under the map 
$$\mathrm{res}^{(p)}:\,H^1(K_\Sigma/K,Y)\stackrel{\res^{(p)}}{\lra} {\displaystyle \bigoplus_{v\in \Sigma^{(p)}}} H^1(K_v^{\textup{ur}},Y)\,\,,\,\, Y=\bfT,\bfT/l\bfT\,.$$ 
Consider the following commutative diagram with exact rows and columns:
$$
\xymatrix{&H^1_{\FFc}(K,\bfT)\otimes R/(l)\ar[r]\ar[d]^{g_l}&H^1(K_\Sigma/K,\bfT)\otimes R/(l) \ar[r]\ar@{^{(}->}[d]& \frak{K}\otimes R/(l)\ar[r]\ar[d]^{\Xi_l}&0\\
0\ar[r]&H^1_{\FFc}(K,\bfT/ l\bfT)\ar[r]\ar[d]&H^1(K_\Sigma/K,\bfT/ l\bfT) \ar[r]\ar[d]& \frak{K}_l\ar[r]&0\\
&\textup{coker}(g_l)\ar[r]&H^2(K_\Sigma/K,\bfT)[l]
}$$
where $\Xi_l$ is induced by the rest of this diagram. 
In order to conclude the proof our claim via Snake Lemma, it remains to prove that $\ker(\Xi_l)$ is isomorphic to a a subquotient of the module $H^1_{\FFc^*}(K,\bfT^\vee(1))^\vee[l]$. 
By definition we have the following short exact sequence: 
$$
0 \longrightarrow \frak{K} \longrightarrow H^1(K_\Sigma/K,\bfT) \longrightarrow \frak{C} \longrightarrow 0 ,
$$ 
where we set $\frak{C}:={\textup{coker}\left(H^1(K_\Sigma/K,\bfT)\stackrel{\res^{(p)}}{\lra} {\displaystyle \bigoplus_{v\in \Sigma^{(p)}}} H^1(K_v^{\textup{ur}},\bfT)\right)}$. 
By applying $\otimes R /(l)$ to this sequence, we have the connecting morphism 
$$
\Omega_l:\, \textup{Tor}_1^R(R/l,\frak{C})=\frak{C}[l] \longrightarrow \frak{K}\otimes R/(l) .
$$ 
By the definitions of the modules $\frak{K}$ and $\frak{K}_l$, we have the following commutative diagram with exact columns:
$$\xymatrix{\frak{C}[l]\ar[d]_{\Omega_l}&\\
\frak{K}\otimes R/(l)\ar[r]^{\Xi_l}\ar[d]&\frak{K}_l\ar[d]\\
{\displaystyle \bigoplus_{v\in \Sigma^{(p)}}} H^1(K_v^{\textup{ur}},\bfT)\otimes R/(l) \ar@{^{(}->}[r]&{\displaystyle \bigoplus_{v\in \Sigma^{(p)}}} H^1(K_v^{\textup{ur}},\bfT/ l\bfT)
}$$
It follows that $\ker(\Xi_l)=\textup{im}(\Omega_l)$ and our claim will be proved once we verify that $\frak{C}$ is a subquotient of $H^1_{\FFc^*}(K,\bfT^\vee(1))^\vee$. This is an immediate consequence of Poitou-Tate global duality, which identifies $\frak{C}$ with $\ker\left(H^1_{\FFc^*}(K,\bfT^\vee(1))^\vee\ra H^1_{\FF_{\textup{str}}^*}(K,\bfT^\vee(1))^\vee\right)$,
where 
$$H^1_{\FF_{\textup{str}}^*}(K,\bfT^\vee(1)):=\ker\left(H^1_{\FFc^*}(K,\bfT^{\vee}(1))\ra \bigoplus_{v\in \Sigma^{(p)}} H^1(K_v,\bfT^\vee(1)) \right).$$ 
(Notice that $H^1_{\FF_{\textup{str}}^*}(K,\bfT^\vee(1))$ and $H^1(K_\Sigma/K,\bfT)$ are dual Selmer groups, in the sense of \cite{mr02}.)
\end{proof}

We resume with the proof of (\ref{eqn:globalcontrolhigherdimerror1}), which will follow (thanks to our observation that $\textup{coker}(f_l)=K_l$ and analysis of $\textup{coker}(g_l)$ above) by Snake Lemma once we check that $Q[l]\stackrel{\sim}{\lra}\ker(h_l)$. To see that, observe that the map $h_l$ is the compositum of the following arrows:
\begin{align*}
h_l:\, C\otimes R/(l) \stackrel{\iota_l}{\lra} H^1_{/+}(K_p,\bfT)\otimes R/(l) &\stackrel{j_l}{\hookrightarrow}  H^1_{/+}(K_p,\bfT/ l\bfT)\\ 
&\hookrightarrow H^1_{/+}(K_p,\bfT/ l\bfT) \oplus{\displaystyle\bigoplus_{\lambda\in \Sigma^{(p)}}}\frac{H^1_{\mathcal{F}_{+,l}}(K_\lambda,\bfT/ l\bfT)}{H^1_{\FF_+}(K_\lambda,\bfT/ l\bfT)}
\end{align*}
where the injection $j_l$ follows form the fact that $H^1_{/+}(K_p,\TT)$ is $R$-torsion free. We conclude by the exactness of the sequence (\ref{eqn:localpreparationdescentatp}) that $Q[l]\stackrel{\sim}{\lra}\ker(h_l)$, as desired.

In order to prove the divisibility (\ref{eqn:globalcontrolhigherdimerror2}), notice that we have
$$\textup{char}_{R/(l)}\left(\ker(\nu_l)\right)\, \, \supset \, \, \textup{char}_{R/(l)}\left(M_\Sigma^+[l]\right) \, \, \supset \, \, \textup{char}_{R/(l)}
\left(\textup{Loc}_\Sigma^+ (\bfT )[l]\right)$$ 
thanks to the diagram (\ref{eqn:nonperfectdescentmodlinearelements}). The proof of (\ref{eqn:globalcontrolhigherdimerror2}) follows as $l$ is a non-exceptional linear element by choice.

\end{proof}

\begin{proof}[Proof of Theorem~\ref{thm:weakleopoldt} {\rm (3)} for $r>1$, based on {\rm (Ind2)}] 
The Poitou-Tate global duality yields the exact sequence
\begin{multline*}
0\lra \frac{H^1_{\FF_{+}}(K,\bfT/ l\bfT)}{R/(l) c_K^{(l)}}\lra \frac{H^1_{\FF_{+,l}}(K,\bfT/ l\bfT)}{R/(l) c_K^{(l)}} \lra \mathcal{D}_{\Sigma,l}^+ 
\\ 
\lra H^1_{\FF_{+,l}^*}(K,(\bfT/ l\bfT )^\vee (1))^\vee\lra H^1_{\FF_{+}^*}(K,(\bfT/ l\bfT )^\vee (1))^\vee\lra 0
\end{multline*}  
which shows that 
\begin{multline}
\label{eqn:comparingquotientsinR1}
\frac{\textup{char}_{R/(l)}\left(H^1_{\FF_{+,l}}(K,\bfT/ l\bfT)/R/(l) c_K^{(l)}\right)}{\textup{char}_{R/(l)}(H^1_{\FF_{+,l}^*}(K,(\bfT/ l\bfT )^\vee (1))^\vee)}
\\ =\frac{\textup{char}_{R/(l)}\left(H^1_{\FF_{+}}(K,\bfT/ l\bfT )/R/(l) c_K^{(l)}\right)}{\textup{char}_{R/(l)}(H^1_{\FF_{+}^*}(K,(\bfT/ l\bfT )^\vee (1))^\vee)}\cdot\textup{char}_{R/(l)}\left(\mathcal{D}_{\Sigma,l}^+\right).
\end{multline}
Since we have 
\begin{align}
\notag\textup{char}_{R/(l)}(H^1_{\FF_{+}^*}(K,(\bfT/ l\bfT )^\vee (1))^\vee)=\textup{char}_{R/(l)}\left(H^1_{\FF_{+}^*}(K,\bfT^\vee (1))^\vee/l 
H^1_{\FF_{+}^*}(K,\bfT^\vee (1))^\vee\right)
\end{align}
thanks to Lemma~3.5.2 of \cite{mr02}, we may rephrase (\ref{eqn:comparingquotientsinR1}) to read
\begin{multline}
\label{eqn:comparingquotientsinR2}
\frac{\textup{char}_{R/(l)}\left(\dfrac{H^1_{\FF_{+,l}}(K,\bfT/ l\bfT)}{R/(l) c_K^{(l)}}\right)}{\textup{char}_{R/(l)}(H^1_{\FF_{+,l}^*}(K,\bfT/ l\bfT^\vee (1))^\vee)}
\\ =\frac{\textup{char}_{R/(l)}\left(\dfrac{H^1_{\FF_{+}}(K,\bfT/ l\bfT)}{R/(l) c_K^{(l)}}\right)}{\textup{char}_{R/(l)}\left(\left(H^1_{\FF_{+}^*}(K,\bfT^\vee (1))^\vee\right)\otimes R/(l)\right)}\cdot\textup{char}_{R/(l)}\left(\mathcal{D}_{\Sigma,l}^+\right).
\end{multline}
Furthermore, we have
\begin{multline*}
\textup{char}_{R/(l)}\left(H^1_{\FF_{+}}(K,\bfT/ l\bfT)/R/(l) c_K^{(l)}\right)\cdot \, {\textup{char}_{R/(l)}}\left(\ker(\nu_\ell)\right)
\\ =
\textup{char}_{R/(l)}\left(\left(H^1_{\FF_{+}}(K,\bfT)/R c_K\right)\otimes R/(l)\right) \textup{char}_{R/(l)}(K_l)
\end{multline*}
where $\nu_\ell$ is as in (\ref{eqn:nonperfectdescentmodlinearelements}) and $K_l$ in Lemma~\ref{lem:globalcontrolhigherdim}. Combining this with (\ref{eqn:comparingquotientsinR2}) and (Ind2), we conclude that 
\begin{multline*}
\textup{char}_{R/(l)}(\left(H^1_{\FF_{+}}(K,\bfT)/R c_K\right) \otimes R/(l)) \cdot\textup{char}_{R/(l)}\left(\mathcal{D}_{\Sigma,l}^+\right)\cdot \textup{char}_{R/(l)}\left(K_l\right)\\
\subset {\textup{char}_{R/(l)}\left(\left(H^1_{\FF_{+}^*}(K,\bfT^\vee(1))^\vee\right)\otimes R/(l)\right)}\cdot {\textup{char}_{R/(l)}}\left(\ker(\nu_\ell)\right)\\
\subset {\textup{char}_{R/(l)}\left(\left(H^1_{\FF_{+}^*}(K,\bfT^\vee(1))^\vee\right)\otimes R/(l)\right)}
\end{multline*}
and hence, by Lemma~\ref{lem:localcontrolhigherdim} and \ref{lem:globalcontrolhigherdim} that
\begin{multline}\label{eqn:charofRHSmain2}
\textup{char}_{R/(l)}\left(\left(H^1_{\FF_{+}}(K,\bfT)/R c_K\right)\otimes R/(l)\right) 
\textup{char}_{R/(l)}\left(\mathbf{Err}_\Sigma[l]\right) \, \textup{char}_{R/(l)}\left(\mathbf{Err}_+[l]\right)
\\ 
\subset {\textup{char}_{R/(l)}\left(\left(H^1_{\FF_{+}^*}(K,\bfT^\vee (1))^\vee\right)\otimes R/(l)\right)}
\end{multline}
for all but finitely many linear elements $l$. 

Now set 
\begin{align*}
& M:=H^1_{\FF_{+}^*}(K,\bfT^\vee (1))^\vee
\\ 
& N:=\left(H^1_{\FF_{+}}(K,\bfT)/R c_K\right)\oplus\mathbf{Err}_\Sigma\oplus\mathbf{Err}_+ 
\end{align*}  
and apply Proposition~\ref{pro:rednn-1} for $(l) \in \mathcal{L}^{(n)}_{\mathcal{O}'}(M_{\mathcal{O}'}) 
\cap \mathcal{L}^{(n)}_{\mathcal{O}'} (N_{\mathcal{O}'})$ for any discrete valuation ring $\mathcal{O}'$ which is finite flat over 
$\mathcal{O}$ and by using induction hypothesis for $r-1$.  
We obtain the desired conclusion noting that $\mathrm{char}_R (\mathbf{Err}_\Sigma\oplus\mathbf{Err}_+ )$ is trivial.
\end{proof}
\begin{rem}
We note that there is Nekov\'a\v{r}'s general descent machine developed as part of his theory of Selmer complexes in \cite{nek}. 
When Nekov\'a\v{r}'s theory applies, it might simplify the descent arguments and yield slightly stronger results. 
See the proof of Proposition~\ref{prop:globalduality} below for an instance of this phenomenon. 
\end{rem}
\subsubsection{Further consequences of Theorem~\ref{thm:weakleopoldt}}
All hypotheses recorded at the start of Section~\ref{sub:ESboundmain} are still in effect. Moreover, we also retain our assumption that $R \cong \LL_{\mathcal{O}}^{(n)}$ is isomorphic to a power series ring.
\begin{prop}
\label{prop:globalduality}
The $R$-module $H^1_{\FF_{\Gr}^*}(K,\bfT^\vee (1))^\vee$ is torsion if and only if $H^1_{\FF_\Gr}(K,\bfT)=0$.
\end{prop}
\begin{proof}
By Global duality theorem, we have the following exact sequence (see Definition \ref{def:Selmerstructures} for the definition of 
local conditions ): 
\begin{equation}\label{equation:globalduality4terms}
0 \lra H^1_{\FF_{\Gr}}(K,\bfT) \lra H^1_{\FF_+}(K ,\bfT) \lra \bigoplus_{v\in \Sigma} 
\frac{H^1_{\FF_+} (K_v ,\bfT)}{H^1_{\FF_{\Gr}}(K_v ,\bfT)} \lra H^1_{\FF_{\Gr}^*}(K,\bfT^\vee (1))^\vee . 
\end{equation}
By Theorem~\ref{thm:weakleopoldt}, the cokernel of the last map in \eqref{equation:globalduality4terms} is a torsion $R$-module. 
The two modules $H^1_{\FF_+}(K ,\bfT)$ and $\displaystyle{\bigoplus_{v\in \Sigma} 
\frac{H^1_{\FF_+} (K_v ,\bfT)}{H^1_{\FF_{\Gr}}(K_v ,\bfT)} }$ in the diagram \eqref{equation:globalduality4terms} are both of generic rank one over $R$. 
It now follows from \eqref{equation:globalduality4terms} that  the $R$-module $H^1_{\FF_{\Gr}}(K,\bfT)$ is torsion if and only if $H^1_{\FF_{\Gr}^*}(K,\bfT^\vee (1))^\vee$ is a torsion $R$-module. Since $H^1(K_\Sigma /K ,\bfT) $ has no non-trivial 
$R$-torsion submodule thanks to our running assumption {\rm (MR1)}, we conclude our proof that $H^1_{\FF_{\Gr}}(K,\bfT)$ is a torsion $R$-module 
if and only if $H^1_{\FF_{\Gr}}(K,\bfT)=0$. 
\end{proof}
\begin{cor}
\label{cor:Greenbergtorsion}
Let $\mathbf{c}\in \textup{ES}^+(\TT_\cyc)$ be an Euler system such that $\res_p^s(c_K)\neq 0$, where 
$c_K\in H^1(K_\Sigma/K,\bfT)$ denotes the image of its initial term. Then $H^1_{\FF_\Gr^*}(K,\bfT^\vee (1))^\vee$ is $R$-torsion.
\end{cor}
\begin{proof}
Let us consider the tautological exact sequence 
$$
0\lra H^1_{\FF_{\Gr}}(K,\bfT)\lra  H^1_{\FF_{+}}(K,\bfT)\stackrel{\res_p^s}{\lra} \dfrac{H^1(K_p,\bfT)}{H^1_{\FF_\Gr}(K_p,\bfT)} . 
$$
Since $H^1_{\FF_{+}}(K,\bfT)$ is torsion-free of rank one by Theorem~\ref{thm:weakleopoldt} (2) and $\res_p^s(c_K)\neq 0$ by assumption,  we have $H^1_{\FF_\Gr}(K,\bfT)=0$. Our assertion follows by Proposition~\ref{prop:globalduality}.
\end{proof}
Recall the restricted singular quotient $H^1_{+/\textup{f}}(K_p,\bfT):=\dfrac{H^1_{++}(K_p,\bfT)}{H^1_{\FF_\Gr}(K_p,\bfT)}$ and the map $\res_{+/\textup{f}}$ which is given as the compositum of the arrows
$$
H^1_{\FF_+}(K,\bfT)\lra H^1_{++}(K_p,\bfT)\lra H^1_{+/\textup{f}}(K_p,\bfT)\,.
$$
\begin{thm}
\label{thm:maingreenbergbound}
Let $\mathbf{c}\in \textup{ES}^+(\TT_\cyc)$ be an Euler system such tha $\res_p^s(c_K)\neq 0$, 
where $c_K\in H^1(K_\Sigma/K,\bfT)$ denotes the image of its initial term. Then, we have 
$$\textup{char}\left(H^1_{\FF_\Gr^*}(K,\bfT^\vee (1))^\vee\right)\,\, \supset \,\, \textup{char}\left(\dfrac{H^1_{+/\textup{f}}(K,\bfT)}{R \res_{+/\textup{f}}\left(c_K\right)}\right)\,.$$
\end{thm}

\begin{proof}
The Poitou-Tate global duality yields (using the proof of Corollary~\ref{cor:Greenbergtorsion} for the injection on the left) the exact sequence
\begin{align*}\label{poitoutatemain}0\lra  H^1_{\FF_{+}}(K,\bfT)\big{/}R c_K\stackrel{\res_{+/\textup{f}}}{\lra} &H^1_{+/\textup{f}}(K_p,\bfT)\big{/}R \res_{+/\textup{f}}(c_K)\\
 &\lra H^1_{\FF_\Gr^*}(K,\bfT^\vee (1))^\vee\lra H^1_{\FF_+^*}(K,\bfT^\vee (1))^\vee  \lra 0\,.\end{align*}
The proof is an immediate consequence of Theorem~\ref{thm:weakleopoldt0} (2) (when $r=0$) and Theorem~\ref{thm:weakleopoldt} (3) (when $r\geq 1$).
\end{proof}

\begin{rem}
\label{rem:allintermsofaKSandsharpening}
We can work with a locally restricted Kolyvagin system $\pmb{\kappa} \in \KS(\TT_\cyc,\FF_+)$ and deduce all our conclusions in this section. Furthermore, Theorem~\ref{thm:theexistenceofKS} shows that the bounds obtained in this section via $\pmb{\kappa}$ are sharp (resp.  sharp after inverting $p$) if and only if $\pmb{\kappa}$ generates the cyclic module $\KS(\TT_\cyc,\FF_+)$ (resp.  generates a submodule of $\KS(\TT_\cyc,\FF_+)$ of finite index). 
\end{rem}


\section{Coleman maps for rank-one subquotients}
\label{sec:Colemanmaps}
\subsection{The setting}\label{section:colemanmap}
Let us denote by $\Gamma_{\mathrm{cyc}}$ the Galois group of the local cyclotomic 
$\mathbb{Z}_p$-extension of $\QQ_p$, which is canonically 
isomorphic to $1+p\mathbb{Z}_p$ via the cyclotomic character $\chi_{\mathrm{cyc}} :  \Gamma_\cyc 
\overset{\sim}{\longrightarrow} 1+p\mathbb{Z}_p$. 
Let $\Gamma_1, \ldots ,\Gamma_r$ be profinite groups each of which are isomorphic to $1+p\mathbb{Z}_p$ via the collection of characters $\chi_i : \Gamma_i \overset{\sim}{\longrightarrow} 1+p\mathbb{Z}_p$ ($i=1,\ldots,r$). 

For each $i$, consider the following Galois characters:
$$
\widetilde{\chi}_i:\ 
G_{\mathbb{Q}_p} \twoheadrightarrow \Gamma_{\mathrm{cyc}} 
\xrightarrow{\chi^{-1}_i \circ \chi_{\mathrm{cyc}}} \Gamma_i 
\hookrightarrow \Gamma_1 \times \cdots \times \Gamma_r 
\hookrightarrow \mathbb{Z}_p [[\Gamma_1 \times \cdots \times \Gamma_r]]^\times . 
$$
Let $\mathcal{R}$ be a local domain which is finite flat over 
$\mathbb{Z}_p [[\Gamma_1 \times \cdots \times \Gamma_r]]$ and let $(\TTo ,\RRn ,\mathcal{S})$ be a deformation datum in the sense of Definition \ref{definition:deformationdatum} with $K=\QQ$. 
Note that we slightly relax the assumption on $\mathcal{R}$ be regular in this portion of our article. 
The hypothesis (H.++) is in effect throughout this section, which we recall for the convenience of the reader:
\\\\
(H.++) There exists an $R$-module direct summand $F^{++}{\TT}$ of $\TT$ which is 
an $R$-module of rank $1+d^{(p)}_+$ containing $F^+_{p}\,\TT$ and is stable under $G_{p}$-action. 
\begin{define}\label{definition:arithmeticpoints}
Let $\mathcal{R}$ be a local domain which is finite flat over 
$\mathbb{Z}_p [[\Gamma_1 \times \cdots \times \Gamma_r]]$ and let $(\TTo ,\RRn ,\mathcal{S})$ be a deformation datum. 
We will be interested in the following list of objects.
\begin{enumerate}
\item 
A continuous ring homomorphism $\kappa : \mathbb{Z}_p [[\Gamma_1 \times \cdots \times \Gamma_r]] 
\lra \overline{\QQ}_p$ is called \emph{arithmetic} if there is an open subgroup 
$U \subset \Gamma_1 \times \cdots \times \Gamma_r$ such that 
$\kappa \vert_U $ coincides with $\chi_1^{w_1 (\kappa )} \times \cdots \times \chi_r^{w_r (\kappa )}$
for an ordered $r$-tuple $\left(w_1 (\kappa ),\ldots , w_r(\kappa )\right) \in \mathbb{Z}^r$. 
When $\kappa$ is an arithmetic specialization, we note that the set
$(w_1 (\kappa ),\ldots , w_r(\kappa )) \in \mathbb{Z}^r$ is independent of the choice of $U$ and the $r$-tuple 
$\left(w_1 (\kappa ),\ldots , w_r(\kappa)\right)$ is called the weight of $\kappa$.  
\item 
Let $\RRn$ be a complete local Noetherian $\ZZ_p$-algebra which is a finite module over 
$\mathbb{Z}_p [[\Gamma_1 \times \cdots \times \Gamma_r]]$. 
A continuous ring homomorphism $\kappa : \RRn \longrightarrow \overline{\QQ}_p$ is called 
\emph{arithmetic} if $\kappa \vert_{\mathbb{Z}_p [[\Gamma_1 \times \cdots \times \Gamma_r]]}$ 
is arithmetic in the sense of previous paragraph. The weight of $\kappa$ is defined to be the weight of 
$\kappa \vert_{\mathbb{Z}_p [[\Gamma_1 \times \cdots \times \Gamma_r]]}$. 
For an arithmetic specialization $\kappa$, we set $V_\kappa:=\TT\otimes_{\kappa} \textup{Frac}(\kappa (\RRn ))$.
\end{enumerate}
\end{define}

\begin{define}\label{definition:arithmeticpoints2}
Let $(\TTo ,\RRn ,\mathcal{S})$ be as in the previous definition. Let $\RRc = \RRn [[\Gamma_{\mathrm{cyc}}]]$ and 
we define $\TTc$ to be $\TT \widehat{\otimes}_{\mathbb{Z}_p} (\Lambda^\sharp_{\mathrm{cyc}})^\iota$ 
where $(\Lambda^\sharp_{\mathrm{cyc}})^\iota$ is a free $\Lambda_{\mathrm{cyc}}$-module of rank one on which $G_{\mathbb{Q}_p}$ 
acts via the inverse tautological character: 
$$
\widetilde{\chi}^{-1}_{\mathrm{cyc}}:\ 
G_{\mathbb{Q}_p} \twoheadrightarrow  \Gamma_{\cyc} \stackrel{\textup{inv}}{\lra} \Gamma_{\mathrm{cyc}}
\hookrightarrow \Lambda^\times_{\mathrm{cyc}} 
= \mathbb{Z}_p [[\Gamma_{\mathrm{cyc}}]]^\times . 
$$
We define 
\begin{equation}
\mathcal{S}_\cyc \subset \{\kappa \in \textup{Hom}(\RRc,\overline{\QQ}_p): \RRc \stackrel{\kappa}{\lra} \overline{\QQ}_p \hbox{ is continuous}\}
\end{equation}
to be the set of specializations $\RRc \stackrel{\kappa}{\lra} \overline{\QQ}_p$ 
such that $\kappa \vert_{\RRn } \in \mathcal{S}$. 
Throughout, the following objects associated to the deformation datum $(\TT_\cyc,\mathcal{R}_\cyc,\mathcal{S}_\cyc)$ will be of interest.
\begin{enumerate}
\item 
A continuous ring homomorphism 
$$
\kappa : \mathbb{Z}_p [[\Gamma_1 \times \cdots \times \Gamma_r ]][[ \Gamma_{\mathrm{cyc}}]]
= \mathbb{Z}_p [[\Gamma_1 \times \cdots \times \Gamma_r \times \Gamma_{\mathrm{cyc}}]]
\lra \overline{\QQ}_p
$$ 
is called \emph{arithmetic} if there is an open subgroup 
$U \subset \Gamma_1 \times \cdots \times \Gamma_r \times \Gamma_{\mathrm{cyc}}$ such that 
$\kappa \vert_U $ coincides with $\chi_1^{w_1 (\kappa )} \times \cdots \times \chi_r^{w_r (\kappa )} \times \chi^{w_\cyc (\kappa)}_{\mathrm{cyc}}$
for an ordered set of integers $\left(w_1 (\kappa ),\ldots , w_r(\kappa ),w_\cyc (\kappa)\right) \in \mathbb{Z}^{r+1}$. 
\item 
Since $\RRc$ is a complete local Noetherian $\ZZ_p$-algebra which is a finite module over 
$\mathbb{Z}_p [[\Gamma_1 \times \cdots \times \Gamma_r \times \Gamma_{\mathrm{cyc}}]]$, 
the notion of an \emph{arithmetic homomorphism} $\kappa : \RRc\longrightarrow \overline{\QQ}_p$ is defined in the same manner as in Definition~\ref{definition:arithmeticpoints}.
\end{enumerate}
\end{define}

\begin{define}\label{definition:ord12}
Let $\RRn$ be a local domain which is finite flat over 
$\mathbb{Z}_p [[\Gamma_1 \times \cdots \times \Gamma_r]]$ and let $(\TTo ,\RRn ,\mathcal{S})$ be a deformation. 
Given a strictly decreasing, 
$G_{\QQ_p}$-stable, exhaustive filtration $\{\mathrm{Fil}_p^i\TT\}_{i \in \mathbb{Z}}$,  
we consider the following condition: 
\begin{enumerate}
\item[(Ord)] 
The $m$-th graded piece $\mathrm{Gr}_p^m\TT:=\mathrm{Fil}_p^{m}\TT/\mathrm{Fil}_p^{m+1}\TT$ 
is a free $\mathcal{R}$-module of rank one for each $m$ such that $0 \leq m \leq d-1$. 
We have $\mathrm{Fil}_p^{m}\TT=\TT$ for $m<0$ and we have $\mathrm{Fil}_p^{m}\TT=0$ for $m\geq d$. 
Moreover, the action of $G_{\QQ_p}$ on $\mathrm{Gr}_p^m\TT$ is given by the character 
$ 
\widetilde{\chi}_1^{e_{m,1}}\cdots \widetilde{\chi}_r^{e_{m,r}} \omega^{1-a_m}\chi^{1-b_m}_{\mathrm{cyc}} \widetilde{\alpha}_m $ 
 where $\omega$ is the Teichm\"uller character and we have
\begin{align*}
& \text{$e_{m,i} \in \{ 0 ,1  \}  $ for $1\leq i \leq r$\,}, \\ 
& a_m ,b_m \in \mathbb{Z}, \\
& \text{$\widetilde{\alpha}_m :\ G_{\QQ_p } \ra \mathcal{R}^\times$ is a non trivial unramified character}
\end{align*}
for each $m\in \{0,1, \ldots,d-1\}$. 
Furthermore, we have $e_{0,i} \leq e_{1,i} \leq \cdots \leq e_{d-1,i}$ for each $i$. 
\end{enumerate}

We say that $\TT_\cyc$ verifies (Ord) if $\TT$ does. We note in this case that the $G_{\QQ_p}$-action on $\mathrm{Gr}_p^m\TT_\cyc$ is given by the character 
$ 
\widetilde{\chi}_1^{e_{m,1}}\cdots \widetilde{\chi}_r^{e_{m,r}} \widetilde{\chi}_{\mathrm{cyc}}^{-1} 
\omega^{1-a_m} \chi^{1-b_m}_{\mathrm{cyc}}  \widetilde{\alpha}_m $\,. 
\end{define}
When $(\TTo_\cyc ,\RRn_\cyc ,\mathcal{S}_\cyc)$ is a deformation datum satisfying (Ord) and $\kappa : \mathcal{R}_\cyc \lra \overline{\QQ}_p$ is an arithmetic specialization  of $\mathcal{R}_\cyc$ (not necessarily an element of $\mathcal{S}_\cyc$) of weight $(w_1 (\kappa ) ,\ldots ,w_r (\kappa ),w_\cyc (\kappa ))$, 
we set
\begin{align*}
& c_{m}(\kappa ):=\left( \displaystyle{\sum_{i=1}^{r} }w_i (\kappa ) e_{m,i}\right) - w_\cyc (\kappa ) +1-b_m 
\\ 
& d_m(\kappa):=-c_m(\kappa)+1
\end{align*}
for each $m\in \{0,1, \ldots,d-1\}$. 

\begin{define}\label{def:criticalarithmeticpoint} 
Let $(\TTo_\cyc ,\RRn_\cyc ,\mathcal{S}_\cyc)$ be a deformation datum that satisfies (Ord). 
For each $m$, we define $\mathcal{S}^{(m),+}_\cyc$ (resp. $\mathcal{S}^{(m),-}_\cyc$) to be the set of 
continuous ring homomorphisms $\kappa : \mathcal{R}_\cyc \longrightarrow \overline{\QQ}_p$ 
such that the rank-one representation 
$\mathrm{Gr}_p^m V_\kappa:=\mathrm{Gr}_p^m \TT_\cyc\otimes_\kappa \overline{\QQ}_p$ is de Rham and $d_m(\kappa)> 0$ (resp. $d_m(\kappa)\leq 0$).
\end{define}
\begin{rem}
\label{rem:ord1ord2whenV++}
It is not hard to see that when a filtration exists verifying (Ord) then it is necessarily unique. Assume the validity of the hypotheses (Pan), (Ord) and (H.++). It then follows that 
$$\textup{F}_{p}^{+}\TT=\textup{Fil}_{p}^{d_{-}^{(p)}}\TT \hbox{\, , \,\,} F^{++}\TT=\textup{Fil}_{p}^{d_{-}^{(p)}-1}\TT\hbox{\,\,\,\,\,\, and \,\,\,\,\,} \textup{Gr}_{p}^{d_-^{(p)}-1}\TT=F^{++}\TT/\textup{F}_{p}^{+}\TT$$
where $d_-^{(p)}:=d-d_+^{(p)}$.
\end{rem}
\subsection{Examples of Galois Deformations and Admissible specializations}
\subsubsection{Rankin-Selberg Convolutions}
\label{subs:exampleRankinSelberg}
Let $\Bf_1=\sum_{n=1}^{\infty} a_n(\Bf_1)q^n$ and $\Bf_2=\sum_{n=1}^{\infty} a_n(\Bf_2)q^n$ be two primitive Hida families as in the introduction. Let $\Sigma$ denote the set of places that contains all rational primes dividing $pN_1 N_2 $ as well as the archimedean prime. Thanks to Hida and Wiles, we have a two-dimensional continuous irreducible Galois representation 
$$\rho_{\Bf_i}: G_{\QQ,\Sigma}\lra \textup{GL}_2(\textup{Frac}(\mathbb{I}_{\Bf_i}))$$
which is characterized by the property that
$$\textup{tr}\left(\rho_{\Bf_i}(\mathrm{Fr}_\ell)\right)=a_\ell(\Bf_i) \hbox{  for every prime } \ell \nmid Np.$$

Assume that both of the following hold true for $i=1,2$:
\begin{itemize}
\item[(A)] We have a free $\mathbb{I}_{\Bf_i}$-module $\mathbb{T}_{\Bf_i}$ that realizes the Galois representation $\rho_{\Bf_i}$. 
\item[(B)] There exists a $G_{\QQ_p}$-stable free $\mathbb{I}_{\Bf_i}$-direct summand $F^+_p\TT_{\Bf_i}\subset \TT_{\Bf_i}$ of rank one with which $\TT_{\Bf_i}$ satisfies (Pan). 
\end{itemize}

The condition (A) is true under the following hypothesis.

{\rm (F.Eis)} $\Bf_i$ is non-Eisenstein mod $p$ (in the sense that the residual representation is absolutely irreducible). 

The following hypothesis together with (A) guarantees the validity of (B).

{\rm (F.Dist)} $\Bf_i$ is $p$-distinguished (in the sense that the semi-simplification of its residual representation restricted to $G_{\QQ_p}$ is a direct sum of distinct characters).

Let us define
$\mathbb{T} = \mathbb{T}_{\Bf_1} \widehat{\otimes}_{\mathbb{Z}_p} \mathbb{T}_{\Bf_2}$, which is a free  $\mathcal{R} = \mathbb{I}_{\Bf_1} \widehat{\otimes}_{\mathbb{Z}_p} \mathbb{I}_{\Bf_2}$-module of rank $4$.  
Then $\mathbb{T}_{\mathrm{cyc}} = \mathbb{T} \widehat{\otimes}_{\mathbb{Z}_p} (\Lambda^\sharp_{\mathrm{cyc}})^\iota$ 
is a free module of rank $4$ over $\RRc:= \mathcal{R}\widehat{\otimes}_{\mathbb{Z}_p} \Lambda_{\mathrm{cyc}}$, 
on which we allow $G_{\QQ,\Sigma}$ act diagonally. It is also easy to see that $d_+(\TT)=d_-(\TT)=2$ and $d_+(\TTc)=d_-(\TTc)=2$.

\par 
In this set up, the set 
$$\mathcal{S}_\cyc
\subset \{\kappa \in \textup{Hom}(\RRc,\overline{\QQ}_p): \RRc \stackrel{\kappa}{\lra} \overline{\QQ}_p \hbox{ is continuous}\}$$ 
introduced in Definition \ref{definition:deformationdatum} is the set of specializations of the form $\kappa_1\otimes\kappa_2\otimes \kappa_{\mathrm{cyc}} $ which are characterized by the following properties:
\begin{enumerate}
\item[(i)] $\kappa_i$ is an arithmetic specialization of weight $w_i \geq 0$ on the ordinary Hida family $\mathbb{I}_{\Bf_i}$, which corresponds to a cuspform of weight $k_i=w_i+2$  ($i=1,2$). 
\item[(ii)] $\kappa_{\mathrm{cyc}}$ is of the form $\chi_\cyc^{j}\phi$ where $j$ is an integer and $\phi$ is a character of $\Gamma_\cyc$ of finite order. 
\item[(iii)] $w_1$, $w_2$ and $j$ satisfy one of the following conditions: 
\begin{enumerate}
\item[(a)] $w_2+1\leq  j <w_1+1$.
\item[(b)] $w_1+1\leq  j < w_2+1$.
\end{enumerate}
\end{enumerate}

The filtration on $\TT_{\Bf_i}$ in (B) above induces a filtration $\{\mathrm{Fil}_p^i\TTc \}_{i \in \mathbb{Z}}$ 
which is characterized by the grading given as follows: 
\begin{align*}
& \mathrm{Gr}_p^0\TTc = (\mathbb{T}_{\Bf_1} /F^+_p \mathbb{T}_{\Bf_1} ) \widehat{\otimes}_{\mathbb{Z}_p} (\mathbb{T}_{\Bf_2} /F^+_p \mathbb{T}_{\Bf_2} ) \widehat{\otimes}_{\mathbb{Z}_p} (\Lambda^\sharp_{\mathrm{cyc}})^\iota\\
& \mathrm{Gr}_p^1\TTc = (\mathbb{T}_{\Bf_1} /F^+_p \mathbb{T}_{\Bf_1} )\widehat{\otimes}_{\mathbb{Z}_p} F^+_p\mathbb{T}_{\Bf_2} \widehat{\otimes}_{\mathbb{Z}_p} (\Lambda^\sharp_{\mathrm{cyc}})^\iota\\
& \mathrm{Gr}_p^2\TTc = F^+_p\mathbb{T}_{\Bf_1} \widehat{\otimes}_{\mathbb{Z}_p} (\mathbb{T}_{\Bf_2}/ F^+_p \mathbb{T}_{\Bf_2} )\widehat{\otimes}_{\mathbb{Z}_p} (\Lambda^\sharp_{\mathrm{cyc}})^\iota\\
& \mathrm{Gr}_p^3\TTc = F^+_p\mathbb{T}_{\Bf_1} \widehat{\otimes}_{\mathbb{Z}_p} F^+_p\mathbb{T}_{\Bf_2} \widehat{\otimes}_{\mathbb{Z}_p} (\Lambda^\sharp_{\mathrm{cyc}})^\iota
\end{align*} 
and $ \mathrm{Gr}_p^m\TT_\cyc=0$ for all other integers $m$. We define 
\begin{align} 
\notag& F^{++} \TTc:=
(F^+_p \TT_{\Bf_1}\widehat{\otimes}_{\mathbb{Z}_p} \TT_{\Bf_2} 
+\TT_{\Bf_1}\widehat{\otimes}_{\mathbb{Z}_p} F^+_p\TT_{\Bf_2})\widehat{\otimes}_{\mathbb{Z}_p} (\Lambda^\sharp_{\mathrm{cyc}})^\iota
\subset \TTc \\
\label{eqn:defineFpplusfroRS}& F^+_p \TTc:=F^+_p \TT_{\Bf_1}\widehat{\otimes}_{\mathbb{Z}_p} \TT_{\Bf_2}\widehat{\otimes}_{\mathbb{Z}_p} (\Lambda^\sharp_{\mathrm{cyc}})^\iota \subset F^{++} \TTc 
\end{align}
We note that $F^+_p \TTc$ (resp. $F^{++} \TTc$) is nothing but $\mathrm{Fil}_p^{2}\TTc$ (resp. $\mathrm{Fil}_p^{1}\TTc$) and (H.++) is valid with these choices. 
\par 
Note that, as $\mathbb{T}_{\Bf_1} \widehat{\otimes}_{\mathbb{Z}_p} \mathbb{T}_{\Bf_2} \widehat{\otimes}_{\mathbb{Z}_p} (\Lambda^\sharp_{\mathrm{cyc}})^\iota$ is isomorphic to 
$\mathbb{T}_{\Bf_2} \widehat{\otimes}_{\mathbb{Z}_p} \mathbb{T}_{\Bf_1} \widehat{\otimes}_{\mathbb{Z}_p} (\Lambda^{\sharp}_{\mathrm{cyc}})^\iota$, 
the representation $\TTc$ remains the same  if we interchange the roles of $\Bf_1$ and $\Bf_2$. 
Similarly, it is clear that the filtration $F^{++} \TTc=\mathrm{Fil}_p^{1}\TTc$ remains unchanged when we interchange the roles of $\Bf_1$ and $\Bf_2$. However, $F^+_p \TTc=\mathrm{Fil}_p^{2}\TT_\cyc$ does change if we interchange $\Bf_1$ and $\Bf_2$.  

\par 
For each $m\in \{0,1,2,3\}$, the set ${\mathcal{S}^{(m),-}_\cyc}$ (resp. ${\mathcal{S}^{(m),+}_\cyc}$) of Definition \ref{def:criticalarithmeticpoint} is characterized in the current set up by the following conditions: 
\begin{enumerate}
\item[$m=0$: ] $w_1$, $w_2$ are arbitrary and $j<0 $ \,\,\,\,(resp. $j\geq 1$). 
\item[$m=1$: ] $w_1$ is arbitrary and $j<\omega_2+1$ \,\,\,\,(resp. $j\geq \omega_2+1$). 
\item[$m=2$: ] $w_2$ is arbitrary and $j<\omega_1+1$ \,\,\,\,(resp. $j\geq\omega_1+1$). 
\item[$m=3$: ] $j<w_1 +w_2+2$ \,\,\,\,\,\,\,\,\,\,\,\,\,\,\,\,\,\,\,\,\,\,\,\,\,\,\,\,\,\,\,\,\,\,\,\,\,\,\,\,(resp. $j\geq w_1 +w_2+2$).
\end{enumerate}

Denote by $\mathcal{S}_\cyc^{(1)}$ the subset of $\mathcal{S}_\cyc$ that consists of all specializations for which the module $F_p^+\TT_\cyc$ in the statement of the condition (Pan) is chosen as in (\ref{eqn:defineFpplusfroRS}) above. It is not difficult to see that $\kappa \in \mathcal{S}_\cyc^{(1)}$ if and only if it verifies (i), (ii), (iii)-(a). We may similarly define $\mathcal{S}_\cyc^{(2)}$ exchanging the roles of $\Bf_1$ and $\Bf_2$.

\subsubsection{Siegel modular forms}
Our main reference in this example is \cite[Section 2]{ochiai-AJM14} that we summarize here after altering their notation to fit our general set up here. Let $\Bf$ be the branch of an ordinary Hida family with tame level $N$ for the group 
$\mathrm{GSp}(4)$, which is defined over a local domain  $\mathbb{I}_{\Bf}$ finite flat over a two-variable 
Iwasawa algebra $\mathbb{Z}_p [[\Gamma_1 \times \Gamma_2 ]]$. We note that the groups $\Gamma_1$ and $\Gamma_2$ correspond to the groups $G_1$ and $G_2$ in \cite{ochiai-AJM14}, whereas our $\mathbb{I}_{\Bf}$ corresponds to $R^{\textup{ord}}$ in op. cit. Attached to $\Bf$, the work of Tilouine-Urban, Urban and Pilloni equips us (under mild technical hypothesis on the residual representation, c.f. Theorem~2.5 in \cite{ochiai-AJM14}) with a free $\mathbb{I}_{\Bf}$-module $\mathbb{T}_{\Bf}$  
of rank four on which $G_{\mathbb{Q},\Sigma}$ acts continuously (that corresponds to $\overline{\TT}^{\textup{ord}}$ in loc. cit.). 
Let us define
$$ 
\mathbb{T}_\cyc = \mathbb{T}_{\Bf} \widehat{\otimes}_{\mathbb{Z}_p}  (\Lambda^\sharp_{\mathrm{cyc}})^\iota
$$
which is free of rank $4$ over the ring $\mathcal{R}:= \mathbb{I}_{\Bf} \widehat{\otimes}_{\mathbb{Z}_p} \Lambda_{\mathrm{cyc}}$ and endow it with the diagonal $G_{\QQ,\Sigma}$ action. We note that the coefficient ring $\mathcal{R}$ here corresponds to $R^{\textup{n.ord.}}$ and $\TT_\cyc$ here to $\overline{\TT}^{\textup{n.ord}}$ in op. cit. 
\par 
The set $\mathcal{S}_\cyc$
of Definition~\ref{definition:deformationdatum} in this set up is the set of specializations of the form $\lambda\otimes \kappa_{\mathrm{cyc}} $\,, which are characterized by the following properties (thanks to the description of the associated local Galois representation in \cite[Corollary 2.7]{ochiai-AJM14} and with the choice $F^+_p\TT_\cyc:=(F^-\TT^{\textup{n.ord}})^{\mathcal{R}}(1)$ where the quotient $F^-\TT^{\textup{n.ord}}$ is given as on pg. 739 of op. cit.):

\begin{enumerate}
\item[(i)] $\lambda$ is an arithmetic specialization of the Hida family $\mathbb{I}_{\Bf_i}$ of weight $(w_1,w_2)$ with integers $0<w_2<w_1$,
which corresponds to a cuspidal automorphic representation $\pi_\lambda$ of $\textup{GSp}_4(\mathbb{A}_{\QQ})$. 
\item[(ii)] $\kappa_{\mathrm{cyc}}$ is of the form $\chi_\cyc^{j}\phi$ where $j$ is an integer and $\phi$ is a character of $\Gamma_\cyc$ of finite order. 
\item[(iii)] $w_2+2 \leq  j < w_1+3$\,.
\end{enumerate}
It is also clear that (H.{++}) is satisfied in this case. The set $\mathcal{S}^{(m),-}$ (resp. $\mathcal{S}^{(m),+}$) for $m\in \{0,1,2,3\}$ is characterized in the current set up by the following conditions: 

\begin{enumerate}
\item[$m=0$: ] $j < w_1 +w_2+4$\,\,\,\,\,\,\,\,\,\,\,\,\,\,\,\,\,\,\,\,\,\,\,\,\,\,\,\,\,\,\,\,\,\,\,\,\,\,\,\,\,(resp. $j\geq w_1+w+2+4$).
\item[$m=1$: ] $w_2$ is arbitrary and $j < w_1 +3$\,\,\,\,(resp. $j\geq w_1+3$)
\item[$m=2$: ] $w_1$ is arbitrary and $j < w_2 +2$  \,\,\,(resp. $j\geq w_2+2$) 
\item[$m=3$: ] $w_1$, $w_2$ are arbitrary and $j< 1$\,\,\,\,\,(resp. $j\geq 1$). 
\end{enumerate}



\subsection{Coleman map and its interpolation property}
\label{subsec:largeColemanmap}
We retain the notation of Section \ref{section:colemanmap}. 
The goal of this section is to give a construction of the so called big exponential map and big dual exponential map for each 
graded piece of a deformation datum $(\TTo_\cyc ,\RRn_\cyc ,\mathcal{S}_\cyc)$ satisfying the condition (Ord), as stated in 
Section~\ref{section:colemanmap}. Our main result in this section (Theorem~\ref{thm:interpolationforcoleman}) is presented at the end.   
We first introduce some notation and state various preparatory results. 
{
Let $\mathbb{D}(\ZZ_p [[G_{\mathrm{cyc}}]]^\sharp)$ denote the canonical $\mathbb{Z}_p$-lattice inside the module 
$\mathcal{D}_{\infty}(\mathbb{Q}_p)\otimes_{\mathbb{Q}_p} 
\mathbb{Q}_p (\mu_p)$ given as in the introduction of \cite{perrinriou}. The lattice $\mathbb{D}(\ZZ_p [[G_{\mathrm{cyc}}]]^\sharp)$ may be explicitly described as $\mathbb{D}(\ZZ_p [[G_{\mathrm{cyc}}]]^\sharp):= D_{\mathrm{crys}} (\mathbb{Z}_p) {\otimes}_{\mathbb{Z}_p} \ZZ_p [[G_{\mathrm{cyc}}]]$, where $D_{\mathrm{crys}} (\mathbb{Z}_p):=(A_{\mathrm{crys}})^{G_{\mathbb{Q}_p}} $ is the canonical lattice in $D_{\mathrm{crys}} (\mathbb{Q}_p) = (B_{\mathrm{crys}})^{G_{\mathbb{Q}_p}} $ (which we may and we do canonically identify with $\mathbb{Z}_p$). Note then that $\mathbb{D}(\ZZ_p [[G_{\mathrm{cyc}}]]^\sharp)$ a free $\ZZ_p [[G_{\mathrm{cyc}}]]$-module of rank one.
\begin{lemma}\label{lemma:D}
\label{lemma:choiceofalatticeinbigderham}
The $\ZZ_p [[G_{\mathrm{cyc}}]]$-module $\mathbb{D}(\ZZ_p [[G_{\mathrm{cyc}}]]^\sharp)$ has the property that $\chi_\cyc^j \phi (\mathbb{D}(\ZZ_p [[G_{\mathrm{cyc}}]]^\sharp))$ is identified with a lattice in $D_{\textup{dR}}(\QQ_p(j)\otimes \mathcal{O}(\omega^{-j} \phi ) )$ for every character $\phi$ of $G_\cyc$ of finite order 
where we define $\mathcal{O}(\omega^{-j} \phi )$ to be a free $\mathbb{Z}_p [\omega^{-j} \phi ]$-module of rank one on which $G_{\mathbb{Q}_p}$ acts by the character $\omega^{-j} \phi$. 
\end{lemma}
\begin{proof} This is implicitly proved in \cite{perrinriou}. 
Note that in loc. cit., specialization of an element $x \in \mathcal{D}_{\infty}(\mathbb{Q}_p)\otimes_{\mathbb{Q}_p} \mathbb{Q}_p (\mu_p)$ via the character $\chi_\cyc^j \phi$ is equivalent  to specializing the element 
$x\otimes e_1^{\otimes j} \in \mathcal{D}_{\infty}(\mathbb{Q}_p(j)) \otimes_{\mathbb{Q}_p} 
\mathbb{Q}_p (\mu_p)
= \mathcal{D}_{\infty}(\mathbb{Q}_p) \otimes e_1^{\otimes j}$ 
via the the character $ \phi$, where $e_1$ is a basis of 
$D_{\mathrm{crys}} (\mathbb{Q}_p (1))$ specified by our fixed norm compatible system $\{ \zeta_{p^n}\}$ of $p$-power roots of unity. It therefore remains to explain the definition of the specialization by $\chi_\cyc^j \phi$ when $j=0$. Let $n$ be the smallest natural number such that the character $\phi$  factors through the quotient $\mathrm{Gal} (\mathbb{Q}_p( \mu_{p^n})/\mathbb{Q}_p) $. 
Consider the projection map 
\begin{equation}\label{equation:evaluation1}
\mathbb{D}(\ZZ_p [[G_{\mathrm{cyc}}]]^\sharp):= D_{\mathrm{crys}} 
(\mathbb{Z}_p) {\otimes}_{\mathbb{Q}_p} G_{\mathrm{cyc}}
\longrightarrow \mathbb{Q}_p 
[\mathrm{Gal} (\mathbb{Q}_p( \mu_{p^n})
/\mathbb{Q}_p)]
\end{equation} 
as well as the following $\mathrm{Gal} (\mathbb{Q}_p( \mu_{p^n})
/\mathbb{Q}_p)$-equivariant identifications  
\begin{equation}\label{equation:evaluation2}
D_{\mathrm{dR}} (\mathbb{Q}_p [\mathrm{Gal} (\mathbb{Q}_p( \mu_{p^n})/\mathbb{Q}_p)]^\sharp )
\cong  \mathbb{Q}_p (\mu_{p^n})\cong \mathbb{Q}_p [\mathrm{Gal} (\mathbb{Q}_p( \mu_{p^n})/\mathbb{Q}_p)].
\end{equation}
The evaluation map $\mathrm{Gal} (\mathbb{Q}_p(\mu_{p^n})/\mathbb{Q}_p)]^\sharp \stackrel{\phi}{\longrightarrow} E_\phi (\phi)$ 
induces 
\begin{equation}\label{equation:evaluation3}
D_{\mathrm{dR}} (\mathbb{Q}_p [\mathrm{Gal} (\mathbb{Q}_p( \mu_{p^n})
/\mathbb{Q}_p)]^\sharp ) \longrightarrow 
D_{\mathrm{dR}} 
(E_\phi (\phi) ). 
\end{equation}
The desired specialization map 
$$
\phi : \mathbb{D}(\ZZ_p [[G_{\mathrm{cyc}}]]^\sharp):= D_{\mathrm{crys}} 
(\mathbb{Z}_p) {\otimes}_{\mathbb{Z}_p}G_{\mathrm{cyc}} 
\longrightarrow D_{\mathrm{dR}} 
(E_\phi (\phi)) 
$$
is now defined to be the composition of the maps 
\eqref{equation:evaluation1}, 
\eqref{equation:evaluation2} and 
\eqref{equation:evaluation3} above. 
\end{proof}}
 
Let us set $\mathbb{D}^{\mathrm{ur}}(\ZZ_p [[G_{\mathrm{cyc}}]]^\sharp):=\mathbb{D}(\ZZ_p [[G_{\mathrm{cyc}}]]^\sharp)\widehat{\otimes}
_{\mathbb{Z}_p} \widehat{\mathbb{Z}}^{\mathrm{ur}}_p$. {For any $G_{\mathbb{Q}_p}$-representation $V$, we denote by $D^{\mathrm{ur}}_{\textup{dR}}(V )$ 
the filtered module $(V \otimes B_{\mathrm{dR}})^{G_{\mathbb{Q}^{\mathrm{ur}}_p}}$. }
The main construction we shall carry out in this section relies on the following theorem, which essentially is a reformulation of the Coleman map in its most classical form (that was introduced by Coleman himself). 
\begin{thm}\label{thm:interpolationclassic_unr} 
We have a $\ZZ_p [[G_{\mathrm{cyc}}]]$-linear isomorphism  
$$\mathrm{EXP}^{\mathrm{ur}}_{\ZZ_p [[G_{\mathrm{cyc}}]]^\sharp} \,:\, 
\mathbb{D}^{\mathrm{ur}}(\ZZ_p [[G_{\mathrm{cyc}}]]^\sharp)  \lra 
{H^1 (\QQ_p^{\mathrm{ur}},\ZZ_p [[G_{\mathrm{cyc}}]]^\sharp)}\big{/}{\ZZ_p} 
$$ 
such that, for every arithmetic specialization 
$\kappa$ on $\ZZ_p [[G_{\mathrm{cyc}}]]$ of weight $w_\cyc (\kappa)>0$, 
the following diagram commutes:
$$\xymatrix{\mathbb{D}^{\mathrm{ur}}(\ZZ_p [[G_{\mathrm{cyc}}]]^\sharp )\ar[d]_\kappa 
\ar[rrr]^{\mathrm{EXP}^{\mathrm{ur}}_{\ZZ_p [[G_{\mathrm{cyc}}]]^\sharp }}&&& 
{H^1 (\QQ_p^{\mathrm{ur}},\ZZ_p [[G_{\mathrm{cyc}}]]^\sharp )}\big{/}
{\ZZ_p} 
\ar[d]^\kappa\\
D^{\mathrm{ur}}_{\textup{dR}}(E_\kappa (\kappa ) )\ar[rrr]_{
\exp^{\mathrm{ur}}_{E_\kappa (\kappa ) }
 \,\circ \,  e^{\mathrm{ur},+}_p
 }&&& 
 H^1(\QQ_p^{\mathrm{ur}},E_\kappa (\kappa ))\,.
 }
$$
Here $e^{\mathrm{ur},+}_p:=(-1)^{w_\cyc (\kappa)-1}(w_\cyc (\kappa)-1)!\,e^{\mathrm{ur}}_{p}$ and 
$e^{\mathrm{ur}}_{p}=e^{\mathrm{ur}}_{p}(E_\kappa (\kappa ))$ is the $p$-adic multiplier given by
$$
e^{\mathrm{ur}}_{p}:=\begin{cases} 
\left( 
1 - p^{-1}\varphi^{-1}
\right) \left( 
1 - \varphi
\right)^{-1} & \text{when $\kappa  
=(\chi_\cyc \omega )^{w_\cyc (\kappa )} $,} \vspace*{5pt} \\ 
\left( p^{-1}\varphi^{-1 } \right)^n & 
\text{when $\kappa = (\chi_\cyc \omega )^{w_\cyc (\kappa )}\phi$ with $\phi$ of conductor $p^n >1$\,.}  \\ 
\end{cases}
$$
Also, for every arithmetic specialization 
$\kappa$ on $\ZZ_p [[G_{\mathrm{cyc}}]]$ of weight $w_\cyc (\kappa)\leq 0$, 
the following diagram commutes:
$$
\xymatrix{\mathbb{D}^{\mathrm{ur}}(\ZZ_p [[G_{\mathrm{cyc}}]]^\sharp )\ar[d]_\kappa 
\ar[rrr]^{\mathrm{EXP}^{\mathrm{ur}}_{\ZZ_p [[G_{\mathrm{cyc}}]]^\sharp }}&&& 
{H^1 (\QQ_p^{\mathrm{ur}},\ZZ_p [[G_{\mathrm{cyc}}]]^\sharp )}\big{/}{\ZZ_p} 
\ar[d]^\kappa\\
D^{\mathrm{ur}}_{\textup{dR}}(E_\kappa (\kappa ) )\ar[rrr]_{
 (\exp^{\mathrm{ur},\ast}_{E_\kappa (\kappa ) })^{-1}
 \,\circ \,  e^{\mathrm{ur},-}_p
 }&&& 
 H^1(\QQ_p^{\mathrm{ur}},E_\kappa (\kappa ))
 }
$$
where $e^{\mathrm{ur},-}_p:=\dfrac{ e^{\mathrm{ur}}_p}{(-w_\cyc (\kappa))!}$\,. 
\end{thm}
The proof of Theorem \ref{thm:interpolationclassic_unr} was explained in \cite[Prop. 5.10]{ochiai-AJM03} (see also \cite[Prop. 4.2]{ochiai-AJM14}), except for the second interpolation diagram (concerning the non-positive weights) which follows from \cite[Theorem II.10]{BergerBlochKatoDocumenta}. We do not repeat the proof of Theorem~\ref{thm:interpolationclassic_unr} for this reason, but simply note that the construction of the map in Theorem~\ref{thm:interpolationclassic_unr} follows very closely the theory of classical Coleman power series for the extension $\QQ^{\mathrm{ur}}_p (\mu_{p^\infty})/\QQ^{\mathrm{ur}}_p$. Indeed, notice that the group $H^1(\QQ_p^{\mathrm{ur}},\ZZ_p (1)\otimes \ZZ_p [[G_{\mathrm{cyc}}]]^\sharp)$ is isomorphic by Kummer theory to the inverse limit (with respect to norm maps) of $\QQ_p^{\mathrm{ur}}(\mu_{p^{n+1}})^{\times,\wedge}$ (where the superscript $\wedge$ stands for $p$-adic completion). With this identifications at hand, Theorem~\ref{thm:interpolationclassic_unr} is a reformulation of the classical theory of Coleman power series. 

\begin{cor}\label{lem:invariantunramified}
Let $\RRn$ be a complete Noetherian semi-local $\ZZ_p$-algebra of characteristic $0$ 
and let $\widetilde{\alpha}:G_{\mathbb{Q}_p} \longrightarrow R ^{\times}$ be a non-trivial
continuous unramified character. 
Then, the Galois cohomology group $H^1(\mathbb{Q}_p, \ZZ_p [[G_{\mathrm{cyc}}]]^\sharp \widehat{\otimes}_{\ZZ_p}\RRn (\widetilde{\alpha}) )$
is a free $\RRn$-module of rank one over 
$\ZZ_p [[G_{\mathrm{cyc}}]] \widehat{\otimes}_{\ZZ_p} \RRn$.
\end{cor}
\begin{proof}
We shall apply the formal tensor product functor $\left(-\widehat{\otimes}_{\ZZ_p} \RRn (\widetilde{\alpha})\right)$
to the isomorphism $\mathrm{EXP}^{\mathrm{ur}}_{\ZZ_p [[G_{\mathrm{cyc}}]]^\sharp }$ 
and take the $\mathrm{Gal} (\mathbb{Q}^{\mathrm{ur}}_p/\mathbb{Q}_p)$-invariants. 
We shall need the following lemma that we recall from \cite[Lemma 3.3]{ochiai-AJM03}:
\begin{lemma}\label{rank} 
Let $R$ be a complete semi-local Noetherian $\ZZ_p$-algebra of mixed characteristic and let
$M$ be a free $R$-module of finite rank $e$ that is endowed with an unramified action of $G_{\mathbb{Q}_p}$. Then $(M \widehat{\otimes}_{\mathbb{Z}_p}
\widehat{\mathbb{Z}}^{\mathrm{ur}}_p)^{G_{\mathbb{Q}_p}}$ is a free $R$-module of rank $e$.
\end{lemma}
Recall that $\mathbb{D}^{\mathrm{ur}}(\ZZ_p [[G_{\mathrm{cyc}}]]^\sharp ) \widehat{\otimes}_{\ZZ_p} \RRn (\widetilde{\alpha})$ 
is isomorphic to $\widehat{\mathbb{Z}}_p^{\mathrm{ur}}[[G_\cyc ]] \widehat{\otimes}_{\ZZ_p} \RRn (\widetilde{\alpha})$. 
Applying Lemma~\ref{rank} with $R=\ZZ_p [[G_{\mathrm{cyc}}]] ) \widehat{\otimes}_{\ZZ_p} \RRn$, it follows that
\begin{equation}\label{equation:invariant_unr_1}
\left( 
\mathbb{D}^{\mathrm{ur}}(\ZZ_p [[G_{\mathrm{cyc}}]]^\sharp ) \widehat{\otimes}_{\ZZ_p} \RRn (\widetilde{\alpha})
\right)^{\mathrm{Gal} (\mathbb{Q}^{\mathrm{ur}}_p/\mathbb{Q}_p)}
\cong \ZZ_p [[G_{\mathrm{cyc}}]] ) \widehat{\otimes}_{\ZZ_p} \RRn . 
\end{equation}
On the other hand, since $\RRn (\widetilde{\alpha})^{G_{\QQ_p}} =0$ by assumption we have 
\begin{equation}\label{equation:invariant_unr_2}
\left( \left({H^1 (\QQ_p^{\mathrm{ur}},\ZZ_p [[G_{\mathrm{cyc}}]]^\sharp )}\big{/}{\ZZ_p} \right)
\widehat{\otimes}_{\ZZ_p} \RRn (\widetilde{\alpha}) \right)^{\mathrm{Gal} (\mathbb{Q}^{\mathrm{ur}}_p/\mathbb{Q}_p)}
\cong H^1 (\QQ_p^{\mathrm{ur}},\ZZ_p [[G_{\mathrm{cyc}}]]^\sharp \widehat{\otimes}_{\ZZ_p} \RRn (\widetilde{\alpha}))
^{\mathrm{Gal} (\mathbb{Q}^{\mathrm{ur}}_p/\mathbb{Q}_p)}
\end{equation}
 The proof of the corollary follows combining the isomorphisms \eqref{equation:invariant_unr_1} and \eqref{equation:invariant_unr_2}. 
\end{proof}

For integers $a,b$, we define $\mathbb{D}(\ZZ_p (\omega^a) \otimes \mathbb{Z}_p ( b ) \otimes \Lambda^\sharp_{\mathrm{cyc}})$ by setting
$$
\mathbb{D}(\ZZ_p (\omega^a) \otimes \mathbb{Z}_p ( b ) \otimes \Lambda^\sharp_{\mathrm{cyc}}) :=
 \mathbb{D}(\Lambda^\sharp_{\mathrm{cyc}}) \otimes_{\ZZ_p} D_{\textup{dR}}(\ZZ_p(\omega^{a})) \otimes_{\ZZ_p} D_{\textup{crys}}(\ZZ_p (b)). 
$$
Here $D_{\textup{crys}}(\ZZ_p (b))$ is the $\ZZ_p$-lattice in $D_{\textup{crys}}(\QQ_p(b)) \subset B^{+}_{\mathrm{dR}}$ 
generated by $\{ \zeta^b_{p^m} \} \otimes t^{\otimes (-b)}$ over $\mathbb{Z}_p$, 
where $t:=\log [\epsilon ] $ is the well-known uniformizer of the ring of $p$-adic periods $B^{+}_{\mathrm{dR}}$ of Fontaine. 
It is not difficult to see that $\mathbb{D}(\ZZ_p (\omega^a) \otimes \mathbb{Z}_p ( b ) \otimes \Lambda^\sharp_{\mathrm{cyc}})$ 
is naturally a free $\LL_\cyc$-module of rank one. 
We further define the module  
$$
\mathbb{D}^{\mathrm{ur}}(\ZZ_p (\omega^a) \otimes \mathbb{Z}_p ( b ) \otimes \Lambda^\sharp_{\mathrm{cyc}})
 := \mathbb{D} (\ZZ_p (\omega^a) \otimes \mathbb{Z}_p ( b ) \otimes \Lambda^\sharp_{\mathrm{cyc}})\widehat{\otimes}_{\ZZ_p}
 \widehat{\ZZ}^{\mathrm{ur}}_p. 
$$
The following is a slight generalization of Theorem~\ref{thm:interpolationclassic_unr}.
\begin{thm}\label{perrinriou3}
For an arbitrary pair of  integers $a$ and $b$, we have an $\Lambda_{\mathrm{cyc}}$-linear isomorphism  
$$
\mathrm{EXP}^{\mathrm{ur},(a,b)}_{\ZZ_p (\omega^a) \otimes \mathbb{Z}_p ( b ) \otimes \Lambda^\sharp_{\mathrm{cyc}}} \,:\, 
\mathbb{D}^{\mathrm{ur}}(\ZZ_p (\omega^a) \otimes \mathbb{Z}_p ( b ) \otimes \Lambda^\sharp_{\mathrm{cyc}})  \lra 
\dfrac{H^1(\QQ_p^{\mathrm{ur}},\ZZ_p (\omega^a) \otimes \mathbb{Z}_p ( b ) \otimes \Lambda^\sharp_{\mathrm{cyc}})}{\ZZ_p (\omega^a) \otimes \mathbb{Z}_p ( b )}
$$ 
such that, for every arithmetic specialization 
$\kappa$ on $\Lambda_{\mathrm{cyc}}$ of weight $w_\cyc (\kappa ) \geq 1-b$, 
the following diagram commutes:
$$
\xymatrix{\mathbb{D}^{\mathrm{ur}}(\ZZ_p (\omega^a) \otimes \mathbb{Z}_p ( b ) \otimes \Lambda^\sharp_{\mathrm{cyc}})\ar[d]_\kappa 
\ar[rrr]^{\mathrm{EXP}^{\mathrm{ur},(a,b)}_{\ZZ_p (\omega^a) \otimes \mathbb{Z}_p ( b ) \otimes \Lambda^\sharp_{\mathrm{cyc}}}}&&& 
\dfrac{H^1 (\QQ_p^{\mathrm{ur}},\ZZ_p (\omega^a) \otimes \mathbb{Z}_p ( b ) \otimes \Lambda^\sharp_{\mathrm{cyc}} )}
{\ZZ_p (\omega^a) \otimes \mathbb{Z}_p ( b )}
\ar[d]^\kappa\\
D^{\mathrm{ur}}_{\textup{dR}}(\mathbb{Q}_p(b)\otimes E_\kappa (\omega^a \kappa ))\ar[rrr]_{
\exp^{\mathrm{ur}}_{\mathbb{Q}_p (b)\otimes E_\kappa (\omega^a \kappa ) }
 \,\circ \,  e^{\mathrm{ur},+}_p 
\hspace*{10pt} }&&& 
 H^1(\QQ_p^{\mathrm{ur}},\mathbb{Q}_p(b)\otimes E_\kappa (\omega^a \kappa ) )
 }
$$
Here $e^{\mathrm{ur},+}_p:=(-1)^{w_\cyc (\kappa)-1}(w_\cyc (\kappa)+b-1)!\,e^{\mathrm{ur}}_{p}$ and 
$e^{\mathrm{ur}}_{p}=e^{\mathrm{ur}}_{p}(\mathbb{Q}_p(b)\otimes E_\kappa (\omega^a \kappa ))$ is the $p$-adic multiplier given by
$$
e^{\mathrm{ur}}_{p}:=\begin{cases} 
\left( 
1 - p^{-1}\varphi^{-1}
\right) \left( 
1 - \varphi
\right)^{-1}  &  \text{when $E_\kappa (\omega^a \kappa )$ is crystalline} \vspace*{5pt} \\ 
\left( p^{-1} \varphi \right)^n & 
\text{when } E_\kappa (\omega^a \kappa ) \cong \QQ_p (w_\cyc (\kappa)) \otimes E_\kappa(\phi) \text{ with }\\
& \textup{ord}_p(\textup{cond}(\phi))=n\geq1.
 \\ 
\end{cases}
$$
Also, for every arithmetic specialization 
$\kappa$ on $\Lambda_{\mathrm{cyc}}$ of weight $w_\cyc (\kappa ) < 1-b$, 
the following diagram commutes:
$$
\xymatrix{\mathbb{D}^{\mathrm{ur}}(\ZZ_p (\omega^a) \otimes \mathbb{Z}_p ( b ) \otimes \Lambda^\sharp_{\mathrm{cyc}})\ar[d]_\kappa 
\ar[rrr]^(.47){\mathrm{EXP}^{\mathrm{ur},(a,b)}_{\ZZ_p (\omega^a) \otimes \mathbb{Z}_p ( b ) \otimes \Lambda^\sharp_{\mathrm{cyc}}}}&&& 
\dfrac{H^1 (\QQ_p^{\mathrm{ur}},\ZZ_p (\omega^a) \otimes \mathbb{Z}_p ( b ) \otimes \Lambda^\sharp_{\mathrm{cyc}})}
{\ZZ_p (\omega^a) \otimes \mathbb{Z}_p ( b )}
\ar[d]^\kappa\\
D^{\mathrm{ur}}_{\textup{dR}}(\mathbb{Q}_p(b)\otimes E_\kappa (\omega^a \kappa ) )\ar[rrr]_(.48){
 (\exp^{\mathrm{ur},\ast}_{\mathbb{Q}_p(b)\otimes E_\kappa (\omega^a \kappa ) })^{-1}
 \,\circ \,  e^{\mathrm{ur},-}_p
 }&&& 
 H^1(\QQ_p^{\mathrm{ur}},\mathbb{Q}_p(b)\otimes E_\kappa (\omega^a \kappa ))
 }
$$
where $e^{\mathrm{ur},-}_p:=\dfrac{ e^{\mathrm{ur}}_p}{(-w_\cyc (\kappa) -b)!}$\,. 
\end{thm}
\begin{proof}
On twisting the diagrams in the statement of Theorem~\ref{thm:interpolationclassic_unr} (that characterizes the map $\mathrm{EXP}^{\mathrm{ur}}_{\ZZ_p [[G_{\mathrm{cyc}}]]^\sharp}$)  by the character $(\chi_\cyc \omega )^b$, we obtain a map $\mathrm{EXP}^{\mathrm{ur}}_{\ZZ_p(b)\otimes\ZZ_p [[G_{\mathrm{cyc}}]]^\sharp}$ that is characterized by an interpolation diagram that is the twisted version of those appear Theorem \ref{thm:interpolationclassic_unr}. Recall that $\ZZ_p [[G_\cyc]] \cong \Lambda_\cyc [\Delta ]$ where $\Delta$ stands for $\mathrm{Gal } (\QQ_p (\mu_p) ) /\QQ_p )
\cong (\ZZ_p /p \ZZ )^\times$. We define the map 
$\mathrm{EXP}^{\mathrm{ur},(a,b)}_{\ZZ_p (\omega^a) \otimes \mathbb{Z}_p ( b ) \otimes \Lambda^\sharp_{\mathrm{cyc}}}$ as the restriction of $\mathrm{EXP}^{\mathrm{ur}}_{\ZZ_p(b)\otimes\ZZ_p [[G_{\mathrm{cyc}}]]^\sharp}$ to the $\omega^{a-b}$-isotypic components. This map has the desired interpolation properties by construction. 
\end{proof}
In order to formulate a further generalization of Theorem~\ref{perrinriou3}, we introduce some notation. 
\begin{define} 
\label{define:thefatdieudonnemodule}
Let $\mathcal{R}$ be a local domain which is finite flat over 
$\mathbb{Z}_p [[\Gamma_1 \times \cdots \times \Gamma_r]]$ and let $(\TT,\mathcal{R} ,\mathcal{S})$ be a deformation datum. 
Suppose that we have a strictly decreasing, 
$G_{\QQ_p}$-stable, exhaustive filtration $\{\mathrm{Fil}_\fp^i\TT\}_{i \in \mathbb{Z}}$ satisfying 
the conditions (Ord) of Definition \ref{definition:ord12}. 
For every integer $m\in \{ 0,1,\ldots ,d-1\}$, recall that $\mathrm{Gr}_p^m\TTc$ is the free $\mathcal{R}_\cyc$-module of rank one on which $G_{\QQ_p}$ acts via the character $\widetilde{\chi}_1^{e_{m,1}}\cdots \widetilde{\chi}_r^{e_{m,r}}\widetilde{\chi}_\cyc^{-1} \omega^{1-a_m}\chi^{1-b_m}_{\mathrm{cyc}} \widetilde{\alpha}_m$. We define the free $\mathcal{R}_\cyc$-module $\mathbb{D}((\mathrm{Gr}_\fp^{m}\TTc )^{\RRc}(1))$ of rank one as the tensor product
\begin{align*}
\bigg(\left(\underset{i\in A_m}{\widehat{\otimes}}\,\mathbb{D}(-e_{m,i})\right)
 \underset{\mathbb{Z}_p}{\widehat{\otimes}} \mathbb{D}(1)
 \underset{\mathbb{Z}_p}{\otimes} 
 D_{\textup{dR}}(\ZZ_p(\omega^{a_m}))
 &
  \underset{\mathbb{Z}_p}{\otimes} D_{\textup{crys}}(\ZZ_p(b_m))\bigg)
  \\
& \underset{\ZZ_p[[\Gamma_1\times\cdots\times\Gamma_r]]}\otimes
 \left(\mathcal{R}(\widetilde{\alpha}^{-1}_{m})\widehat{\otimes}_{\ZZ_p}\widehat{\ZZ}_p^{\textup{ur}}\right)^{G_{\QQ_p}}
\, ,
\end{align*} 
where $A_m$ is the subset of $\{ 0, 1,\ldots , d-1\}$ which consists of $i \in \{ 0, 1, \ldots ,d-1\}$ such that $e_{m,i} =1$.  
We define $\mathbb{D}(-1)$ to be is a $\LL_\cyc$-linear dual of of $\mathbb{D}(1)$ and 
the completed tensor product $\underset{i\in A_m}{\widehat{\otimes}}\,\mathbb{D}(-e_{m,i})$ is calculated over $\ZZ_p$ and it is endowed with the $\ZZ_p[[\Gamma_1\times\cdots\times\Gamma_r]]$-module structure via the characters $\{(\chi_\cyc)^{-1}\circ\chi_i\}_i$. 
\par 
We also set
$$
\mathbb{D}^{\mathrm{ur}}(\ZZ_p (\omega^a) \otimes \mathbb{Z}_p ( b ) \otimes \Lambda^\sharp_{\mathrm{cyc}})
 := \mathbb{D} ((\mathrm{Gr}_\fp^{m}\TTc)^{\RRc}(1))\widehat{\otimes}_{\ZZ_p}
 \widehat{\ZZ}^{\mathrm{ur}}_p. 
$$
\end{define}
The following theorem is deduced from Theorem~\ref{perrinriou3} by a coordinate change trick that was introduced in \cite{ochiai-AJM14} (particularly, see the end of \S 6 of loc. cit.). 
\begin{thm}\label{thm:interpolationforcoleman_unr}
Let $\mathcal{R}$ be a local domain which is finite flat over 
$\mathbb{Z}_p [[\Gamma_1 \times \cdots \times \Gamma_r]]$ and let $(\TT,\mathcal{R} ,\mathcal{S})$ be a deformation datum. 
Suppose that we have a strictly decreasing, 
$G_{\QQ_p}$-stable, exhaustive filtration $\{\mathrm{Fil}_\fp^i\TT \}_{i \in \mathbb{Z}}$ satisfying 
the condition (Ord) of Definition \ref{definition:ord12}. 
Let $m$ be an integer in $\{0,1, \ldots , d-1\} $. 
\par 
Then, we have an $\RRc$-linear isomorphism  
$$
\mathrm{EXP}^{\mathrm{ur}}_{(\mathrm{Gr}_\fp^{m}\TTc )^{\RRc}(1)} \,:\, \mathbb{D}^{\mathrm{ur}}((\mathrm{Gr}_\fp^{m}\TTc)^{\RRc}(1))  \lra 
\dfrac{H^1(\QQ_p^{\mathrm{ur}},(\mathrm{Gr}_\fp^{m}\TTc )^{\RRc}(1))}{(\mathrm{Gr}_\fp^{m}\TT)^{\RRn}(1)} 
$$ 
such that, for every $\kappa \in \mathcal{S}^{(m),+}_\cyc$ the following diagram commutes:
$$\xymatrix{\mathbb{D}^{\mathrm{ur}}((\mathrm{Gr}_\fp^{m}\TTc )^{\RRc}(1))\ar[d]_\kappa \ \ \ \ \ \ar[rrr]^(.47){\mathrm{EXP}^{\mathrm{ur}}_{(\mathrm{Gr}_\fp^{m}\TTc )^{\RRc}(1)}} &&& 
\ \ \ 
\dfrac{H^1(\QQ_p^{\mathrm{ur}},(\mathrm{Gr}_\fp^{m}\TTc )^{\RRc}(1))}{(\mathrm{Gr}_\fp^{m}\TT)^{\RRn}(1)} \ar[d]^\kappa\\
D^{\mathrm{ur}}_{\textup{dR}}((\mathrm{Gr}_\fp^{m}V_\kappa)^{\kappa (\RRc)}(1))\ \ \ \ \ \ar[rrr]_(.42) 
{\ \ \ 
\exp^{\mathrm{ur}}_{(\mathrm{Gr}_\fp^{m}V_\kappa)^{\kappa (\RRc )}(1)}
 \,\circ \,  e^{\mathrm{ur},+}_p
}
&&& \ \ \ H^1(\QQ_p^{\mathrm{ur}},(\mathrm{Gr}_\fp^{m}V_\kappa)^{\kappa (\RRc)}(1))
}
$$
Here $e^{\mathrm{ur},+}_p:=(-1)^{d_m(\kappa)-1}(d_m(\kappa)-1)!\,e^{\mathrm{ur}}_{p}$ and $e^{\mathrm{ur}}_{p}=e^{\mathrm{ur}}_{p}((\mathrm{Gr}_p^{m}V_\kappa)^{\kappa (\RRc)}(1))$ is the $p$-adic multiplier given by
$$
e^{\mathrm{ur}}_{p}:=\begin{cases} 
\left( 
1 - p^{-1}\varphi^{-1}
\right) \left( 
1 - \varphi
\right)^{-1}  &  \text{when $\mathrm{Gr}^{m}_p V_\kappa$ is crystalline,} \vspace*{5pt} \\ 
\left( p^{-1}\varphi^{-1} \right)^n & 
\text{when }  \mathrm{Gr}^{m}_p V_\kappa\vert_{I_p} \cong E_\kappa(c_m(\kappa))(\phi) \text{ with }\\
& \textup{ord}_p(\textup{cond}(\phi))=n\geq1.  \\ 
\end{cases}
$$
Also, for every $\kappa \in \mathcal{S}^{(m),-}_\cyc$ we also have the following commutative diagram:
$$
\xymatrix{\mathbb{D}^{\mathrm{ur}}((\mathrm{Gr}_\fp^{m}\TTc )^{\RRc}(1))\ar[d]_\kappa \ar[rrr]^(.48){\mathrm{EXP}^{\mathrm{ur}}_{(\mathrm{Gr}^{m}\TTc)^{\RRc}(1)}}&&& 
\dfrac{H^1(\QQ_p^{\mathrm{ur}},(\mathrm{Gr}_p^{m}\TTc)^{\RRc}(1))\ar[d]^\kappa}{{(\mathrm{Gr}_\fp^{m}\TT)^{\RRn}(1)}}\\
D^{\mathrm{ur}}_{\textup{dR}}((\mathrm{Gr}_\fp^{m}V_\kappa)^{\kappa (\RRc)}(1))
\ar[rrr]_(.45){\hspace*{10pt}
\left(\exp^{\mathrm{ur},*}_{\mathrm{Gr}^{m}_p V_\kappa}\right)^{-1}
 \,\circ \,  e^{\mathrm{ur},-}_p
} 
&&& H^1(\QQ_p^{\mathrm{ur}},(\mathrm{Gr}_p^{m}V_\kappa)^{\kappa (\RRc)}(1))
}$$
where $ e^{\mathrm{ur},-}_p:=\dfrac{ e^{\mathrm{ur}}_p}{(-d_m(\kappa))!}$\,.
\end{thm}

\begin{proof}
Recall that $\RRc$ is finite flat over $\ZZ_p [[\Gamma_1 \times \cdots \times \Gamma_r \times \Gamma_{\mathrm{cyc}}]]$ 
and take another set of groups $\Gamma'_1 ,\ldots ,\Gamma_r' $ and 
$\Gamma'_\cyc$ such that 
\begin{enumerate}
\item[(i)] 
We have an isomorphism $\Gamma_1 \times \cdots \times \Gamma_r \times \Gamma_{\mathrm{cyc}} \overset{\sim}{\lra}
\Gamma'_1 \times \cdots \times \Gamma'_r \times \Gamma'_{\mathrm{cyc}}$. 
\item[(ii)] 
The group $\Gamma'_i$ has an isomorphism $\chi'_i : \Gamma'_i \overset{\sim}{\lra} 1+p\ZZ_p$ ($i=1,\ldots ,r$) and 
$\Gamma'_{\mathrm{cyc}}$ has an isomorphism $\chi'_\cyc : \Gamma'_{\mathrm{cyc}} \overset{\sim}{\lra} 1+p\ZZ_p$. 
\end{enumerate}
We call such a set of groups $\{\Gamma'_1 ,\ldots ,\Gamma'_r ,\Gamma'_\cyc \}$ a coordinate change of 
$\{ \Gamma_1 ,\ldots ,\Gamma_r , \Gamma_\cyc \}$. For a given coordinate change $\{\Gamma'_1 ,\ldots ,\Gamma'_r ,\Gamma'_\cyc \}$, we define 
a Galois character $\widetilde{\chi}'_i : G_{\QQ_p} \lra \ZZ_p [[\Gamma'_i]]^\times$ to be 
$$
G_{\QQ_p} \twoheadrightarrow \Gamma_{\mathrm{cyc}} \xrightarrow[(\chi'_i )^{-1}\circ \chi_\cyc]{\sim} 
\Gamma'_i \hookrightarrow \ZZ_p [[\Gamma'_i]]^\times.  
$$ 
We define a Galois character $\widetilde{\chi}'_\cyc : G_{\QQ_p} \lra \ZZ_p [[\Gamma'_\cyc]]^\times$ in the same way.  
\par
The crucial observation is that for each $m \in \{ 0,1,\ldots , d-1\}$, there exists a coordinate change such that the action of $G_{\QQ_p}$
on $(\mathrm{Gr}_\fp^{m}\TTc )^{\RRc}(1)$ is given by $\widetilde{\chi}'_\cyc  \omega^{a}  \chi^{b}_\cyc \widetilde{\alpha}^{-1}_{m}$. 
\par 
Let us identify the cyclotomic Iwasawa algebra $\LL_\cyc$ of Theorem~\ref{perrinriou3} and $\ZZ_p [[\Gamma'_\cyc]]$ here. 
Then the commutative diagrams of Theorem~\ref{thm:interpolationforcoleman_unr} are obtained by taking 
base extension functor $\left( - \otimes_{\ZZ_p [[\Gamma'_\cyc]]} \RRc \right)$ to the commutative diagrams of Theorem \ref{perrinriou3} and by twisting 
by unramified character $\widetilde{\alpha}^{-1}_m : G_{\QQ_p} \lra \RRn^\times$. The proof follows. 
\end{proof}
The following result is the main theorem of this section, which is deduced on passing to $\mathrm{Gal} (\QQ^{\mathrm{ur}}_p /\QQ_p)$-invariants in Theorem~\ref{thm:interpolationforcoleman_unr}. 
\begin{thm}\label{thm:interpolationforcoleman}
Let $\mathcal{R}$ be a local domain which is finite flat over 
$\mathbb{Z}_p [[\Gamma_1 \times \cdots \times \Gamma_r]]$ and let $(\TT,\mathcal{R} ,\mathcal{S})$ be a deformation datum. 
Suppose that we have a strictly decreasing, 
$G_{\QQ_p}$-stable, exhaustive filtration $\{\mathrm{Fil}_p^i\TT\}_{i \in \mathbb{Z}}$ satisfying 
the conditions (Ord) of Definition \ref{definition:ord12}. 
Let $m$ be an integer in $\{0,1, \ldots , d-1\} $ and assume that the unramified character $\widetilde{\alpha}_m$ is non-trivial. 
\par 
Then, we have an $\RRc$-linear isomorphism 
$$
\mathrm{EXP}_{(\mathrm{Gr}_p^{m}\TTc)^{\RRc}(1)} \,:\, 
\mathbb{D}((\mathrm{Gr}_p^{m}\TTc)^{\RRc}(1))  \lra H^1(\QQ_p,(\mathrm{Gr}_p^{m}\TTc)
^{\RRc}(1)) $$ 
such that, for every $\kappa \in \mathcal{S}^{(m),+}_\cyc$ the following diagram commutes:
$$
\xymatrix{\mathbb{D}((\mathrm{Gr}_p^{m}\TTc)^{\RRc}(1))\ar[d]_\kappa \ar[rrr]^{\mathrm{EXP}_{(\mathrm{Gr}^{m}_p\TTc)^{\RRc}(1)}}&&& 
H^1(\QQ_p,(\mathrm{Gr}_p^{m}\TTc)^{\RRc}(1))\ar[d]^\kappa \\ 
D_{\textup{dR}}((\mathrm{Gr}_p^{m}V_\kappa)^{\kappa (\RRc)}(1))
\ar[rrr]_{\hspace*{10pt}e_p^+\,\times\, 
\exp_{(\mathrm{Gr}^{m}_p V_\kappa)^{\kappa (\RRc)}(1)}} 
&&& H^1(\QQ_p, (\mathrm{Gr}_p^{m}V_\kappa)^{\kappa (\RRc)}(1))
}$$
Here $e_p^+:=(-1)^{d_m(\kappa)-1}(d_m(\kappa)-1)!\,e_p$ and $e_p=e_p((\mathrm{Gr}_p^{m}V_\kappa)^{\kappa (\RRc)}(1))$ is the $p$-adic multiplier given by
$$
e_{p}:=\begin{cases} 
\left( 
1 - \dfrac{p^{d_m(\kappa)-1}}{\kappa\vert_{\mathcal{R}} (\widetilde{\alpha}_m (\mathrm{Frob}_p))}
\right) \left( 
1 - \dfrac{\kappa\vert_{\mathcal{R}} (\widetilde{\alpha}_m (\mathrm{Frob}_p))}{p^{d_m(\kappa)}}
\right)^{-1}  &  \text{when $\mathrm{Gr}^{m}_p V_\kappa$ is crystalline,} \vspace*{5pt} \\ 
\left( \dfrac{p^{d_m(\kappa)-1}}{\kappa\vert_{\mathcal{R}} (\widetilde{\alpha}_m (\mathrm{Frob}_p))}\right)^n & 
\text{when }  \mathrm{Gr}^{m}_p V_\kappa\vert_{I_p} \cong E_\kappa(c_m(\kappa))(\phi) \\
& \text{ with }\textup{ord}_p(\textup{cond}(\phi))=n\geq1.\\ 
\end{cases}
$$
Also, for every $\kappa \in \mathcal{S}^{(m),-}_\cyc$ we also have the following commutative diagram:
$$
\xymatrix{\mathbb{D}((\mathrm{Gr}_p^{m}\TTc)^{\RRc}(1))\ar[d]_\kappa \ar[rrr]^(.48){\mathrm{EXP}_{(\mathrm{Gr}^{m}\TTc)^{\RRc}(1)}}&&& 
H^1(\QQ_p,(\mathrm{Gr}_p^{m}\TTc)^{\RRc}(1))\ar[d]^\kappa\\
D_{\textup{dR}}((\mathrm{Gr}_p^{m}V_\kappa)^{\kappa (\RRc)}(1))
\ar[rrr]_(.45){\hspace*{10pt}e_p^-\,\times\, 
\left(\exp^*_{\mathrm{Gr}^{m}_p V_\kappa}\right)^{-1}} 
&&& H^1(\QQ_p,(\mathrm{Gr}_p^{m}V_\kappa)^{\kappa (\RRc)}(1))
}$$
where $e_p^-:=\dfrac{e_p}{(-d_m(\kappa))!}$\,.
\end{thm}
\begin{proof}
We begin with a study of the ${\mathrm{Gal} ( \QQ_p^{\mathrm{ur}} / \QQ_p )}$-invariants of the modules that appear in Theorem~\ref{thm:interpolationforcoleman_unr}. 
\par 
By definition, we have the following identification by definition: 
\begin{equation}\label{equation:invarian_part_D}
\mathbb{D}((\mathrm{Gr}_\fp^{m}\TTc )^{\RRc}(1))= 
\mathbb{D}^{\mathrm{ur}}((\mathrm{Gr}_\fp^{m}\TTc )^{\RRc }(1))
^{\mathrm{Gal} ( \QQ_p^{\mathrm{ur}} / \QQ_p )}\,.
\end{equation}
Let us also calculate the ${\mathrm{Gal} ( \QQ_p^{\mathrm{ur}} / \QQ_p )}$-invariants in the Galois cohomology side (in the diagrams of Theorem~\ref{thm:interpolationforcoleman_unr}). 
By our requirement that $\widetilde{\alpha}_m$ be non-trivial, we have 
\begin{equation}\label{equation:invarian_part_H}
\left( \dfrac{H^1(\QQ_p^{\mathrm{ur}},(\mathrm{Gr}_\fp^{m}\TTc )^{\RRc}(1))}{(\mathrm{Gr}_\fp^{m}\TT)^{\RRn}(1)} \right)^{{\mathrm{Gal} ( \QQ_p^{\mathrm{ur}} / \QQ_p )}} 
= H^1(\QQ_p^{\mathrm{ur}},(\mathrm{Gr}_\fp^{m}\TTc )^{\RRc}(1))^{{\mathrm{Gal} ( \QQ_p^{\mathrm{ur}} / \QQ_p )}} 
\end{equation}
Let us calculate the right-hand side. 
The quotient $(\mathrm{Gr}_\fp^{m}\TTc )^{\RRc }(1)/\mm^r_{\RRc} (\mathrm{Gr}_\fp^{m}\TTc )^{\RRc }(1)$ 
is a finite abelian group with continuous action of $G_{\QQ_p}$ for any natural number $r$. 
Consider the restriction map 
\begin{multline*}
H^1\bigl( \QQ_p,(\mathrm{Gr}_\fp^{m}\TTc )^{\RRc }(1)/\mm^r_{\RRc } (\mathrm{Gr}_\fp^{m}\TTc )^{\RRc }(1)\bigr) 
\\ \longrightarrow 
H^1\bigl( \QQ_p^{\mathrm{ur}},(\mathrm{Gr}_\fp^{m}\TTc )^{\RRc }(1)/\mm^r_{\RRc } (\mathrm{Gr}_\fp^{m}\TTc )^{\RRc }(1)\bigr)^{G_{\QQ_p}}. 
\end{multline*}
By the inflation-restriction sequence, the kernel and the cokernel of this map are the modules  
$H^1 (\QQ_p^{\mathrm{ur}} /\QQ_p , A^{G_{\QQ_p^{\mathrm{ur}}}}_r)$ 
and $H^2 (\QQ_p^{\mathrm{ur}} /\QQ_p , A^{G_{\QQ_p^{\mathrm{ur}}}}_r)$ respectively, where  
we have set 
$$A_r:= (\mathrm{Gr}_\fp^{m}\TTc )^{\RRc }(1)/\mm^r_{\RRc } (\mathrm{Gr}_\fp^{m}\TTc )^{\RRc }(1)$$ 
to ease our notation. Since $\mathrm{Gal} (\QQ_p^{\mathrm{ur}} /\QQ_p)\cong\widehat{\mathbb{Z}}$ has cohomological dimension 
one, we have $H^2 (\QQ_p^{\mathrm{ur}}/\QQ_p , A^{G_{\QQ_p^{\mathrm{ur}}}}_r)=0$ and 
$H^1 (\QQ_p^{\mathrm{ur}} /\QQ_p , A^{G_{\QQ_p^{\mathrm{ur}}}}_r)$ is isomorphic to 
the largest $\mathrm{Gal} (\QQ_p^{\mathrm{ur}} /\QQ_p)$-coinvariant quotient 
$\left(A^{G_{\QQ_p^{\mathrm{ur}}}}_r\right)_{\mathrm{Gal} (\QQ_p^{\mathrm{ur}} /\QQ_p)}$ of $A^{G_{\QQ_p^{\mathrm{ur}}}}_r$. 
Since we have 
$$
\left( (\mathrm{Gr}_\fp^{m}\TTc )^{\RRc }(1) \right)^{G_{\QQ_p^{\mathrm{ur}}}}=0$$ 
thanks to our running hypothesis that the unramified character $\widetilde{\alpha}_m$ (that appears in the formulation of (Ord)) is non-trivial, it follows that $\varprojlim_r A^{G_{\QQ_p^{\mathrm{ur}}}}=0$. We therefore infer that 
$$\varprojlim_r H^1 (\QQ_p^{\mathrm{ur}} /\QQ_p , A^{G_{\QQ_p^{\mathrm{ur}}}}_r)=0$$
and that we have an isomorphism 
\begin{multline*}
H^1 \bigl( \QQ_p ,(\mathrm{Gr}_\fp^{m}\TTc )^{\RRc }(1)/\mm^r_{\RRc } (\mathrm{Gr}_\fp^{m}\TTc )^{\RRc }(1)\bigr) 
\\ \cong 
H^1\bigl( \QQ_p^{\mathrm{ur}},(\mathrm{Gr}_\fp^{m}\TTc )^{\RRc }(1)/\mm^r_{\RRc } (\mathrm{Gr}_\fp^{m}\TTc )^{\RRc}(1)\bigr)^{G_{\QQ_p}}
\end{multline*}
induced by the restriction map. Taking the inverse limit with respect to $r$, we have 
\begin{equation}\label{equation:invarian_part_H1}
H^1 \bigl( \QQ_p ,(\mathrm{Gr}_\fp^{m}\TTc )^{\RRc }(1)\bigr) 
\\ \cong 
H^1\bigl( \QQ_p^{\mathrm{ur}},(\mathrm{Gr}_\fp^{m}\TTc )^{\RRc }(1) \bigr)^{G_{\QQ_p}}
\end{equation}
Let us define $\mathrm{EXP}_{(\mathrm{Gr}_\fp^{m}\TTc )^{\RRc }(1)}$ to be the map induced by 
$\mathrm{EXP}^{\mathrm{ur}}_{(\mathrm{Gr}_\fp^{m}\TT)^{\RRc}(1)}$ on the ${\mathrm{Gal} ( \QQ_p^{\mathrm{ur}} / \QQ_p )}$-invariants.  
Thanks to \eqref{equation:invarian_part_D}, \eqref{equation:invarian_part_H} and \eqref{equation:invarian_part_H1}, we have an $\RRc$-linear isomorphism 
$$
\mathrm{EXP}_{(\mathrm{Gr}_\fp^{m}\TTc )^{\RRc }(1)}\,:\, \mathbb{D}((\mathrm{Gr}_\fp^{m}\TTc )^{\RRc }(1))  \lra H^1(F,(\mathrm{Gr}_\fp^{m}\TTc )^{\RRc }(1))\, 
.
$$ By its very construction, the map $\mathrm{EXP}_{(\mathrm{Gr}_\fp^{m}\TTc )^{\RRc }(1)}$ verifies the desired interpolation property for every $\kappa \in \mathcal{S}^{(m),+}_\cyc \cup \mathcal{S}^{(m),-}$. 
\end{proof}
\par 
Let us define 
$$
\mathbb{D}(\mathrm{Gr}_p^{m}\TTc ) := \Hom_{\RRc} (\mathbb{D}(\mathrm{Gr}_\fp^{m}\TTc )^{\RRc }(1)) ,\RRc )\,. 
$$ 
The following result is an important consequence of Theorem~\ref{thm:interpolationforcoleman}. 
\begin{cor}\label{cor:dualexponentialmapforH++} 
Let $\mathcal{R}$ be a local domain which is finite flat over 
$\mathbb{Z}_p [[\Gamma_1 \times \cdots \times \Gamma_r]]$ and let $(\TT,\mathcal{R} ,\mathcal{S})$ be a deformation datum. 
Suppose that we have a strictly decreasing, 
$G_{\QQ_p}$-stable, exhaustive filtration $\{\mathrm{Fil}_p^i\TT\}_{i \in \mathbb{Z}}$ satisfying 
the conditions (Ord) of Definition \ref{definition:ord12}. 
Let $m$ be an integer in $\{0,1, \ldots , d-1\} $ and assume that the unramified character $\widetilde{\alpha}_m$ is non-trivial.  
\par 
Then, we have an $\RRc$-linear isomorphism 
$$
\mathrm{EXP}^\ast_{(\mathrm{Gr}_p^{m}\TTc)^{\RRc}(1)} \,:\, 
H^1(\QQ_p,\mathrm{Gr}_p^{m}\TTc )
\lra 
\mathbb{D}(\mathrm{Gr}_p^{m}\TTc )  
$$ 
such that, for every $\kappa \in \mathcal{S}^{(m),+}_\cyc$ the following diagram commutes:
$$
\xymatrix{
H^1(\QQ_p,\mathrm{Gr}_p^{m}\TTc)
\ar[d]_\kappa \ar[rrr]^{\mathrm{EXP}^\ast_{(\mathrm{Gr}^{m}_p\TTc)^{\RRc}(1)}} &&& 
\mathbb{D}(\mathrm{Gr}_p^{m}\TTc )
\ar[d]^\kappa \\ 
H^1(\QQ_p, \mathrm{Gr}_p^{m}V_\kappa )
\ar[rrr]_{\hspace*{10pt}e_p^+\,\times\, 
\exp^\ast_{(\mathrm{Gr}^{m}_p V_\kappa)^{\kappa (\RRc)}(1)}} 
&&& 
D_{\textup{dR}}(\mathrm{Gr}_p^{m}V_\kappa )
}$$
Here $e_p^+:=(-1)^{d_m(\kappa)-1}(d_m(\kappa)-1)!\,e_p$ and $e_p=e_p((\mathrm{Gr}_p^{m}V_\kappa)^{\kappa (\RRc)}(1))$ is the $p$-adic multiplier given by
$$
e_{p}:=\begin{cases} 
\left( 
1 - \dfrac{p^{d_m(\kappa)-1}}{\kappa\vert_{\mathcal{R}} (\widetilde{\alpha}_m (\mathrm{Frob}_p))}
\right) \left( 
1 - \dfrac{\kappa\vert_{\mathcal{R}} (\widetilde{\alpha}_m (\mathrm{Frob}_p))}{p^{d_m(\kappa)}}
\right)^{-1}  &  \text{when $\mathrm{Gr}^{m}_p V_\kappa$ is crystalline,} \vspace*{5pt} \\ 
\left( \dfrac{p^{d_m(\kappa)-1}}{\kappa\vert_{\mathcal{R}} (\widetilde{\alpha}_m (\mathrm{Frob}_p))}\right)^n & 
\text{when }  \mathrm{Gr}^{m}_p V_\kappa\vert_{I_p} \cong E_\kappa(c_m(\kappa))(\phi) \\
& \text{ with }\textup{ord}_p(\textup{cond}(\phi))=n\geq1.\\ 
\end{cases}
$$
Also, for every $\kappa \in \mathcal{S}^{(m),-}_\cyc$ we also have the following commutative diagram:
$$
\xymatrix{
H^1(\QQ_p,\mathrm{Gr}_p^{m}\TTc )
\ar[d]_\kappa \ar[rrr]^(.48){\mathrm{EXP}^\ast_{(\mathrm{Gr}^{m}\TTc)^{\RRc}(1)}}&&& 
\mathbb{D}(\mathrm{Gr}_p^{m}\TTc )
\ar[d]^\kappa\\
H^1(\QQ_p,\mathrm{Gr}_p^{m}V_\kappa )
\ar[rrr]_(.45){\hspace*{15pt}e_p^-\,\times\, 
\log_{(\mathrm{Gr}^{m}_p V_\kappa)^{\kappa (\RRc)}(1)}
} 
&&& 
D_{\textup{dR}}(\mathrm{Gr}_p^{m}V_\kappa )
}$$
where $e_p^-:=\dfrac{e_p}{(-d_m(\kappa))!}$\,.
\end{cor}
\begin{proof}
We define $\mathrm{EXP}^\ast_{(\mathrm{Gr}_p^{m}\TTc)^{\RRc}(1)}$ to be $\RRc$-linear Kummer dual of 
the big exponential map $\mathrm{EXP}_{(\mathrm{Gr}_p^{m}\TTc)^{\RRc}(1)}$. 
Note that we have 
$$
H^1(\QQ_p,\mathrm{Gr}_p^{m}\TTc ) \cong \Hom_{\RRc} (H^1(\QQ_p,(\mathrm{Gr}_p^{m}\TTc)^{\RRc}(1) ),\RRc ) 
$$
by local Tate duality theorem of Galois cohomology. 
Recall that, for any de Rham $p$-adic representation $V$ of $G_{\QQ_p}$, the Kummer dual 
of $\mathrm{exp}_V$ (resp. $(\mathrm{exp}^\ast_V)^{-1}$) is known to be $\mathrm{exp}^\ast_V$ (resp. $\mathrm{log}_V$). 
These facts ensure that $\mathrm{EXP}^\ast_{(\mathrm{Gr}_p^{m}\TTc)^{\RRc}(1)}$ satisfies the desired interpolation property 
and completes the proof. 
\end{proof}

\begin{rem} (For readers who might be distressed about the absence of Gauss sum in the interpolation formula of Coleman map)
\label{rem:comparetorubinwithGausssums} 
In the most basic set up with the cyclotomic deformation of the $p$-adic Tate module of an ordinary elliptic curve $E$, Rubin in \cite[Prop. A.2]{r98} presents the following interpolation formula for any $n \geq m+1$:
\begin{equation}\label{equation:rubin'sarticle_gauss_sum}
\chi (\mathrm{Col}_n (z)) = \alpha^{-m}\tau (\chi ) \sum_{\gamma \in G_n} \chi^{-1} (\gamma ) \mathrm{exp}^\ast_{\omega_E} (z^\gamma) .
\end{equation}
Here $\chi $ is a nontrivial Dirichlet character of conductor $p^m$, $\tau (\chi)$ is the Gauss sum for $\chi$, and 
$\alpha$ is the $p$-unit root of the $p$-Euler polynomial for $E$. 
The group $G_n$ in the summation above is nothing but the group $\Gamma_{\mathrm{cyc}}/\Gamma^{p^n}_{\mathrm{cyc}}$ in the current article.  The map $\mathrm{Col}_n$ above is a map from $H^1 (\mathbb{Q}_p , T_p (E) \otimes \mathbb{Z}_p [G_n]^\sharp )
\cong H^1 (\mathbb{Q}_{p,n} , T_p (E) )
$ 
where $\mathbb{Q}_{p,n} $ is the $n$-th layer of the local cyclotomic $\mathbb{Z}_p$-extention 
$\mathbb{Q}_{p,\infty} /\mathbb{Q}_p$. 
\par 
 {Let $\chi$ be a character of $G_n$. Recall that $\mathcal{O}(\chi)$ is a free $\mathbb{Z}_p [\chi]$-module of rank one on which $G_{\mathbb{Q}_p}$ acts by the character $\chi$. The character $\chi$ induces a $G_{\mathbb{Q}_p}$-equivariant map $\mathbb{Z}_p [G_n]^\sharp \rightarrow \mathcal{O}(\chi)$ and we have the following identity by direct calculation:
\begin{equation}\label{equation:rubin'sarticle_gauss_sum2}
\tau (\chi )
\sum_{\gamma \in G_n} \chi^{-1} (\gamma) z^\gamma  =  \chi (z).
\end{equation}
For each $z \in H^1 (\mathbb{Q}_p , T_p (E) \otimes_{\mathbb{Z}_p} \mathbb{Z}_p [G_n ]^\sharp )$, 
we also write $\chi(z)$  for the image of $z$ under the map 
\be\label{eqn:twistedprojectionmap} H^1 (\mathbb{Q}_p , T_p (E) \otimes \mathbb{Z}_p [G_n]^\sharp ) \longrightarrow 
H^1 (\mathbb{Q}_p , T_p (E) \otimes \mathcal{O}(\chi))
\ee
for any character $\chi$ of $G_n$. 
It therefore follows from \eqref{equation:rubin'sarticle_gauss_sum} and\eqref{equation:rubin'sarticle_gauss_sum2} that 
$$
\chi (\mathrm{Col}_n (z)) = \alpha^{-m} \mathrm{exp}^\ast_{\omega_E} (\chi^{-1}(z)) .
$$   
This perfectly matches up with our interpolation formulae: In the setting of \cite{r98}, note that the only possibility that the integer $j$ in the interpolation formulae in Theorem~\ref{thm:interpolationforcoleman} can assume is the value $1$. In other words, Gauss sums are not really missing in our formulae, but rather encoded in the projection maps (\ref{eqn:twistedprojectionmap}).}
\end{rem}
\section{{Nearly ordinary families of} R\lowercase{ankin}-S\lowercase{elberg convolutions}}
\label{sec:BFES}
From Section~\ref{sec:selmerstructures} to Section~\ref{sec:Colemanmaps}, we established a general formalism of the theory 
and machineries to attack Iwasawa Main conjecture of general Galois deformations. 
In this section, we apply our results to the setting of Section~\ref{subs:exampleRankinSelberg} with help of Beilinson--Flach elements. 
\par 
Until the end of Section~\ref{sec:BFES}, we shall work in the setting of Section~\ref{subs:exampleRankinSelberg}. 
Throughout the section, we take the base field $K$ to be $\mathbb{Q}$ and we set $N_1$ and $N_2$ to be positive integers which are prime to $p$.  
We shall work with a pair of Hida families $\Bf_i$ ($i=1,2$) with respective tame levels $N_i$ and central characters 
$$\Psi_i: \left(\ZZ/pN_i\ZZ\right)^\times \lra \mathcal{O}^\times$$
by setting $\Psi_i(\ell)$ to be the eigenvalue of the diamond operator $\langle \ell \rangle$ acting on the family $\Bf_i$. Here, $\mathcal{O}$ is the ring of integers of a finite extension $\mathcal{E}$ of $\QQ_p$ which contain the images of both Dirichlet characters $\Psi_i$. Recall also the local domain $\mathcal{R}:=\mathbb{I}_{\Bf_1}\widehat{\otimes}\,\mathbb{I}_{\Bf_2}$ and the two-dimensional $\mathcal{R}[[G_{\QQ,\Sigma}]]$-representation $\TT:=\TT_{\Bf_1}\widehat{\otimes}\,\TT_{\Bf_2}$.

Suppose that $p \geq 7$. In particular, the hypothesis {\rm ({MR4})} holds true. 
\subsection{Main conjectures for the nearly deformations of Rankin-Selberg products}
First, we will verify that the required conditions to apply our theory holds true in a great variety of cases: The conditions {\rm (H.0}, {\rm (H.2)}, 
{\rm (H.0$^-$)}, {\rm (H.2$^+$)} and {\rm H.2$^{++}$)} are covered by Lemma~\ref{lem:allhypoHholdstrue}; whereas {\rm (MR1)} by Lemma~\ref{lem:absoluteirredOK}, {\rm (MR2)} by Theorem~\ref{thm:MR2holds}. Notice that {\rm (MR3)} readily follows as a consequence of  {\rm (H.0)} and 
{\rm (H.2)}, whereas {\rm (MR4)} also holds since we have assumed $p \geq 7$.

We will consider the following conditions for both families $\Bf_i$ ($i=1,2$): 

{\rm (F.CM)} $\Bf_i$ is non-CM. 

{\rm (F.PS)} For $i=1,2$, the ring $\mathbb{I}_{\Bf_i}$ is isomorphic to a power series ring in one variable with coefficients in $\mathcal{O}$. 

The condition {\rm (F.PS)} is expected to be valid very often; c.f. the discussion in \cite[Lemma 2.7]{fouquetochiai}. 
\begin{rem}
Notice that the case when $\Bf_2$ has CM is the subject of \cite{buyukbodukleiordmainconj, castellaordmainconj, skinnerurban, wanordmainconj} and our main results in the context of Rankin-Selberg convolutions (c.f. Corollary~\ref{cor:mainthmRankinSelberg} below) handle the case when neither of the forms have CM.
\end{rem}

\begin{define}
Let $f_i=\sum a_n(f_i)q^n \in S_{k_i}(\Gamma_1(N_i))$ ($i=1,2$) be a pair of newforms of respective weights $k_1, k_2$, levels $N_1,N_2$. Let $L_1$ and $L_2$ denote the normal closure of their respective Hecke fields. We say that $f_1$ and $f_2$ are twisted conjugates to mean there exists an embedding $\delta:L_1\hookrightarrow \mathbb{C}$ and a Dirichlet character $\chi_\delta$ (necessarily of conductor dividing $4N_1$) such that $\delta(a_\ell(f_1))=\chi_\delta(\ell)a_\ell(f_2)$.
\end{define}

\begin{lemma}
\label{lem:absoluteirredOK}
\textup{(1)} Let $f_i \in S_{k_i}(\Gamma_1(N_i  ))$ ($i=1,2$) be non-CM newforms which are not twisted-conjugate to each other. Then, the residual representation $\overline{\rho}_f\otimes \overline{\rho}_g$ (modulo $p$) is absolutely irreducible for every sufficiently large $p$. 
\\ 
\textup{(2)} Fix a large enough $p$ so that the conclusion of the first part holds and such that both forms $f_i$ are $p$-ordinary. Let $\Bf_i$ be denote the unique Hida family that admits the $p$-stabilization of $f_i$  as a weight-$k_i$ arithmetic specialization ($i=1,2$). 
Then the condition {\rm (F.CM)} holds true for $\Bf_i$ ($i=1,2$) and the residual representation $\mathbb{T}/\frak{m}_\mathcal{R}\TT$ is absolutely irreducible.
\end{lemma}
See \cite[\S 4.2]{loeffler17Glasgow} for the proof of Lemma \ref{lem:absoluteirredOK}. 
\par 
Next, we turn our attention to the hypotheses {\rm (H.0)}, {\rm (H.2)}, \rm{(H.0$^-$)}, {\rm (H.2$^+$)} and {\rm (H.2$^{++}$)}. To that end, we let $\alpha_i \in \texttt{k}$ denote the reduction of the eigenvalue for the $U_p$-action on $\Bf_i$, modulo the maximal ideal $\frak{m}_i$ of $\mathbb{I}_{\Bf_i}$. Writing $\overline{\rho}_i:G_{\QQ,\Sigma} \ra \textup{GL}_2(\texttt{k})$ for the residual representation carried by $\TT_{\Bf_i}/\frak{m}_i\TT_{\Bf_i}$, it follows that
\begin{equation}
\label{eqn:themodpreprestrictedtoGQpexplicit}
\overline{\rho}_i\big{|}_{G_{\QQ_p}} \sim \left(\begin{array}{cc} \overline{\Psi}_i\alpha_i^{-1} & \star\\ 0 & \alpha_i\end{array}\right)
\,\ee
where, by abuse of notation, we let $\alpha_i$ to denote also the unramified character that assumes the value $\alpha_i$ at the arithmetic Frobenius at $p$. Recall also that, the filtration given in \eqref{eqn:themodpreprestrictedtoGQpexplicit} may be lifted to $\mathbb{I}_{\Bf_i}$ thanks to our assumption that each $\Bf_i$ is $p$-distinguished.  This gives rise to a $4$-step filtration of $\TT$. We recall the steps $F_p^{+}\TT \subset F_p^{++}\TT$, which are both direct summands of $\TT$ of respective ranks $2$ and $3$. Recall also the subquotients $F_p^{--}\TT:=\TT/F_p^{++}\TT$ and $F_p^{-+}\TT:=F_p^{++}\TT/F_p^{+}\TT$.
\begin{lemma}
\label{lem:allhypoHholdstrue} $\,$\\
$\mathbf{(1)}$ If $\alpha_1\alpha_2 \not\equiv 1 \mod \pi_\mathcal{O}$ then {\rm (H.0$^-$)} holds true. If in addition
\begin{itemize}
\item[(i)] either $\overline{\Psi}_1\overline{\Psi}_2$ is ramified at $p$,
\item[(ii)] or else $\overline{\Psi}_1\overline{\Psi}_2(p) \not\equiv \alpha_1\alpha_2 \mod \pi_\mathcal{O}$
\end{itemize}
then $({H.0})$ also holds true.
\\
$\mathbf{(2)}$ Suppose 
\begin{itemize} 
\item[(i)] either that $\overline{\Psi}_1\overline{\Psi}_2\omega^{-1}$ is ramified at $p$,
\item[(ii)] or else $\overline{\Psi}_1\overline{\Psi}_2\omega^{-1}(p) \not\equiv \alpha_1\alpha_2 \mod \pi_\mathcal{O}$
\item[(iii)] either that $\overline{\Psi}_1\omega^{-1}$ is ramified at $p$,
\item[(iv)] or else $\overline{\Psi}_1\omega^{-1}(p) \not\equiv \alpha_1\alpha_2^{-1} \mod \pi_\mathcal{O}$
\end{itemize}
Then  $({H.2^+})$ holds true. If in addition 
 \begin{itemize} 
\item[(v)] either $\overline{\Psi}_2\omega^{-1}$ is ramified at $p$,
\item[(vi)] or else $\overline{\Psi}_2\omega^{-1}(p) \not\equiv \alpha_1^{-1}\alpha_2\mod \pi_\mathcal{O}$
\end{itemize}
then both {\rm (H.2)} and {\rm (H.2$^{++}$)} also hold true.
\end{lemma}
\begin{proof}
This is evident thanks to the local description in (\ref{eqn:themodpreprestrictedtoGQpexplicit}) of the residual representations.
\end{proof}
Finally, we shall provide an explicit sufficient condition for the hypothesis {\rm (MR2)} to hold true. 

\begin{lemma}
\label{lem:f1andf2arenotgaloisconjugatescomparisonoftraces}
Let $f_i \in S_{k_i}(\Gamma_1(N_i  ))$ ($i=1,2$) be non-CM newforms of respective weights $k_1, k_2$, levels $N_1,N_2$ which are not twisted-conjugate to each other. Let $L_i$ denote the normal closure of the Hecke field of $f_i$ and fix an embedding of $L_1L_2$ into $\overline{\QQ}$. 
\\
\textup{(1)} For every $\delta\in \textup{Gal}(L_2/\QQ)$, the set of primes $\ell$ for which we have $a_\ell({{f}}_1)^2=\ell^{k_1-k_2}\delta(a_\ell({{f}}_2))^2$ has  zero density.
\\
\textup{(2)} For a given number field $F/\QQ$, there exists $B(f_1,f_2,F)\in \ZZ^+$ such that for every prime $p>B(f_1,f_2,F)$ and any $\delta\in \textup{Gal}(L_2/\QQ)$ we have
$$v_p\left(a_\ell({{f}}_1)^2-\ell^{k_1-k_2}\delta(a_\ell({{f}}_2))^2\right)=0$$
for every prime $\ell$ which splits completely in $F/\QQ$.
\end{lemma}
\begin{proof}
The first assertion follows from a theorem of Ramakrishnan~\cite[Theorem A]{Ramakrishnan2000}, as in \cite[Lemma 3.1.1]{loeffler17Glasgow}. The second assertion is an immediate consequence from the first, since there only finitely many primes at which all members of a non-zero collection of algebraic numbers have positive valuation.
\end{proof}
We let $F_0$ denote the compositum of $L_1$ and $L_2$. We also choose an integer $B\geq B(f_1,f_2,F_0)$ such that for every $p>B$, both conclusions of Lemma~\ref{lem:absoluteirredOK}(ii) are valid.
\begin{thm}
\label{thm:MR2holds}
Let $f_i$ be as in Lemma~\ref{lem:f1andf2arenotgaloisconjugatescomparisonoftraces} and suppose $p>B$ is a good ordinary prime for both forms. Let $\Bf_i$ be denote the unique Hida family that admits the $p$-stabilization of $f_i$  as a weight-$k_i$ arithmetic specialization ($i=1,2$). Suppose that the hypothesis {\rm (F.Dist)} holds true. The Rankin-Selberg Galois representation $\TT$ 
satisfies {\rm (MR2)} if we further assume the following  conditions: 
\begin{itemize}
\item[(BI.1)] \begin{itemize}
\item[(i)] Either $(N_1,N_2)=1$ and there exists $u$ such that $\Psi_2(u)=-1$, 
\item[(ii)] or the product of the reductions of two central characters $\overline{\Psi}_1\overline{\Psi}_2$ is non-trivial.
\end{itemize}
\item[(BI.2)] The residual representations associated to both $\Bf_1$ and $\Bf_2$ are full in the sense that they contain a conjugate of $\textup{SL}_{2}(\mathbb{F}_p)$.
\end{itemize}
\end{thm}
\begin{rem}
Let $f \in S_k(\Gamma_1(N),\varepsilon)$ be a non-CM newform. Then the results of Momose, Ribet, and Ghate--Gonzalez-Jimenez--Quer~\cite{GGQ} guarantee that for all but finitely primes $p$, mod $p$ representation $\overline{\rho}_f$ associated to $f$ is full.
\end{rem}
\textbf{Notation.} For $f_i$ as in the statement of Theorem~\ref{thm:MR2holds}, we fix an embedding $F_0 \hookrightarrow \overline{\QQ}_p$ which sends both $a_p(f_i)$ to units. Let $\frak{p}$ denote the prime of $F_0$ induced by this embedding. By slight abuse, we shall denote the prime of $L_i$ lying below $\frak{p}$ also by the symbol $\frak{p}$.
\begin{proof}[Proof of Theorem~\ref{thm:MR2holds}]
Our proof builds on the work of Fischman and Loeffler; our notation in this proof is a hybrid of that used in these two articles.  For $i=1,2$, we let $H_{{\bf{f}}_i}
\subset G_\QQ$ denote the subgroup defined at the beginning of \cite[\S3.2]{fischmanAIF} and let $H=H_{{\bf{f}}_1}\cap H_{{\bf{f}}_2}$. 
Since $H_{{\bf{f}}_i}$ is of finite index in $G_\QQ$,  $H$ is a subgroup of finite index 
in $H_{{\bf{f}}_1}$ and $H_{{\bf{f}}_2}$ and both $\Psi_1$ and $\Psi_2$ are trivial on $H$.  We also recall the subring $R_{{\bf{f}}_i}\subset \mathbb{I}_{{\bf{f}}_i}$ defined in loc. cit. Set $\rho:=\rho_{{\bf{f}}_1}\otimes \rho_{{\bf{f}}_2}$. When BI.1(i) is valid, we shall prove that
\be\label{eqn:goodeleemntinthejointimage}\left(\left(\begin{array}{cc} 1&1 \\
0& 1 \end{array}\right) , \left(\begin{array}{cc} 1&0 \\
0& -1\end{array}\right)\right) \in \rho(G_{\QQ(\mu_{p^\infty})})
\ee
 
\begin{itemize}
\item \textbf{Step 1.} We let $G_i^{(n)}:=\rho_{{\bf{f}}_i}\left(H\cap G_{\QQ(\mu_{p^n})}\right)$ and $G_i^{\circ}:=\displaystyle{\bigcap_n}\, G_i^{(n)}=\rho_{{\bf{f}}_i}\left(H\cap G_{\QQ(\mu_{p^\infty})}\right)$. We also set $\mathscr{G}_n=\Gal(\QQ(\mu_{p^{n+1}})/\QQ)$.
\end{itemize}
Then $G_i^\circ=\textup{SL}_2(R_{{\bf{f}}_i})$. 

Indeed, it follows from \cite[Corollary 4.11]{fischmanAIF} that 
\be\label{eqn:Fischmanreduction1}
G_i^{(n)}\subset \{M\in \GL_2(R_{{\bf{f}}_i}) \mid \det(M)\in \left(1+p\ZZ_p\right)^{p^n}\}
\ee
(where Fischman uses the notation $\Gamma^\prime$ for the group  $1+p\ZZ_p$). 
Note also that the quotient group $\rho_{{\bf{f}}_i}\left(H\right)/G_i^{(n)}$ is a cyclic group of dividing $(p-1)p^n$, being the homomorphic image of $H/H\cap G_{\QQ(\mu_{p^\infty})}$ under the map induced from $\det\circ\,\rho_{{\bf{f}}_i}$. On the other hand, it also follows from  \cite[Corollary 4.11]{fischmanAIF} combined with (\ref{eqn:Fischmanreduction1}) that
$$\rho_{{\bf{f}}_i}\left(H\right)/G_i^{(n)}\stackrel{\det}{\twoheadrightarrow} (\mu_p\times(1+p\ZZ_p))\big{/}(1+p\ZZ_p)^{p^n}\cong (\ZZ/p^{n+1}\ZZ)^\times$$
and we conclude that the containment in (\ref{eqn:Fischmanreduction1}) is an equality. We may more precisely write 
\begin{align*}
\notag G_i^{(n)}=\left\{ 
M\in \GL_2(R_{{\bf{f}}_i})\, \middle| \, \det(M)\in \{\Psi_i^{(p)}\widetilde{\chi}_i(\gamma)\}_{\gamma\in G_{\QQ(\mu_{p^n})}}\right\}.
\end{align*}
Since 
$$\left\{ 
M\in \GL_2(R_{{\bf{f}}_i})\, \middle| \, \det(M)\in \{\Psi_i^{(p)}\widetilde{\chi}_i(\gamma)\}_{\gamma\in G_{\QQ(\mu_{p^n})}}\right\}= 
\bigcup_{\gamma\in G_{\QQ(\mu_{p^n})}}\left(\begin{array}{cc} 1&0 \\0& \Psi_i^{(p)}\widetilde{\chi}_i(\gamma) \end{array}\right)\textup{SL}_2(R_{{\bf{f}}_i})
$$
we conclude that
\begin{align}
\label{eqn:theimageforrhoiexplicit}
 G_i^{(n)}=\bigcup_{\gamma\in G_{\QQ(\mu_{p^n})}}\left(\begin{array}{cc} 1&0 \\0& \Psi_i^{(p)}\widetilde{\chi}_i(\gamma) \end{array}\right)\textup{SL}_2(R_{{\bf{f}}_i})\,.
\end{align}
On considering the intersection $\displaystyle{\bigcap_n G_i^{(n)}}$, we conclude that
$$G^\circ_i= \left\{ M\in \GL_2(R_{{\bf{f}}_i}) 
\, \middle| \, \det(M)\in \bigcap_{n}\left(1+p\ZZ_p\right)^{p^n}\right\}= \textup{SL}_2(R_{{\bf{f}}_i})\,.
$$
\begin{itemize}
\item\textbf{Step 2.} Now set $G^{(n)}:=\rho\left(H\cap G_{\QQ(\mu_{p^n})}\right) \subset G_1^{(n)}\times G^{(n)}_2$ and $G^{\circ}=\displaystyle{\bigcap_n} \,G^{(n)}$. We claim that $G^\circ=G_1^\circ\times G^\circ_2=\textup{SL}_2(R_{{\bf{f}}_1})\times \textup{SL}_2(R_{{\bf{f}}_2})$.
\end{itemize}
We shall very closely follow the arguments of Loeffler in the proofs of \cite[Proposition 3.2.1 and Theorem 3.2.2]{loeffler17Glasgow} in order to verify this claim. To that end, we let $U<\rho(H)\subset \rho_{{\bf{f}}_1}(H)\times  \rho_{{\bf{f}}_2}(H)$ denote the subgroup of elements $(M_1,M_2)$ such that $M_i \in \textup{SL}_2(R_{{\bf{f}}_i})$. By the discussion in the first step, notice that both natural projection maps $U\ra \textup{SL}_2(R_{{\bf{f}}_i})$ are surjective. 

We shall need the following result, which is a particular instance of Goursat's Lemma:
\begin{lemma}\label{lemma:goursat}
Let $\mathscr{N} < \textup{SL}_2(R_{{\bf{f}}_1}) \times \textup{SL}_2(R_{{\bf{f}}_2})$ be 
a closed  subgroup that surjects onto each factor under the natural projection maps.  Then there exists closed normal subgroups $\frak{N}_i< \textup{SL}_2(R_{{\bf{f}}_i})$ such that the image of $\mathscr{N}$ in $\textup{SL}_2(R_{{\bf{f}}_1})/\frak{N}_1\times \textup{SL}_2(R_{{\bf{f}}_2})/\frak{N}_2$ is the graph of an isomorphism $\textup{SL}_2(R_{{\bf{f}}_1})/\frak{N}_1\cong \textup{SL}_2(R_{{\bf{f}}_2})/\frak{N}_2$. Moreover, $\mathscr{N}$ is a proper subgroup of $\textup{SL}_2(R_{{\bf{f}}_1})\times \textup{SL}_2(R_{{\bf{f}}_2})$ if and only if 
$\frak{N}_i$ is a proper normal subgroup of $\textup{SL}_2(R_{{\bf{f}}_i})$.
\end{lemma}
Next, we next explain that the only maximal proper normal subgroup of $\textup{SL}_2(R_{{\bf{f}}_1})$ is the kernel of the projection to $\textup{PSL}_2(\texttt{k}_{i})$ (where $\texttt{k}_{i}=R_{{\bf{f}}_i}/\frak{m}_{i}$ is the residue field). Indeed, if $\frak{N}<\textup{SL}_2(R_{{\bf{f}}_i})$ is a proper normal subgroup, then so is its image $\overline{\frak{N}}$ in $\textup{PSL}_2(\texttt{k}_{i})$. However, since we have assumed that $p>7$, it follows that $\textup{PSL}_2(\texttt{k}_{i})$ is simple and hence we either have $\overline{\frak{N}}=\{1\}$ or else $\overline{\frak{N}}=\textup{PSL}_2(\texttt{k}_{i})$.  We shall next explain that the latter option is impossible: If  $\overline{\frak{N}}=\textup{PSL}_2(\texttt{k}_{i})$, then since only proper non-trivial normal subgroup of $\textup{SL}_2(\texttt{k}_{i})$ is $\{\pm 1\}$, it follows that $\frak{N}$ maps onto  $\textup{SL}_2(\texttt{k}_{i})$ (under the obvious reduction map induced from $R_{\bf{f}_i}\ra \texttt{k}_{i}$). This observation combined with \cite[Proposition 3.15]{fischmanAIF} shows that $\frak{N}=\textup{SL}_2(R_{{\bf{f}}_i})$, contrary to our assumption that $\frak{N}$ is a proper normal subgroup of $\textup{SL}_2(R_{{\bf{f}}_i})$, so $\overline{\frak{N}}=\{1\}$. We conclude that $\frak{N}\subset \ker\left(\textup{SL}_2(R_{{\bf{f}}_i})\twoheadrightarrow\textup{PSL}_2(\texttt{k}_{i})\right)$, as claimed.

Suppose now that $U$ is a proper subgroup of  $\textup{SL}_2(R_{{\bf{f}}_1})\times \textup{SL}_2(R_{{\bf{f}}_2})$. By Lemma \ref{lemma:goursat}, 
 we have isomorphisms 
$$\xymatrix{\textup{SL}_2(R_{{\bf{f}}_1})/\frak{N}_1\ar[r]^\sim_{\phi}\ar@{->>}[d]&\textup{SL}_2(R_{{\bf{f}}_2})/\frak{N}_2\ar@{->>}[d]\\
\textup{PSL}_2(\texttt{k}_{1})\ar[r]^{\sim}&\textup{PSL}_2(\texttt{k}_{2})
}$$
(where the isomorphism on the second row is induced from the one on the first row), which in turn also induces an isomorphism $\texttt{k}_{1}\stackrel{\sim}{\ra} \texttt{k}_{2}$ and $R_{{\bf{f}}_1}\stackrel{\sim}{\ra}R_{{\bf{f}}_2}$. Henceforth, we shall identify $\texttt{k}_1$ and $\texttt{k}_2$  
through this isomorphism and denote either one of them by $\texttt{k}$. Since all automorphisms of $\textup{PSL}_2(\texttt{k})$ arise as the compositum of a field automorphism of $\texttt{k}$ and conjugation by an element of $\textup{PSL}_2(\texttt{k})$, we conclude with the following:
\begin{claim}
There exists $\delta \in \textup{Gal}(\texttt{k}/\mathbb{F}_p)$ such that $U$ is contained in the group
$$
\left\{(M_1,M_2)\in \textup{SL}_2(R_{{\bf{f}}_1})\times \textup{SL}_2(R_{{\bf{f}}_2}) \, \middle| \,  M_1 \hbox{ mod }\frak{m}_{1} = 
\delta\left(\pm M_2  \hbox{ mod }\frak{m}_{2}\right) \right\}\,.$$
\end{claim}
Fix a choice of $\delta$ as in the claim above. Let $(A_1,A_2) \in G^{(n)}$ be any pair and set 
\begin{equation}
\label{eqn:twillbescalarlater}
t=(A_1 \hbox{ mod } \frak{m}_{1})^{-1}\delta(A_2 \hbox{ mod } \frak{m}_{2}) \in \GL_2(\texttt{k}).
\end{equation}
Let $[t] \in  \GL_2(\texttt{k})/\{\pm1\}$ denote its image (and similarly, we shall also talk about $[M]$ for any $M\in \GL_2(R_{{\bf{f}}_i})$). For any $(u,v) \in U$, we have the identity 
$$[u^{-1}tu]=[u^{-1}A_1^{-1}(\delta A_2)u]=[A_1^{-1}][(A_1uA_1^{-1})^{-1}(\delta( A_2 v A_2)^{-1})][\delta A_2][(\delta v)^{-1}u]=[A_1^{-1}(\delta A_2)]=[t]$$
where we have used for the third equality the fact that $U$ is a normal subgroup of $G^{(n)}$ and the Claim above. This in particular shows that $[t]$ commutes with every element of $\textup{PSL}_2(\texttt{k})$
 and therefore that $t$ is a scalar matrix.

The fact that $t$ is a scalar matrix combined with (\ref{eqn:twillbescalarlater}) shows that
\be\label{eqn:tsquareisratioofdets}
t^2=(\det(A_1) \mod \frak{m}_1)^{-1}\delta(\det(A_2)\mod \frak{m}_2),
\ee
as well as that 
\be\label{eqn:tisratiooftraces}t\cdot\left(\textup{tr}\,\rho_{{\bf{f}}_1}(\sigma) \hbox{ mod } \frak{m}_1\right)\pm  \delta \left(\textup{tr}\,\rho_{{\bf{f}}_2}(\sigma) \hbox{ mod } \frak{m}_2\right)=0
\ee
for every $\sigma \in H$. Let $\ell$ be any prime that splits completely in $\overline{\QQ}^{H}/\QQ$. Then decomposition groups at $\ell$ are subgroups of $H$. By choosing $\sigma$ to be any lift of the Frobenius at $\ell$ to $H$ in  (\ref{eqn:tisratiooftraces}), we conclude that
\be\label{eqn:aellhidafamiliesandt}t\cdot \left(a_\ell({\bf{f}}_1)\hbox{ mod } \frak{m}_1\right)\pm \delta(a_\ell({\bf{f}}_2) \hbox{ mod } \frak{m}_2)=0\,.\ee 
Let $\overline{x}$ denote the image of $x$ in the residue field $\texttt{k}$. 
We have the following equality which takes place in $\texttt{k}$ 
\begin{multline}\label{eqn:residualrepsaretesamesotracesare}
\overline{a_\ell({{f}}_1)}^2 -{\ell}^{k_1-k_2}\delta(\overline{a_\ell({{f}}_2)})^2=
\\   \left(a_\ell({\bf{f}}_1)^2\hbox{ mod } \frak{m}_1\right) -  \left(\det\left(\rho_{{\bf{f}}_1}(\textup{Fr}_\ell)\right) \hbox{ mod } \frak{m}_1\right)\delta(\det\left(\rho_{{\bf{f}}_2}(\textup{Fr}_\ell) \hbox{ mod } \frak{m}_2\right)^{-1})
\delta(a_\ell({\bf{f}}_2)\hbox{ mod } \frak{m}_2)^2
\end{multline} 
since $f_i$ is a specialization of the Hida family $\bf{f}_i$ and since we have $\Psi_i(\ell)=1$ by the choice of $\ell$. Moreover, it follows from (\ref{eqn:tsquareisratioofdets}) and (\ref{eqn:aellhidafamiliesandt}) that
$$- \left(\det\left(\rho_{{\bf{f}}_1}(\textup{Fr}_\ell)\right) \hbox{ mod } \frak{m}_1\right)\delta(\det\left(\rho_{{\bf{f}}_2}(\textup{Fr}_\ell) \hbox{ mod } \frak{m}_2\right)^{-1}) \delta(a_\ell({\bf{f}}_2)\hbox{ mod } \frak{m}_2)^2+  \left(a_\ell({\bf{f}}_1)^2\hbox{ mod } \frak{m}_1\right)=0.$$ 
Combining this with (\ref{eqn:residualrepsaretesamesotracesare}), we see for every $\ell$ as above that
$$-\ell^{k_1-k_2}\delta(\overline{a_\ell({{f}}_2)})^2+\overline{a_\ell({{f}}_1)}^2 =0.$$
Recall that $L_2$ is the normal closure of the Hecke field of $f_2$. We conclude that there exists a choice of $\widetilde{\delta}\in{D}_{\fp}(L_2/\QQ):=\{\tau \in \textup{Gal}(L_2/\QQ): \tau(\fp)=\fp\}$ lifting the automorphism $\delta$ of $\texttt{k}$  such that 
$$v_p\left(a_\ell({{f}}_1)^2-\ell^{k_1-k_2}\widetilde{\delta}(a_\ell({{f}}_2))^2\right)>0$$
for a set of primes $\ell$ with positive upper density. This contradicts our choice of the pair of eigenforms $f_1, f_2$ and the prime $p$ (in view of Lemma~\ref{lem:f1andf2arenotgaloisconjugatescomparisonoftraces}) and shows that 
\be\label{eqn:pindowntheimagewithdetone}
U=\textup{SL}_2(R_{{\bf{f}}_1})\times \textup{SL}_2(R_{{\bf{f}}_2})\,.
\ee
It follows from the identification  (\ref{eqn:pindowntheimagewithdetone}) combined with the explicit description (\ref{eqn:theimageforrhoiexplicit}) that 
$$G^{(n)}=\bigcup_{\gamma\in G_{\QQ(\mu_{p^n})}}\left(\left(\begin{array}{cc} 1&0 \\0& \Psi_1^{(p)}\widetilde{\chi}_1(\gamma) \end{array}\right),\left(\begin{array}{cc} 1&0 \\0& \Psi_2^{(p)}\widetilde{\chi}_2(\gamma) \end{array}\right)\right)\textup{SL}_2(R_{{\bf{f}}_1})\times \textup{SL}_2(R_{{\bf{f}}_2})\,.$$
This concludes the proof that
$$ G^\circ=\bigcap_n G^{(n)}=\textup{SL}_2(R_{{\bf{f}}_1})\times \textup{SL}_2(R_{{\bf{f}}_2})=G^\circ_1\times G^\circ_2\,.$$
\begin{itemize}
\item\textbf{Step 3.} There exists $\tau\in G_{\QQ(\mu_{p^\infty})}$ such that $\rho(\tau)=\left(\left(\begin{array}{cc} 1&0 \\0&1 \end{array}\right),\left(\begin{array}{cc} 1&0 \\0& -1 \end{array}\right)\right)$.
\end{itemize}
Indeed, since we assumed that $N_1,N_2$ and $p$ are pairwise coprime, we may choose $\sigma \in  G_{\QQ(\mu^{p_\infty})}$ such that $\psi(\sigma)=1$ for all Dirichlet characters $\psi$ of conductor $N_1$, $\Psi_2(\sigma)=-1$. In the proof of \cite[Theorem 4.15]{fischmanAIF}, Fischman has verified for each one of the Hida families ${\bf{f}}_1, {\bf{f}}_2$ and any $g\in G_{\QQ(\mu_{p^\infty})}$ that
$$\left(\begin{array}{cc} \beta_i(g)&0 \\0&\Psi_i^{(N_i)}\beta_i^{-1}(g) \end{array}\right)\in \rho_{{\bf f}_i}(G_{\QQ_{\mu_{p^\infty}}}),$$
where we write $\beta_i(g)$ in place of Fischman's $\alpha(g)$ when we are working with the Hida family ${\bf{f}}_i$. Thus, 
$$\left(\left(\begin{array}{cc} \beta_1(g)&0 \\0&\Psi_1^{(N_1)}\beta_1^{-1}(g) \end{array}\right),\left(\begin{array}{cc} \beta_2(g)&0 \\0&\Psi_2^{(N_2)}\beta_2^{-1}(g) \end{array}\right)\right)\in \rho(G_{\QQ_{\mu_{p^\infty}}})\,.$$
Let the automorphism $\gamma$ and the Dirichlet character $\chi_\gamma$ be as in the paragraph preceding Lemma~4.14 of \cite{fischmanAIF}, with $F={\bf{f}}_1$. For $\sigma$ chosen as above, we have $\chi_\gamma(\sigma)=1$ since the conductor of $\chi_\gamma$ divides $N_1$. The choice of $\beta_1(\sigma) \in \mathbb{I}_{{{\bf{f}}}_1}$ which is explained in the same paragraph of op. cit. (where we remind the reader that our $\beta_1$ corresponds to Fischman's $\alpha$), we may take $\beta_1(\sigma)=1$. Hence,
$$
\left(\begin{array}{cc} \beta_1(\sigma)&0 \\0&\Psi_1^{(N_1)}\beta_1^{-1}(\sigma) \end{array}\right)=\left(\begin{array}{cc} 1&0 \\0&1\end{array}\right).
$$
Furthermore, 
$$
\left(\begin{array}{cc} \beta_2(\sigma)&0 \\0&\Psi_2^{(N_2)}\beta_2^{-1}(\sigma) \end{array}\right)=
\left(\begin{array}{cc} \beta&0 \\0&-\beta^{-1} \end{array}\right)\,.
$$
where we have set $\beta=\beta_2(\sigma)$. Combining with the main conclusion of Step 2, we deduce that
$$\left(\left(\begin{array}{cc} 1&0 \\0&1\end{array}\right), \left(\begin{array}{cc} \beta&0 \\0&-\beta^{-1} \end{array}\right)\right)\textup{SL}_2(R_{{\bf{f}}_1})\times \textup{SL}_2(R_{{\bf{f}}_2})\in \rho(G_{\QQ_{\mu_{p^\infty}}})\,.$$
Since we have $\left(\begin{array}{cc} 1&1 \\0&1 \end{array}\right) \in \textup{SL}_2(R_{{\bf{f}}_1})$ and $\left(\begin{array}{cc} \beta^{-1}&0 \\0&\beta \end{array}\right) \in \textup{SL}_2(R_{{\bf{f}}_2})$, we conclude our proof that 
\be\label{goodelementwhenBIiholds}\left(\left(\begin{array}{cc} 1&1 \\0&1 \end{array}\right),\left(\begin{array}{cc} 1&0 \\0& -1 \end{array}\right)\right) \in \rho(G_{\QQ_{\mu_{p^\infty}}})\,.\ee
\begin{itemize}
\item\textbf{Step 4.} Conclusion.
\end{itemize}
The image of the element (\ref{goodelementwhenBIiholds}) inside $\textup{GL}_4(\mathcal{R})$ is the matrix 
$$\left(\begin{array}{cccc} 1&1&0&0 \\0&1&0&0\\0&0&-1&-1\\0&0&0&-1\end{array}\right)$$ 
and we have  verified the existence of an element in $\rho(G_{\QQ_{\mu_{p^\infty}}})$ whose image has the desired shape when BI1.(i) is valid.
\\
\\

When BI.1(ii) is in effect, one may similarly prove for any $x \in \ZZ_p[[X]]^\times$ (which we view as an element of $\LL_i:=\ZZ_p[[\Gamma_i]]$ via $\chi_i$) and any $u \in (\ZZ/pN_1N_2)^\times$ with $\overline{\Psi}_1(u)\overline{\Psi}_2(u) \neq 1$ that 
\be\label{eqn:goodeleemntinthejointimageii}\left(\left(\begin{array}{cc} x^{-1}&0 \\0& x\Psi_1(u) \end{array}\right), \left(\begin{array}{cc} x&0 \\0& x^{-1}\Psi_2(u) \end{array}\right)\right)\in \rho\left(G_{\QQ(\mu_{p^\infty})}\right)\,.
\ee

The image of the element (\ref{eqn:goodeleemntinthejointimageii}) inside $\textup{GL}_4(\mathcal{R})$ is the diagonal matrix 
$$\textup{Diag}\left(1, x^{-2}\Psi_2(u), x^{2}\Psi_1(u), \Psi_1(u)\Psi_2(u)\right).$$ 
On choosing $x$ in a way that 
$$
x^{-2}\Psi_2(u) \not\equiv 1   \hbox{ and }  x^{2}\Psi_2(u) \not\equiv 1  \hbox{ mod } (\pi_\mathcal{O}, X),
$$
we have again verified the existence of an element in $\rho(G_{\QQ_{\mu_{p^\infty}}})$ whose image has the desired shape.
\end{proof}

\subsection{Beilinson--Flach elements and their relation to $p$-adic $L$-functions}\label{section:BFelements}
Consider the following condition on the central characters of the Hida families $\Bf_1$ and $\Bf_2$:
\begin{itemize}
\item[(H.c)] The conductor of ${\Psi}_1{\Psi}_2$ is prime to $p$.
\end{itemize}

Until the end of this article, we shall assume the validity of (H.c). For every positive integer $n$ that is coprime to $6N_1N_2$, there exists a generalized Beilinson--Flach element
$$\BF^{\Bf_1,\Bf_2}_{n}\in H^1(\QQ(\mu_{n}),\TT_\cyc)$$
that was constructed in \cite{KLZ1, KLZ2}, building on \cite{LLZ}. 
{\begin{rem}
The collection $\{\BF^{\Bf_1,\Bf_2}_{n}\}_n$ verifies \emph{some} Euler system distribution relation as $n$ varies, which is not  the distribution relation we have asked for in Section~\ref{subsec:locallyrestricKSexists}. However, one may rely on \cite[\S9.6]{r00} to obtain an Euler system with the distribution relation given as in Definition~\ref{define_ES} and such that the initial term of the new Euler system agrees with $\BF^{\Bf_1,\Bf_2}_{1}$; see \cite[\S 7.3]{LLZ} for details.
\end{rem}}
\begin{rem}
Since we assume the validity of (H.c), we may work with the normalised collection of classes $\{\BF^{\Bf_1,\Bf_2}_{n}\}$ rather than those classes denoted by $\{_c\BF^{\Bf_1,\Bf_2}_{n}\}$ in \cite{KLZ1, KLZ2}. Here, $c$ is the ``smoothing factor'' of \cite{LLZ}.  We refer the readers to \cite[\S 6.8.1]{LLZ} where the authors explain how to dispose this auxiliary factor and verify the integrality of the original classes $\BF^{\Bf_1,\Bf_2}_{n}$. 
\end{rem}

The following statement is proved in \cite[Proposition 8.1.7]{KLZ2}.
\begin{prop}
\label{prop:BFislocallyrestrictedES}
The collection $\BF=\{\BF^{\Bf_1,\Bf_2}_{n}\}$ gives rise to an element of ${\ES}^{+}(\TT_\cyc)$, in the sense of Definition \ref{def:locallyrestrictedES}. 
\end{prop}

In more concrete terms, Proposition~\ref{prop:BFislocallyrestrictedES} asserts that the projection of the local class 
$$\res_p(\BF^{\Bf_1,\Bf_2}_{n}) \in H^1(\QQ(\mu_n)_p,\TT_\cyc)$$ to 
$H^1(\QQ(\mu_n)_p,F^{--}_p\TT_\cyc)$
is trivial and the collection $\BF$ forms a locally restricted Euler system in the sense of Definition~\ref{def:locallyrestrictedES}. This means that our Theorem~\ref{thm:weakleopoldt} (and its corollaries) may be applied with the collection $\BF$ to yield the following result.
\begin{thm}
\label{thm:mainconjwithoutpadicLfunc}
Let $f_i \in S_{k_i}(\Gamma_1(N_i ))$ ($i=1,2$) be non-CM newforms of respective weights $k_1, k_2$, levels $N_1,N_2$ which are not twisted-conjugate to each other and suppose $p>B$ is a good ordinary prime for both forms.  Let $\Bf_i$ be the unique Hida family that admits the $p$-stabilization of $f_i$ as a weight-$k_i$ arithmetic specialization ($i=1,2$). 
\par 
Assume the conditions listed in Lemma~\ref{lem:allhypoHholdstrue} which ensure the validity of {\rm (H.0)}, {\rm (H.2)}, 
{\rm (H.0$^-$)}, {\rm (H.2$^+$)} and {\rm (H.2$^{++}$)}. Assume in addition the hypotheses \textup{(H.c), (BI.1), (BI.2)} and {\rm (F.Dist)} as well as that the class $\BF^{\Bf_1,\Bf_2}_{1} \in H^1(\QQ,\TT_\cyc)$ is non-trivial.
\begin{itemize}
\item[(1)]  The module $H^1_{\FF_+^*}(\QQ,\TT^\vee_\cyc (1))^\vee$ is $\mathcal{R}_\cyc$-torsion. 
\item[(2)] The module $H^1_{\FF_{+}}(\QQ,\TT_\cyc)$ is a torsion-free $\mathcal{R}_\cyc$-module of rank one.
\item[(3)] We have
$$\textup{char}_{\mathcal{R}_\cyc}\left(H^1_{\FF_+^*}(\QQ,\TT_\cyc^\vee (1))^\vee\right)\,\, \supset \,\, \textup{char}_{\mathcal{R}_\cyc}\left(H^1_{\FF_+}(\QQ,\TT_\cyc)\big{/}\mathcal{R}_\cyc\cdot\BF^{\Bf_1,\Bf_2}_{1}\right).$$
\end{itemize}
Suppose in addition that $\res_p(\BF^{\Bf_1,\Bf_2}_{1})\neq 0$. Then,
\begin{itemize}
\item[(4)] The module $H^1_{\FF_\Gr^*}(K,\TT_\cyc^\vee (1))^\vee$ is $\mathcal{R}_\cyc$-torsion.
\item[(5)] We have 
$$\textup{char}\left(H^1_{\FF_\Gr^*}(\QQ,\TT_\cyc^\vee (1))^\vee\right)\,\, \supset \,\, \textup{char}\left(H^1_{+/\textup{f}}(\QQ_p,\TT_\cyc)\big{/}\mathcal{R}_\cyc\cdot \res_{+/\textup{f}}\left(\BF^{\Bf_1,\Bf_2}_{1}\right)\right)$$
where we recall that 
$$
H^1_{+/\textup{f}}(\QQ_p,\TT_\cyc):=H^1_{\FF_{+}}(\QQ_p,\TT_\cyc)/H^1_{\FF_{\Gr}}(\QQ_p,\TT_\cyc)\stackrel{\sim}{\lra} H^1(\QQ_p,F^{-+}_p\TT_\cyc)$$
and the map $\res_{+/\textup{f}}$ is the compositum of the arrows 
$$
H^1(\QQ,\TT_\cyc)\stackrel{\res_p}{\lra} H^1(\QQ_p,\TT_\cyc)\lra H^1(\QQ_p,F^{-+}_p\TT_\cyc)\,.$$
\end{itemize}
\end{thm}
\begin{proof}
(1), (2) and (3) follows from Theorem~\ref{thm:weakleopoldt}, which applies in the present situation thanks to Lemma~\ref{lem:allhypoHholdstrue} and Theorem~\ref{thm:MR2holds}.  Parts (4) and (5) are then respective restatements of Corollary~\ref{cor:Greenbergtorsion} and Theorem~\ref{thm:maingreenbergbound} in the current set up.
\end{proof}
\begin{define}\label{DEF_extensionsofbranches}
{Set $N=\gcd(N_1,N_2)$ and let $\bf{h}$ denote Hida's universal ordinary Hecke algebra of level $Np^\infty$, which is finite flat over $\LL:=\ZZ_p[[1+p\ZZ_p]]$. Let us write ${\rm Quot}(\bf h)$ for the total ring of quotients of $\bf{h}$. Let $\mathscr{K}\subset \overline{{\rm Frac}(\LL)}$ denote a sufficiently large extension of ${\rm Frac}(\LL)$ that contains all Galois conjugates of all subfields of ${\rm Quot}(\bf h)$. Let $\mathbb{I}$ denote the integral closure of $\LL$ in $\mathscr{K}$.} 

{We fix once and for all embeddings $\mathbb{I}_{\bf{f}_1}\hookrightarrow \mathbb{I}$ and $\mathbb{I}_{\bf{f}_2}\hookrightarrow \mathbb{I}$ of $\LL$-algebras, which in turn induces a $\LL\widehat{\otimes}_{\ZZ_p}\LL$-algebra embedding $\iota:\mathbb{I}_{\bf{f}_1}\widehat{\otimes}_{\ZZ_p}\mathbb{I}_{\bf{f}_2}\hookrightarrow  \mathbb{I}\widehat{\otimes}_{\ZZ_p}\mathbb{I}$. We let $H \in \mathbb{I}_{\Bf_1}$ to denote the congruence divisor associated to the $\LL$-adic form $\bf{f}_1$. See \cite[Section 4]{Hida88} for a precise definition and further properties of $H$. We will only note here that the congruence divisor $H$ is an invariant controlling the congruences between ${\bf{f}}_1$ and other Hida families of the same tame conductor. Through the embedding we have fixed above, we will regard $H$ also as an element of $\mathbb{I}$.}
\end{define}
Suppose $\kappa_i:\mathbb{I}_{\Bf_i}\lra \overline{\mathbb{Q}}_p$ is an arithmetic specialization. We pick any extension of $\kappa_i$ to a morphism $\mathbb{I}\lra \overline{\mathbb{Q}}_p$ (which exists thanks to going-up theorem) and denote it also by $\kappa_i$ with an abuse of notation. 

\begin{define}
\label{def:HidaspadicLfunction}
{Let $L_p^{\textup{Hida}}(\Bf_1,\Bf_2,\mathbf{j}) \in \frac{1}{H}
\mathbb{I}\,\widehat{\otimes}_{\ZZ_p}\mathbb{I}\,\widehat{\otimes}_{\ZZ_p}\LL_\cyc$ denote Hida's $p$-adic $L$-function in $3$-variables defined in \cite[Theorem I]{Hida88}, where $\mathbf{j}$ stands for the cyclotomic variable. } The $p$-adic $L$-function 
$L_p^{\textup{Hida}}(\Bf_1,\Bf_2,\mathbf{j})$ is characterized by the interpolation property that 
\begin{align}
\label{equation:interpolationHida}
L_p^{\textup{Hida}}(\Bf_1 (\kappa_1 ),\Bf_2 (\kappa_2 ),j) 
=& \left(1-\frac{p^{j-1}}{\alpha (\Bf_1 (\kappa_1))\alpha (\Bf_2(\kappa_2))} \right) \left(1-\frac{p^{j-1}}{\alpha (\Bf_1 (\kappa_1))\beta (\Bf_2(\kappa_2))}  \right)\\
\notag &\times \left(1-\frac{\beta (\Bf_1 (\kappa_1))\alpha (\Bf_2(\kappa_2))}{p^{j}} \right) \left(1-\frac{\beta (\Bf_1 (\kappa_1))\alpha (\Bf_2(\kappa_2))}{p^{j}}  \right)\\
\notag&\times {\left(1-\frac{\beta (\Bf_1(\kappa_1))}{p\alpha (\Bf_1(\kappa_1))}   \right)^{-1} \left(1-\frac{\beta (\Bf_1(\kappa_1))}{\alpha (\Bf_1(\kappa_1))} \right)^{-1}}\times \dfrac{L(\Bf_1 (\kappa_1 ) ,\Bf_2 (\kappa_2 ), j)}{\Omega (\kappa_1 ,\kappa_2 ,j )}
\end{align}
for arithmetic specializations $\kappa_1,\kappa_2$ of the families $\Bf_1,\Bf_2$ of respective weights $w_1> w_2$ and integers $j \in [w_2+2,w_1+1]$ such that the character $\omega^{-w_2-1+j}\Psi_2^{(p)}\phi_{\kappa_2}\phi$ is trivial (which amounts to the requirement that the $G_{\QQ_p}$-representation $F^{-+}V_\kappa$ be crystalline). Here, 
\begin{enumerate}
\item[(i)] $\alpha (\Bf_i (\kappa_i  )) =\kappa_i\left(a_p (\Bf_i )\right)$ and 
$\beta(\Bf_i(\kappa_i ))=p^{w_i+1}\Psi_i^{(N_i)}(p)\kappa_i\left(a_p (\Bf_i )\right)^{-1}$;
\item[(ii)]
$L(\Bf_1(\kappa_1) , \Bf_2(\kappa_2), s)$ is a Rankin-Selberg $L$-function;
\item[(iii)]$\Omega (\kappa_1 ,\kappa_2 ,j ) \in \mathbb{C}^\times$ is the complex period which is given as the ratio $\dfrac{\Omega(P,Q,R)}{cw}$ in the notation of \cite[Theorem I]{Hida88}.
\end{enumerate}
\end{define}

{We next explain the connection of Beilinson--Flach elements and Hida's $p$-adic $L$-function. In order to do so, we first record a restatement (in Theorem~\ref{thm:mainEXPrankinselberg} below) of Corollary~\ref{cor:dualexponentialmapforH++} in the particular set up when $\textup{Gr}^1_p\TT_\cyc=F^{-+}\TT_\cyc$.} We shall henceforth drop $\left(\textup{Gr}^1_p\TT_\cyc\right)^{\mathcal{R}_\cyc}(1)$ and ${(\mathrm{Gr}^{m}_p V_\kappa)^{\kappa (\RRc)}(1)}$ from the subscripts to ease our notation, as we shall solely work with this choice until the end of this article. Notice further that we have\
$$d_1(\kappa)=j-\omega_2\,\,\,\,\,\,\,\hbox{and}\,\,\,\,\kappa\vert_{\mathcal{R}} (\widetilde{\alpha}_1 (\mathrm{Fr}_p))=a_p(\Bf_1(\kappa_1) )a_p(\Bf_2(\kappa_2) )^{-1}\Psi_2^{(N_2)}(p)$$ 
in the notation of Corollary~\ref{cor:dualexponentialmapforH++}.
\begin{thm}
\label{thm:mainEXPrankinselberg}
There exists an $\mathcal{R}_\cyc$-linear map 
$$\mathrm{EXP}^\ast \,:\, 
H^1(\QQ_p,F^{-+}\TT_\cyc)
\lra 
\mathbb{D}(F^{-+}\TT_\cyc)  
$$ 
which is characterized by the following interpolation property: For every $\kappa \in \mathcal{S}^{(1),+}_\cyc$  (so that $j> w_2$) the following diagram commutes:
$$
\xymatrix{
H^1(\QQ_p,F^{-+}\TT_\cyc)
\ar[d]_\kappa \ar[rrr]^{\mathrm{EXP}^\ast} &&& 
\mathbb{D}(F^{-+}\TT_\cyc )
\ar[d]^\kappa \\ 
H^1(\QQ_p, F^{-+}V_\kappa)
\ar[rrr]_{\hspace*{10pt}e_p^+\,\times\, 
\exp^\ast} 
&&& 
D_{\textup{dR}}(F^{-+}V_\kappa)
}$$
Here $e_p^+:=(-1)^{j-w_2}(j-w_2)!\,e_p$ and $e_p=e_p((F^{-+}V_\kappa)^{\kappa (\RRc)}(1))$ is the $p$-adic multiplier given by
$$
e_p=\left( 1 - \dfrac{p^{j-w_2-1}}{a_p(\Bf_1(\kappa_1) )a_p(\Bf_2(\kappa_2) )^{-1}\Psi_2^{(N_2)}(p)} \right) \left( 
1 - \dfrac{a_p(\Bf_1(\kappa_1) )a_p(\Bf_2(\kappa_2) )^{-1}\Psi_2^{(N_2)}(p)}{p^{j-w_2}}\right)^{-1}$$
in case $\omega^{-w_2-1+j}\Psi_2^{(p)}\phi_{\kappa_2}\phi$ is the trivial character (which is equivalent to the requirement that $F^{-+}V_\kappa$ be crystalline), and

$$ e_{p}=\left( \dfrac{p^{j-w_2-1}}{a_p(\Bf_1(\kappa_1) )a_p(\Bf_2(\kappa_2) )^{-1}\Psi_2^{(N_2)}(p)}\right)^n$$
when  $F^{-+} V_\kappa\,\vert_{I_p} \cong E_\kappa(\omega_2+1-j)(\eta)$ with $\textup{ord}_p(\textup{cond}(\eta))=n\geq1$.

Also, for every $\kappa \in \mathcal{S}^{(1),-}_\cyc$ (so that $j \leq w_2$) we also have the following commutative diagram:
$$
\xymatrix{
H^1(\QQ_p,F^{-+}\TTc )
\ar[d]_\kappa \ar[rrr]^(.48){\mathrm{EXP}^\ast}&&& 
\mathbb{D}(F^{-+}\TTc )
\ar[d]^\kappa\\
H^1(\QQ_p,F^{-+}V_\kappa )
\ar[rrr]_(.45){\hspace*{15pt}e_p^-\,\times\, 
\log
} 
&&& 
D_{\textup{dR}}(\mathrm{Gr}_p^{m}V_\kappa )
}$$
where $e_p^-:=\dfrac{e_p}{(w_2-j)!}$\,.
\end{thm}
Recall that we have an identification
\begin{align}
F^{-+}\TT_\cyc=\textup{Gr}^1_p\TT_\cyc=F_p^-\TT_{\Bf_1}\widehat{\otimes}_{\ZZ_p}F_p^+\TT_{\Bf_2}\widehat{\otimes}_{\ZZ_p}(\LL_\cyc^{\sharp})^\iota
\end{align}
and hence 
\begin{equation}\label{equation:decomposition_of_D}
\mathbb{D}(F^{-+}\TT_\cyc)  =
\mathbb{D}\left(F_p^-\TT_{\Bf_1}\right) \widehat{\otimes}_{\mathbb{Z}_p} \mathbb{D}\left(F_p^+\TT_{\Bf_2}\right)
\widehat{\otimes}_{\mathbb{Z}_p} \Lambda_{\mathrm{cyc}}.
\end{equation}
\begin{rem}
\label{rem:HidamailiesherevsHidafamiliesinKLZ2}
We recall that what we call a Hida family in this article corresponds to a new, cuspidal branch of a Hida family $\mathbf{F}_i$ of tame level $N_i$ in the sense of \cite[Remark 7.2.4]{KLZ2}. 
The ring denoted by $\Lambda_{\mathbf{a}}$ in op. cit. agrees with our $\mathbb{I}_{\bf{f}_1}$, with $\Bf$ in loc. cit. chosen as $\mathbf{F}_1$ and the minimal ideal $\mathbf{a}$ in the notation of \cite{KLZ2} corresponds to our branch $\Bf_1$. 
Finally, the fractional ideal denoted by $I_{\mathbf{a}}$ in \cite[Notation 7.7.1]{KLZ2} agrees with the fractional ideal $\dfrac{1}{H}\mathbb{I}_{\Bf_1}$ above. 
\end{rem}
{
Let $\overline{S}^{\rm ord}(N)$ be the space of ordinary $\LL$-adic cusp forms of tame level $N$. 
Recall that we have a canonical injection 
\begin{equation}\label{equation:ohta}
\mathbb{D}\left(F_p^- \TT^\ast_{\Bf_2} (1)\right) \hookrightarrow \overline{S}^{\rm ord}(N)
\end{equation} 
induced by Ohta's $\LL$-adic Eichler-Shimura isomorphism established in \cite[Theorem 2.1.11]{Ohta2000}. }
\begin{define}
\label{def_specialeigenvectors}
{Let us choose an element 
$$\omega_{\Bf_2} \in \mathbb{D}\left(F_p^- \TT^\ast_{\Bf_2} (1)\right) ={\rm Hom}_{\mathbb{I}_{\Bf_2}}\left(\mathbb{D}\left(F_p^+ \TT_{\Bf_2}\right),\mathbb{I}_{\Bf_2}\right)$$ which interpolates holomorphic differential forms corresponding to arithmetic specializations of the dual Hida family $\overline{\Bf}_2$ of ${\Bf}_2$ via \eqref{equation:ohta}. We denote by $\omega_{\Bf_2}^{(\mathbb{{I}})}$ the extension (by linearity) of $\omega_{\Bf_2}$ to an element of 
${\rm Hom}_{\mathbb{I}}\left(\mathbb{D}\left(F_p^+ \TT_{\Bf_2}\right),\mathbb{I}\right)$. Similarly we define $\omega_{\Bf_1} \in \mathbb{D}\left( F_p^- \TT_{\Bf_1} (1)\right)$ which interpolates holomorphic differential forms corresponding to arithmetic specializations of the Hida family ${\Bf}_1$. 
}
{Mimicking Hida's construction in~\cite[(7.5)]{Hida88}, we let $\eta_{\Bf_1}\in {\rm Hom}_{\mathbb{I}}\left(\mathbb{D}\left(F_p^-\TT_{\Bf_1}\right) ,\dfrac{1}{H}\mathbb{I}\right)$ denote the compositum of the following arrows:
$$\mathbb{D}\left(F_p^-\TT_{\Bf_1}\right)\lra \mathbb{D}\left(F_p^-\TT_{\Bf_1}\right)\otimes_{\mathbb{I}_{\Bf_2}}\mathbb{I} 
\hookrightarrow \overline{S}^{\rm ord}(N)\otimes_{\LL}\mathbb{I} \stackrel{(1_{\Bf_1},\,\,)}{\lra} \frac{1}{H}\mathbb{I}\, 
$$
where $\mathbf{1}_{\Bf_1}\in  \mathbf{h}\otimes_{\LL}\mathscr{K}$ is the idempotent that Hida in op. cit. denotes by $\mathbf{1}_{\mathscr K}$, which corresponds to the choice $\Bf_1$ as a branch and the pairing $(\,,\,)$ is the one given as in \cite[Theorem 7.4]{Hida88}. 
Note that the compositum of these arrows a priori land in $\mathscr{K}$. However, Hida explains that by the very definition of the congruence ideal in \cite[(4.3)]{Hida88}, it follows that $H\cdot \mathbf{1}_{\Bf_1} \in \mathbf{h}\otimes_{\LL}\mathbb{I}$.}
\end{define}
 \begin{define}
 \label{def:geometricpadicLfunction}
Using the identification \eqref{equation:decomposition_of_D}, we define
We define
\begin{equation}\label{equation:definitionL_pBF}
L_p^{\textup{BF}}(\Bf_1,\Bf_2,\mathbf{j}):= 
\left\langle  \EXP^\ast \circ\,\res_p\left(\BF_1^{\Bf_1,\Bf_2} \right)
, \eta_{\Bf_1}\widehat{\otimes} \omega_{\Bf_2}^{(\mathbb{I})}\right\rangle 
 \in \frac{1}{H}\mathbb{I}\widehat{\otimes}\mathbb{I}\widehat{\otimes}\LL_\cyc
\end{equation}
where $\mathbf{j}$ is the cyclotomic variable. 
\end{define}
The following is a very slight variant of \cite[Theorem 10.2.2]{KLZ2} (and its proof follows the argument in op. cit.).
\begin{thm}
\label{thm:excplticitreciprocityPart2} We have the equality 
{\be\label{eqn:BFpadicvsHidapadic} 
 (L_p^{\textup{BF}}(\Bf_1,\Bf_2,\mathbf{j}))
=(L_p^{\textup{Hida}}(\Bf_1,\Bf_2,\Bj)).
\ee
of fractional ideals of $\mathbb{I}\widehat{\otimes}\mathbb{I}\widehat{\otimes}\LL_\cyc$.}
\end{thm}
\begin{proof}
{Let us take any arithmetic point $\kappa =(\kappa_1 ,\kappa_2, j)$ in ${\rm Sp}(\mathbb{I}\widehat{\otimes}\mathbb{I}\widehat{\otimes}\LL_\cyc)$ such that 
\begin{itemize}
\item[(i)] the weights $w_1$ and $w_2$ of $\kappa_1 $ and $\kappa_2$ satisfy the condition 
$1\leq j \leq w_2+1 <w_1 +1$, 
\item[(ii)] $\Bf_1(\kappa_1)$ and $\Bf(\kappa_2)$ are $p$-old,
\item[(iii)] the character $\omega^{-w_2-1+j}\Psi_2^{(p)}\phi_{\kappa_2}\phi$ is trivial.
\end{itemize} }
 
By the definition of $L_p^{\textup{BF}}(\Bf_1,\Bf_2,\mathbf{j})$  given in \eqref{equation:definitionL_pBF}, we have 
\begin{equation}\label{equation:final1}
\kappa (L_p^{\textup{BF}}(\Bf_1,\Bf_2,\mathbf{j}))
= \left\langle \kappa (\EXP^\ast \circ\,\res_p\left(\BF_1^{\Bf_1,\Bf_2}\right)) , 
\eta_{\Bf_1 (\kappa_1 )}\otimes \omega_{\Bf_2 (\kappa_2)}
\right\rangle .
\end{equation}
By the interpolation property of the big exponential map $\EXP^\ast$ 
obtained in Theorem \ref{thm:mainEXPrankinselberg}, we have 
\begin{equation}\label{equation:final2}
\kappa (\EXP^\ast \circ\,\res_p\left(\BF_1^{\Bf_1,\Bf_2}\right))
= e^-_p \cdot \log (\kappa (\res_p\left(\BF_1^{\Bf_1,\Bf_2}\right))). 
\end{equation}
Furthermore, the class 
$\kappa (\res_p\left(\BF_1^{\Bf_1,\Bf_2}\right))$ is related to 
the special value of $L_p^{\textup{Hida}}$ via the identity {
\begin{multline}\label{equation:final3}
\left\langle \log \left(\kappa (\res_p\left(\BF_1^{\Bf_1,\Bf_2}\right))\right), 
\kappa_1(\eta_{\Bf_1})\otimes \kappa_2(\omega_{\Bf_2}^{(\mathbb{I})})\right\rangle
= \lambda_N(\Bf_1(\kappa_1)^\circ)^{-1} (-1)^{j}
 L_p^{\textup{Hida}}(\Bf_1 (\kappa_1),\Bf_2 (\kappa_2) ,j)
\end{multline}
(where $\Bf_1(\kappa_1)^\circ$ is the newform whose $p$-stabilization is $\Bf_1(\kappa_1)$ and $\lambda_N(\Bf_1(\kappa_1)^\circ)$ is the Atkin-Lehner pseudo-eigenvalue) thanks to \cite[Theorem 6.5.9]{KLZ1} combined with  \cite[Proposition 10.1.1]{KLZ2}. }
{Note that the arithmetic points $\kappa = (\kappa_1 ,\kappa_2 ,j)$ satisfying the conditions (i), (ii) and (iii) are Zariski dense in ${\rm Sp}(\mathbb{I}\widehat{\otimes}\mathbb{I}\widehat{\otimes}\LL_\cyc)$. }
Combining the equations \eqref{equation:final1}, \eqref{equation:final2} and  \eqref{equation:final3}, it follows that 
$( L_p^{\textup{BF}}(\Bf_1,\Bf_2,\mathbf{j}))
=( L_p^{\textup{Hida}}(\Bf_1,\Bf_2,\Bj))$, as we have claimed.
\end{proof}

\begin{cor}\label{cor:nontriviality_BF}
The class $\res_p(\BF^{\Bf_1,\Bf_2}_{1})$ is non-trivial.
\end{cor} 
\begin{proof}
Theorem~\ref{thm:excplticitreciprocityPart2} reduces the non-vanishing of the class $\res_p(\BF^{\Bf_1,\Bf_2}_{1})$ to the non-triviality of Hida's $p$-adic Rankin-Selberg $L$-function $L_p^{\textup{Hida}}(\Bf_1,\Bf_2,\mathbf{j})$. 

To see the non-triviality of $L_p^{\textup{Hida}}(\Bf_1,\Bf_2,\mathbf{j})$, 
we consider the interpolation property \eqref{equation:interpolationHida} of $L_p^{\textup{Hida}}(\Bf_1,\Bf_2,\mathbf{j})$. 
We recall that the Euler product of the $L$-function $L(\kappa_1 (\Bf_1) ,\kappa_2 (\Bf_1),s)$ is absolutely convergent (and therefore non-zero) for $\mathrm{Re} (s) > \frac{w_1 +w_2 +2}{2}+1$. 
It is also easy to check the non-vanishing of the Euler-like factor 
\begin{multline*}
\left(1-\frac{p^{j-1}}{\alpha (\Bf_1 (\kappa_1))\alpha (\Bf_2(\kappa_2))} \right) \left(1-\frac{p^{j-1}}{\alpha (\Bf_1 (\kappa_1))\beta (\Bf_2(\kappa_2))}  \right)\left(1-\frac{\beta (\Bf_1 (\kappa_1))\alpha (\Bf_2(\kappa_2))}{p^{j}} \right) \\
 \times \left(1-\frac{\beta (\Bf_1 (\kappa_1))\alpha (\Bf_2(\kappa_2))}{p^{j}}  \right)
 {\left(1-\frac{\beta (\Bf_1(\kappa_1))}{p\alpha (\Bf_1(\kappa_1))}   \right)^{-1} \left(1-\frac{\beta (\Bf_1(\kappa_1))}{\alpha (\Bf_1(\kappa_1))} \right)^{-1}}.
\end{multline*}
In fact, it is known that the archimedean absolute values of the eigenvalues $\alpha (\Bf_1 (\kappa_1))$ and $\beta (\Bf_1 (\kappa_1))$ belong to the range $[p^{\frac{w_1 }{2}}, p^{\frac{w_1 }{2}+1}]$. A similar statement holds true for $\alpha (\Bf_2 (\kappa_2))$ and $\beta (\Bf_2 (\kappa_2))$. 
Based on these facts, it is now easy to check that $L_p^{\textup{Hida}}(\Bf_1,\Bf_2,\mathbf{j})$ has a specialization 
$(\kappa_1,\kappa_2 , w_1)$ with $w_1>w_2+2$ where the value of 
the function $L_p^{\textup{Hida}}(\Bf_1,\Bf_2,\mathbf{j})$ is non zero.  
Thence, the function $L_p^{\textup{Hida}}(\Bf_1,\Bf_2,\mathbf{j})$ itself 
is non zero. 
\par 
Consequentially, Theorem~\ref{thm:excplticitreciprocityPart2} implies that the class $\res_p(\BF^{\Bf_1,\Bf_2}_{1})$ is non-trivial.
\end{proof}

\begin{cor}
\label{cor:mainthmRankinSelberg} Let $f_i \in S_{k_i}(\Gamma_1(N_i  ))$ ($i=1,2$) be non-CM newforms of respective weights $k_1, k_2$, levels $N_1,N_2$ which are not twisted-conjugate to each other and suppose  $p>B$ is a good ordinary prime for both forms.  Let $\Bf_i$ be the unique Hida family that admits the $p$-stabilization of $f_i$ as a weight-$k_i$ arithmetic specialization ($i=1,2$). Suppose that the conditions in Lemma~\ref{lem:allhypoHholdstrue} that ensure the validity of {\rm (H.0)}, {\rm (H.2)}, {\rm (H.0$^-$)}, 
{\rm (H.2$^+$)} and {\rm (H.2$^{++}$)} hold. Assume also the hypotheses {\rm (H.c)}, \textup{(BI.1), (BI.2)} and  {\rm (F.Dist)}. Then, 
we have the following inclusion of integral ideals of $\mathbb{I}\widehat\otimes\mathbb{I}\widehat\otimes\LL_\cyc $: 
{$$\textup{char}_{\mathbb{I}\widehat\otimes\mathbb{I}\widehat\otimes\LL_\cyc }\left( 
H^1_{\FF_\Gr^*}(\QQ,\TT_\cyc^\vee (1))^\vee \otimes_{\mathcal{R}_\cyc } 
\mathbb{I}\widehat\otimes\mathbb{I}\widehat\otimes\LL_\cyc \right) 
\,\, \supset \,\, (H \cdot  L_p^{\textup{Hida}}(\Bf_1,\Bf_2,\Bj) ) \, .
$$}
\end{cor}
\begin{proof}
Under our running hypotheses we have the following chain of isomorphisms
$$
H^1_{+/\textup{f}}(\QQ_p,\TT_\cyc)\stackrel{\sim}{\lra} H^1(\QQ_p,F^{-+}\TT_\cyc)\stackrel{\sim}{\lra} \mathbb{D}(F^{-+}\TT_\cyc).
$$
Since $\omega_{\Bf_2}$ is an isomorphism (see \cite[Proposition 10.1.1(1)]{KLZ2}), it follows from Theorem~\ref{thm:excplticitreciprocityPart2} that
{$$
\textup{char}_{\mathbb{I}\widehat\otimes\mathbb{I}\widehat\otimes\LL_\cyc }
\left(\dfrac{H^1_{+/\textup{f}}(\QQ_p,\TT_\cyc)}{\mathcal{R}_\cyc \res_{+/\textup{f}}\left(\BF^{\Bf_1,\Bf_2}_{1}\right)} 
\otimes_{\mathcal{R}_\cyc } 
\mathbb{I}\widehat\otimes\mathbb{I}\widehat\otimes\LL_\cyc
\right)\supset
(H \cdot  L_p^{\textup{BF}}(\Bf_1,\Bf_2,\Bj)) \,.$$}
Corollary now follows from Theorem~\ref{thm:mainconjwithoutpadicLfunc} (3) 
and Corollary~\ref{cor:nontriviality_BF}. 
\end{proof}


{\bibliographystyle{halpha}
\bibliography{references}
}

\end{document}